\documentclass[12pt,amstex]{amsart}
\usepackage{cite}
\usepackage{amsmath,amsthm,amssymb,amsfonts,amscd,verbatim,color}
\usepackage{mathrsfs}
\usepackage{bbm}
\usepackage{indentfirst}
\usepackage{enumerate}
\usepackage{xcolor}
\usepackage[linktocpage, colorlinks, citecolor=red, anchorcolor=black, linkcolor=red]{hyperref} 
\usepackage[top=35mm, bottom=35mm, left=30mm, right=30mm]{geometry}
\allowdisplaybreaks[4]

\theoremstyle{plain}

\theoremstyle{definition}

\newtheorem{definition}{Definition}[section]
\newtheorem{lemma}{Lemma}[section]
\newtheorem{theorem}{Theorem}[section]
\newtheorem{proposition}{Proposition}[section]

\newtheorem{remark}{Remark}[section]

\numberwithin{equation}{section}
\usepackage{cleveref}

\newcommand{\bs}{\boldsymbol}
\newcommand{\mb}{\mathbb}

\newcommand{\mc}{\mathcal}
\newcommand{\ms}{\mathscr}
\newcommand{\mr}{\mathrm}

\def \pa{\partial}

\begin{document}
\title[]{Exact Global Control of Small Divisors in Rational Normal Form}

\let\thefootnote\relax\footnotetext{Supported by NNSFC (Grant Nos. 11971299, 12090010, 12090013)}

\author{Jianjun Liu \quad Duohui Xiang }
\address[\;Jianjun Liu\;] {School of Mathematics\\ Sichuan University\\ Chengdu 610065, China}
\email{jianjun.liu@scu.edu.cn}

\address[Duohui Xiang ] {School of Mathematics\\ Sichuan University\\ Chengdu 610065, China}
\email{duohui.xiang@outlook.com}

\thanks{}

\begin{abstract}
Rational normal form is a powerful tool to deal with Hamiltonian partial differential equations without external parameters. In this paper, we build rational normal form with exact global control of small divisors. As an application to nonlinear Schr\"{o}dinger equations in Gevrey spaces, we prove sub-exponentially long time stability results for generic small initial data.
\end{abstract}

\maketitle

\tableofcontents
\section{Introduction}
Birkhoff normal form for long time stability of solutions of Hamiltonian partial differential equations has been widely investigated by many authors.
For instance, in \cite{BG06}, Bambusi and Gr\'{e}bert proved an abstract Birkhoff normal form theorem adapted to a wide class of Hamiltonian partial differential equations, such as nonlinear wave equations and nonlinear Schr\"{o}dinger equations with Dirichlet or periodic boundary conditions. To guarantee the non-resonant conditions, these equations are of external parameters, such as the mass in wave equations and the potential in Schr\"{o}dinger equations. Consequently, for these equations in Sobolev spaces $H^s$ with large $s$, they proved polynomially long time stability estimates, i.e., the stability time is of order $\varepsilon^{-r}$ for any given positive integer $r$.
There are many other polynomially long time stability results for equations in Sobolev spaces with external parameters, seeing \cite{B96,B03,BDGS07,B08,GIP09,Z10,D12,YZ14,D15,CLY16,BD18,BFG20,Z20} for example.
Moreover, for equations in analytic or Gevrey spaces with external parameters, there are sub-exponentially long time stability results, i.e., the stability time is of order $\varepsilon^{-|\ln\varepsilon|^{\beta}}$ for some $0<\beta<1$. See \cite{FG13,CLSY18,BMP20,CMW20,CCMW22,CMWZ22} for example.

On the other hand, for an equation without external parameters, its linear frequencies usually don't meet the high-order non-resonant conditions. Then a possible way is to extract parameters from nonlinear integrable terms by amplitude-frequency modulation, in which the amplitudes of initial data are used as parameters. These are the so-called inner parameters.
Compared with external parameters, the size of inner parameters is much smaller, which leads to more sensitive non-resonant conditions.
In \cite{KP96}, for one dimensional nonlinear Schr\"{o}dinger equation, Kuksin and P\"{o}schel introduced four-order normal form terms as part of the unperturbed Hamiltonian to modulate the resonant linear frequencies such that the non-resonant conditions are ensured in KAM iteration.
In \cite{B99b}, Bambusi used this nonlinear modulation to prove exponentially long time stability for a particular set of initial data in $H^1$.
In \cite{B00}, Bourgain also used this nonlinear modulation to prove arbitrarily polynomially long time stability for most initial data in $H^s$ with large $s$. Since the modulated frequencies from four-order normal form terms are linear about parameters, the smallest index of parameters in non-resonant conditions may be not small. This leads to an intricate iterative procedure in \cite{B00}.
Recently in \cite{BFG20b,BG21}, Bernier, Faou and Gr\'{e}bert built rational normal form to research more Hamiltonian partial differential equations without external parameters, including one dimensional nonlinear Schr\"{o}dinger and Schr\"{o}dinger-Poisson equations, generalized KdV and Benjamin-Ono equations. Among them, the first two equations are of bounded nonlinear vector field, and the last two equations are of unbounded nonlinear vector field. An advantage of rational normal form is that it provides a more uniform stability result and enables an accurate description of the typical dynamics. Moreover, compared with \cite{KP96,B99b,B00} where only four-order normal form terms are used to modulate frequencies, in \cite{BFG20b,BG21} six-order normal form terms are also used, such that frequencies are more twisted about parameters.

To our knowledge, in the area of Birkhoff normal form for Hamiltonian partial differential equations with inner parameters, \cite{B00,BFG20b,BG21} mentioned above are the only results with general initial data and arbitrary number of iterative steps. In these three papers, the phase spaces are Sobolev spaces, and the stability time is of arbitrarily polynomial order. Now a natural question is to consider smoother phase spaces and longer stability time.
In smoother spaces, such as Gevrey or analytic spaces, the size of inner parameters is much smaller than it in Sobolev spaces. Thus the non-resonant conditions are more sensitive to Fourier modes. This leads to an essential difficulty of using the rational normal form in \cite{BFG20b,BG21}. In the following, we will illustrate the difficulty in detail.

The rational monomial is roughly of the form
\begin{equation}
\label{form6-30-1}
\frac{u_{a_1}\bar{u}_{a_2}\cdots u_{a_{2p-1}}\bar{u}_{a_{2p}}}{\prod_{i=1}^{q}\Omega_{\bs{b}_i}(I)}
\end{equation}
with $0\leq q<p$,
where $\bs{b}_{i}$ is an integer vector and $\Omega_{\bs{b}_i}(I)$ is a small divisor for every $i=1,\cdots,q$.
%
Notice that the small divisor $\Omega_{\bs{b}_i}(I)$ can be controlled by $\mu_{\min}(\bs{b}_i)$, which denotes the smallest index of $\bs{b}_i$. Specifically in the Sobolev space 
\begin{equation}
\label{form7-7-1'}
H^{s}:=\{u=\{u_a\}_{a\in\mb{Z}}\mid \|u\|_{s}^{2}:=|u_0|^{2}+\sum_{a\in\mb{Z}\backslash\{0\}}|a|^{2s}|u_a|^{2}<+\infty\},
\end{equation}
it very roughly holds that
\begin{equation}
\label{form7-5-1}
|\Omega_{\bs{b}_i}(I)|>\mu_{\min}(\bs{b}_i)^{-2s}\quad\text{and}\quad|u_a|<|a|^{-s}.
\end{equation}
%
In order to control the denominator of \eqref{form6-30-1}, the condition `\emph{distribution of the derivatives}' was built in (64) of \cite{BFG20b}: there exist an injective map $\iota:\{1,\cdots,2q\}\mapsto\{3,\cdots,2p\}$ and a positive constant $c$ such that
\begin{equation}
\label{form110}
\mu_{\min}(\bs{b}_{i})\leq c\mu_{\iota_{2i-1}}(\bs{a}),\quad
\mu_{\min}(\bs{b}_{i})\leq c\mu_{\iota_{2i}}(\bs{a}),
\end{equation}
where $\mu_k(\bs{a})$ denotes the $k$'s largest norm of $\bs{a}:=(a_1,\cdots,a_{2p})$.
Consequencely, one has
\begin{equation}
\label{form7-5-2}
\prod_{i=1}^{q}\big(\mu_{\min}(\bs{b}_{i})\big)^{2}\leq c^{2q}\prod_{k=3}^{2p}\mu_{k}(\bs{a}).
\end{equation}
Hence, the rational monomial \eqref{form6-30-1} can be well estimated by \eqref{form7-5-1} and \eqref{form7-5-2}.
Moreover, \eqref{form110} implies the fact:
\begin{equation}
\label{form111}
\mu_{\min}(\bs{b}_{i})\leq c\mu_{2}(\bs{a}), \;\text{for} \; i=1,\cdots,q,
\end{equation}
which is the condition `\emph{global control of the structure}' in (65) of \cite{BFG20b}. This condition is used to estimate the derivative of the denominator in \eqref{form6-30-1}.
To make sure that the condition \eqref{form110} can be kept in the Possion bracket of any two rational monomials, the constant $c$ must depend on $p,q$ with $c>1$. In fact, even for \eqref{form111}, it is  impossible to fix $c=1$.
Now, for instance, consider the rational monomial \eqref{form6-30-1} in the Gevrey space 
\begin{equation}
\label{form7-7-1''}
\mc{G}_{\rho,\theta}:=\{u=\{u_a\}_{a\in\mb{Z}}\mid \|u\|_{\rho,\theta}^{2}:=\sum_{a\in\mb{Z}}e^{2\rho|a|^{\theta}}|u_a|^{2}<+\infty\},
\end{equation}
where $\rho>0$ and $0<\theta<1$. Instead of \eqref{form7-5-1},  one very roughly has
\begin{equation}
\label{form7-5-3}
|\Omega_{\bs{b}_i}(I)|>e^{-2\rho|\mu_{\min}(\bs{b}_i)|^{\theta}}\quad\text{and}\quad|u_a|<e^{-\rho|a|^{\theta}}.
\end{equation}
%
%
By \eqref{form110} and \eqref{form7-5-3}, one has
\begin{equation}
\label{form7-8-1}
\frac{1}{|\prod_{i=1}^{q}\Omega_{\bs{b}_i}(I)|}<e^{2\rho\sum_{i=1}^{q}|\mu_{\min}(\bs{b}_i)|^{\theta}}\leq e^{\rho\sum_{k=3}^{2p}|c\mu_{k}(\bs{a})|^{\theta}}=\prod_{k=3}^{2p}e^{\rho c^{\theta}|\mu_{k}(\bs{a})|^{\theta}},
\end{equation}
which can not be controlled by the numerator of \eqref{form6-30-1} due to $c^{\theta}>1$.
%
Therefore, for rational normal form, the conditions \eqref{form110} and \eqref{form111} built in \cite{BFG20b} can not be adopted to Gevrey spaces.

To solve this problem, we develop the tool of rational normal form in general Hilbert spaces
$$\mr{h}_{\mr{w}}:=\{u=\{u_a\}_{a\in\mb{Z}}\mid |u|_{\mr{w}}^{2}:=\sum_{a\in\mb{Z}}\mr{w}_a^{2}|u_a|^{2}<+\infty\},$$
where the weight $\mr{w}=\{\mr{w}_a\}_{a\in\mb{Z}}$ is  a real sequence satisfying $1\leq\mr{w}_a\leq\mr{w}_{a'}$ for $|a|\leq|a'|$. Instead of the conditions \eqref{form110} and \eqref{form111}, we build the following two conditions to globally control small divisors:
\begin{equation}
\label{form42'}
\prod_{i=1}^{q}\mr{w}_{\mu_{\min}(\bs{b}_{i})}^{2}\leq\frac{\prod_{k=1}^{2p}\mr{w}_{a_{k}}}{\mr{w}_{\mu_1(\bs{a})}^{2}},
\end{equation}
\begin{equation}
\label{form43'}
\big(\prod_{i=1}^{q}\mr{w}_{\mu_{\min}(\bs{b}_{i})}^{2}\big)\max_{i=1,\cdots,q}\{\mr{w}_{\mu_{\min}(\bs{b}_{i})}^{2}\}\leq\prod_{k=1}^{2p}\mr{w}_{a_{k}}.
\end{equation}
Importantly, we can prove that these two conditions are kept together in the iterative process, namely they are preserved in the Poisson bracket, seeing \Cref{le42}. We emphasize that neither \eqref{form42'} nor \eqref{form43'} could be separately kept under Poisson bracket.
Just as important, the conditions \eqref{form42'} and \eqref{form43'} are exactly adequate to bound the rational monomial \eqref{form6-30-1} and its derivative, seeing \eqref{form421} and \eqref{form418} in the proof of \Cref{le41}. 

Specially, in the Sobolev space $H^{s}$, the conditions \eqref{form42'} and \eqref{form43'} can be written as
\begin{equation}
\label{form112'}
\prod_{i=1}^{q}\big(\mu_{\min}(\bs{b}_{i})\big)^{2}\leq\frac{\prod_{k=1}^{2p}\mu_{k}(\bs{a})}{\big(\mu_1(\bs{a})\big)^{2}},
\end{equation}
\begin{equation}
\label{form113'}
\prod_{i=1}^{q}\big(\mu_{\min}(\bs{b}_{i})\big)^{2}\max_{i=1,\cdots,q}\{\big(\mu_{\min}(\bs{b}_{i})\big)^{2}\}\leq\prod_{k=1}^{2p}\mu_{k}(\bs{a}).
\end{equation}
Notice that in the conditions \eqref{form110} and \eqref{form111}, up to a multiple $c$, every $\mu_{\min}(\bs{b}_{i})$ is controlled. 
Instead, in the conditions \eqref{form112'} and \eqref{form113'}, we exactly control the combines of $\mu_{\min}(\bs{b}_{i})$, $i=1,\cdots,2q$.
Actually, \eqref{form112'} imples
$$\prod_{i=1}^{q}\big(\mu_{\min}(\bs{b}_{i})\big)^{2}\leq\prod_{k=3}^{2p}\mu_{k}(\bs{a}),$$
which is \eqref{form7-5-2} with the multiple $c^{2q}$ being deleted. This means better estimates of coefficients in rational normal form. Therefore, the conditions \eqref{form42'} and \eqref{form43'} may be useful even in Sobolev spaces, but we will not discuss further in the present paper.

In the Gevrey space $\mc{G}_{\rho,\theta}$, the conditions \eqref{form42'} and \eqref{form43'} can be written as
\begin{equation}
\label{form112}
2\sum_{i=1}^{q}\big(\mu_{\min}(\bs{b}_{i})\big)^{\theta}\leq\sum_{k=1}^{2p}\big(\mu_{k}(\bs{a})\big)^{\theta}-2\big(\mu_1(\bs{a})\big)^{\theta},
\end{equation}
\begin{equation}
\label{form113}
2\sum_{i=1}^{q}\big(\mu_{\min}(\bs{b}_{i})\big)^{\theta}+2\max_{i=1,\cdots,q}\{\big(\mu_{\min}(\bs{b}_{i})\big)^{\theta}\}\leq\sum_{k=1}^{2p}\big(\mu_{k}(\bs{a})\big)^{\theta}.
\end{equation}
With the help of \eqref{form112}, the inequality \eqref{form7-8-1} becomes into
$$\frac{1}{|\prod_{i=1}^{q}\Omega_{\bs{b}_i}(I)|}<e^{2\rho\sum_{i=1}^{q}|\mu_{\min}(\bs{b}_i)|^{\theta}}\leq e^{\rho\sum_{k=1}^{2p}|\mu_{k}(\bs{a})|^{\theta}-2\rho|\mu_1(\bs{a})|^{\theta}}\leq\prod_{k=3}^{2p}e^{\rho |\mu_{k}(\bs{a})|^{\theta}},$$
which is controlled by the numerator of \eqref{form6-30-1}. We also mention that \eqref{form113} is used to control the derivative of $\frac{1}{\prod_{i=1}^{q}\Omega_{\bs{b}_i}(I)}$.
%
\\\indent
%
As an application, we study the long time stability in Gevrey spaces $\mc{G}_{\rho,\theta}$ for the nonlinear Schr\"{o}dinger equation
\begin{equation}
\label{form11}
{\rm i}u_{t}=-u_{xx}+\varphi(|u|^{2})u,\quad x\in\mb{T}:=\mb{R}/2\pi\mb{Z},
\end{equation}
where $\varphi$ is a real analytic function on a neighborhood of the origin satisfying $\varphi(0)=0$ and $\varphi'(0)\neq0$.
We prove sub-exponentially long time stability results for generic small initial data, seeing \Cref{th11}. The stability time is of order $\varepsilon^{-|\ln\varepsilon|^{\beta}}$ for any $0<\beta<1$.
We emphasize that the exponent $\beta$ in the stability time could be arbitrarily close to $1$, seeing \Cref{re12} for a comparison with the previous results.
In order to get longer stability time than polynomial order, the number of iterations $r$ must depend on $\varepsilon$. Thus we need more concrete coefficient estimates, especially the dependence on $r$. This is notably different from \cite{BFG20b,BG21} in which  $r$ is an arbitrarily given constant.
%
%
Moreover, the measure is also different. In \cite{BFG20b,BG21}, roughly speaking, the actions $\{I_{a}\}_{a\in\mb{Z}}$ are assumed to be independent and uniformly distributed in an infinite cube 
$\prod_{a\in\mb{Z}\backslash\{0\}}(0,|a|^{-2s-\nu})$ with $\nu>1$. In the present paper, instead of the infinite cube, we directly use the unit ball $B_{\rho,\theta}(1)$ in $\mc{G}_{\rho,\theta}$, that is $\sum_{a\in\mb{Z}}e^{2\rho|a|^{\theta}}I_a<1$.
Correspondingly, we introduce the Gaussian measure instead of the product measure. See \Cref{re13} for a more detailed discussion of the measure.

In addition, we think that the tool of rational normal form with \eqref{form42'} and \eqref{form43'} could be used in a wider range, in the sense of more equations and more general phase spaces.
\\\indent In the present paper, for convenience, we keep fidelity with the notation and terminology from \cite{BFG20b,BG21}. Now we lay out an outline:
\\\indent \Cref{sec2} contains notations and main results. In subsection \ref{subsec21}, we introduce the conditions \eqref{form42} and \eqref{form43} to exactly globally control the small divisors of rational Hamiltonian functions  in general Hilbert spaces. Then we give the estimates of the Poisson bracket and the vector field of rational Hamiltonian functions, seeing \Cref{le42} and  \Cref{le41}.  In subsection \ref{subsec22}, for the nonlinear Schr\"{o}dinger equation without external parameters in Gevrey spaces $\mc{G}_{\rho,\theta}$, we give the sub-exponentially long time stability results, seeing \Cref{th11}.
\\\indent In \Cref{sec3}, we prove \Cref{le42} and  \Cref{le41} in detail. 
\\\indent In \Cref{sec4}--\ref{sec7},  \Cref{th11} is proved.
In \Cref{sec4}, we firstly prove the resonant normal form theorem for nonlinear Schr\"{o}dinger equation, seeing \Cref{th21}. 
Then we define the non-resonant set $\mc{U}^{N}_{\gamma}$ and show that the small divisor conditions are well preserved under perturbations, seeing \Cref{le31}.
In \Cref{sec5}, we firstly eliminate the truncated six-order resonant term $K_6$ by the integrable term $Z_{4}$, 
and thus get \Cref{th51} by one rational normal form step. Then we eliminate the higher order resonant rational terms by $Z_{4}+Z_{6}$, 
and thus get \Cref{th52} by many rational normal form steps.
In \Cref{sec6}, we estimate the measure of the non-resonant set $\mc{U}^{N}_{\gamma}$, seeing \Cref{le32}.
In \Cref{sec7}, combining the rational normal form theorem with the measure estimate, we complete the proof of \Cref{th11}.

\section{Notations and main results}
\label{sec2}

Given a function $u\in L^{2}(\mb{T})$, define its Fourier coefficients $u_a:=\frac{1}{2\pi}\int_{\mb{T}}u(x)e^{-{\rm i}ax}dx$, $a\in\mb{Z}$. In this paper, we identify the function $u$ with its sequence of Fourier coefficients $\{u_a\}_{a\in\mb{Z}}$.
Given a real sequence $\mr{w}=\{\mr{w}_a\}_{a\in\mb{Z}}$ with $1\leq\mr{w}_a\leq\mr{w}_{a'}$ for $|a|\leq|a'|$, define the Hilbert space:
\begin{equation}
\label{form7-7-1}
\mr{h}_{\mr{w}}:=\{u=\{u_a\}_{a\in\mb{Z}}\in L^{2}(\mb{T}) \mid |u|_{\mr{w}}^{2}:=\sum_{a\in\mb{Z}}\mr{w}_a^{2}|u_a|^{2}<+\infty\}.
\end{equation}
Specially, taking $\mr{w}_a:=|a|^{s}$ for $a\neq0$ and $\mr{w}_0=1$, it is the Sobolev space $H^{s}$ defined in \eqref{form7-7-1'}, and taking $\mr{w}_a:=e^{\rho|a|^{\theta}}$, it is the Gevrey space $\mc{G}_{\rho,\theta}$ defined in \eqref{form7-7-1''}.
\\\indent The scale of phase spaces
\begin{equation}
\label{form22}
\ms{P}_{\mr{w}}:=\mr{h}_{\mr{w}}\oplus\mr{h}_{\mr{w}}\ni (u,\bar{u})= (\{u_{a}\}_{a\in\mb{Z}}, \{\bar{u}_{a}\}_{a\in\mb{Z}})
\end{equation}
is endowed with the symplectic structure $-{\rm i}\sum_{a\in\mb{Z}}du_{a}\wedge d\bar{u}_{a}$.
For a Hamiltonian function $H(u,\bar{u})$, define its vector field
\begin{equation}
\label{form23}
X_{H}(u,\bar{u})=-{\rm i}\sum_{a\in\mb{Z}}\Big(\frac{\pa H}{\pa\bar{u}_{a}}\frac{\pa}{\pa u_{a}}
-\frac{\pa H}{\pa u_{a}}\frac{\pa }{\pa\bar{u}_{a}}\Big),
\end{equation}
and for two Hamiltonian functions $H(u,\bar{u})$ and $F(u,\bar{u})$, define their Poisson bracket
\begin{equation}
\label{form24}
\{H,F\}=-{\rm i}\sum_{a\in\mb{Z}}\Big(\frac{\pa H}{\pa u_a}\frac{\pa F}{\pa\bar{u}_a}-\frac{\pa H}{\pa\bar{u}_a}\frac{\pa F}{\pa u_a}\Big).
\end{equation}
\indent In a brief statement, we identify $\mb{C}^{\mb{Z}}\times\mb{C}^{\mb{Z}}\simeq
\mb{C}^{\mb{U}_{2}\times\mb{Z}}$ with $\mb{U}_{2}=\{\pm1\}$ and use the convenient notation $z=(z_{j})_{j=(\delta,a)\in\mb{U}_{2}\times\mb{Z}}$, where $$z_{j}=\left\{\begin{aligned}u_{a}, &\;\text{when}\;\delta=\;\;\,1,
\\\bar{u}_{a}, &\;\text{when}\;\delta=-1.
\end{aligned}\right.$$
Let $\bar{j}=(-\delta,a)$, then $\bar{z}_{j}=z_{\bar{j}}$. Set $|j|=|a|$, $\mr{w}_j=\mr{w}_a$ and define
\begin{equation}
\label{form25}
|z|_{\mr{w}}^{2}:=\sum_{j\in\mb{U}_{2}\times\mb{Z}}\mr{w}_j^{2}|z_{j}|^{2}=|u|^{2}_{\mr{w}}+|\bar{u}|_{\mr{w}}^{2}.
\end{equation}
For $\bs{j}=(j_{1},\cdots,j_{l})\in(\mb{U}_{2}\times\mb{Z})^{2l}$, let $j_{1}^{*}\geq\cdots\geq j_{l}^{*}$ denotes the decreasing rearrangement of  $\{|j_{1}|,\cdots,|j_{l}|\}$, and then $j_{l}^*=\mu_{\min}(\bs{j})$. Denote the monomial $z_{\bs{j}}=z_{j_{1}}\cdots z_{j_{l}}$.
Denote the set
\begin{equation}
\label{form26}
\mc{Z}_{2l}=\{\bs{j}=(\delta_{i},a_{i})_{i=1}^{2l}\in(\mb{U}_{2}\times\mb{Z})^{2l} \mid\sum_{i=1}^{2l}\delta_{i}=0\},
\end{equation}
the zero momentum set
\begin{equation}
\label{form27}
\mc{M}_{2l}=\{\bs{j}\in\mc{Z}_{2l} \mid \mc{M}(\bs{j}):=\sum_{i=1}^{2l}\delta_{i}a_{i}=0\},
\end{equation}
the resonant set
\begin{equation}
\label{form28}
\mc{R}_{2l}=\{\bs{j}\in\mc{M}_{2l} \mid \omega_{\bs{j}}:=\sum_{i=1}^{2l}\delta_{i}a_{i}^{2}=0\}
\end{equation}
and the integrable set
\begin{equation}
\label{form29}
\mc{I}_{2l}=\big\{\bs{j}\in\mc{R}_{2l}\mid \exists\;\text{permutation}\;\sigma,\;\text{s.t.}\;\forall i,\; \delta_{i}=-\delta_{\sigma(i)},\;a_{i}=a_{\sigma(i)}\big\}.
\end{equation}
Write $\mc{Z}=\bigcup_{l\geq0}\mc{Z}_{2l}$, $\mc{M}=\bigcup_{l\geq0}\mc{M}_{2l}$, $\mc{R}=\bigcup_{l\geq0}\mc{R}_{2l}$ and $\mc{I}=\bigcup_{l\geq0}\mc{I}_{2l}$.
For any $\bs{j}\in\mc{Z}$, denote $Irr({\bs{j}})$ as the non-integrable part of $\bs{j}$, namely the subsequence of maximal length $(j'_1,\cdots,j'_{2p})$ containing no action in the sense, i.e., $j'_i\neq\bar{j'_k}$ for any $i,k=1,\cdots,2p$. Let
\begin{equation}
\label{form210}
Irr(\mc{R})=\{Irr(\bs{j})\mid\bs{j}\in\mc{R}\}.
\end{equation}
Denote $\#\bs{j}$ as the length of $\bs{j}$. Notice that if $\bs{j}\in Irr(\mc{R})$, then $\#\bs{j}\geq6$.
In the following paper, we use the notation $a\lesssim b$ to mean that there a positive constant $C$ such that $a\leq Cb$.

\subsection{Rational Hamiltonian functions}
\label{subsec21}
The denominator of rational Hamiltonian functions in the Hilbert space $\mr{h}_{\mr{w}}$ consists of some small divisors with the following assumptions.
For any $N\geq1$, assume that for any $\bs{j}\in Irr(\mc{R})$ with $j_1^*\leq N$ and $\#\bs{j}=2l$,  the small divisor $\Omega_{\bs{j}}^{(4)}$ are real and linear with respect to the action $I:=\{I_a\}_{a\in\mb{Z}}$; 
the small divisor $\Omega_{\bs{j}}^{(6)}$ are real and of order at most two with respect to the action $I$, where $\Omega_{\bar{\bs{j}}}^{(4)}=-\Omega_{\bs{j}}^{(4)}$, $\Omega_{\bar{\bs{j}}}^{(6)}=-\Omega_{\bs{j}}^{(6)}$. For any $a\in\mb{Z}$, the coefficients of $\pa_{I_a}\Omega_{\bs{j}}^{(4)}$ and $\pa_{I_a}\Omega_{\bs{j}}^{(6)}$ are less than $l$ up to a constant multiply.
Besides, assume that there exists $\gamma\in(0,1)$ such that
\begin{equation}
\label{form7-6-1}
|\Omega_{\bs{j}}^{(4)}|>\gamma(2N)^{-12l}|z|_{\mr{w}}^{2}\mr{w}_{\mu_{\min}(\bs{j})}^{-2},
\end{equation}
\begin{equation}
\label{form7-6-2}
|\Omega_{\bs{j}}^{(6)}|>\gamma(2N)^{-12l}|z|_{\mr{w}}^{2}\max\{\mr{w}_{\mu_{\min}(\bs{j})}^{-2},\gamma|z|_{\mr{w}}^{2}\}.
\end{equation}

Firstly, we define a family of index sets to control the rational fractions.
\begin{definition}
\label{def21}
For any $N\geq1$ and integer $l\geq2$, denote $\mc{H}_{2l}^{N}$ as the family of all the elements $\Gamma_{2l}$, where
\begin{equation}
\label{form41}
\Gamma_{2l}\subseteq\mc{R}\times\bigcup_{q\geq0}(Irr(\mc{R}))^{q}\times\bigcup_{q'\geq0}(Irr(\mc{R}))^{q'}\times\mb{N}
\end{equation}
satisfies the following conditions:
\begin{enumerate}
%
\item\label{def41-3} for any $(\bs{j},\bs{h},\bs{k},n)\in\Gamma_{2l}$, one has 
\begin{itemize}
\item[(1-1)] the order $2l$, i.e., $\#\bs{j}-2\#\bs{h}-4\#\bs{k}=2l$;
\item[(1-2)] $\#\bs{j}\geq6$ and $(Irr(\bs{j}))_{1}^{*}\leq N$;
\item[(1-3)]\label{def41-2-3} $\bs{h}=(\bs{h}_{m})_{m=1}^{\#\bs{h}}$ with $6\leq\#\bs{h}_{m}\leq\#\bs{j}$ and $(h_{m})_{1}^{*}\leq N$;
\item[(1-4)]\label{def41-2-4} $\bs{k}=(\bs{k}_{m})_{m=1}^{\#\bs{k}}$ with $6\leq\#\bs{k}_{m}\leq\#\bs{j}$ and $(k_{m})_{1}^{*}\leq N$;
\item[(1-5)] $n\leq\#\bs{h}$;
\end{itemize}
\item\label{def41-2} $\max_{\tilde{\bs{j}}\in\mc{R}_{2p}}\#\{(\bs{j},\bs{h},\bs{k},n)\in\Gamma_{2l}\mid \bs{j}=\tilde{\bs{j}}\}\leq (2p)^{2p-5}$;
\item\label{def41-4} for any $(\bs{j},\bs{h},\bs{k},n)\in\Gamma_{2l}$, one has the global controls
\begin{equation}
\label{form42}
\prod_{m=1}^{\#\bs{h}}\mr{w}_{\mu_{\min}(\bs{h}_{m})}^{2}\leq\frac{\prod_{m=1}^{\#\bs{j}}\mr{w}_{j_{m}}}{\mr{w}_{j^{*}_{1}}^{2}},
\end{equation}
\begin{equation}
\label{form43}
\big(\prod_{m=1}^{\#\bs{h}}\mr{w}_{\mu_{\min}(\bs{h}_{m})}^{2}\big)\max_{m}\{\mr{w}_{\mu_{\min}(\bs{h}_{m})}^{2},\mr{w}_{\mu_{\min}(\bs{k}_{m})}^{2}\}\leq\prod_{m=1}^{\#\bs{j}}\mr{w}_{j_{m}}.
\end{equation}
\end{enumerate}
\end{definition}
Then we define rational Hamiltonian functions in the Hilbert space \eqref{form7-7-1}.
\begin{definition}
\label{def22}
Being given $\Gamma_{2l}\in\mc{H}_{2l}^{N}$, define the formal rational Hamiltonian function in the Hilbert space \eqref{form7-7-1}:
\begin{equation}
\label{form44}
Q_{\Gamma_{2l}}[\bs{c}](z)=\sum_{J:=(\bs{j},\bs{h},\bs{k},n)\in\Gamma_{2l}}c_{J}\frac{z_{\bs{j}}}{\prod\limits_{m=1}^{n}\Omega^{(4)}_{\bs{h}_{m}}\prod\limits_{m=n+1}^{\#\bs{h}}\Omega^{(6)}_{\bs{h}_{m}}\prod\limits_{m=1}^{\#\bs{k}}\Omega^{(6)}_{\bs{k}_{m}}}
\end{equation}
with convention that $\prod_{m=1}^{0}\bullet=1$,  
where the coefficient $\bs{c}=\{c_{J}\}_{J\in\Gamma_{2l}}$ satisfies
the estimate $\|\bs{c}\|_{\ell^{\infty}}:=\sup_{J\in\Gamma_{2l}}|c_{J}|<+\infty$ and the reality condition 
$$c_{(\bs{j},\bs{h},\bs{k},n)}=\bar{c}_{(\bar{\bs{j}},\bs{h},\bs{k},n)}=(-1)^{\#\bs{h}}c_{(\bs{j},\bar{\bs{h}},\bs{k},n)}=(-1)^{\#\bs{k}}c_{(\bs{j},\bs{h},\bar{\bs{k}},n)}.$$ 
\end{definition}
\begin{remark}
\label{re41}
There are several notes for the above two definitions.
\begin{enumerate}[(1)]
%
\item\label{re41-2} If $\Gamma_{2l}\in\mc{H}_{2l}^{N}$ satisfies $\Gamma_{2l}\subseteq\mc{I}\times\bigcup_{q\geq0}(Irr(\mc{R}))^{q}\times\bigcup_{q'\geq0}(Irr(\mc{R}))^{q'}\times\mb{N}$, then $Q_{\Gamma_{2l}}$ is an integrable function, i.e. it only depends on the action $I$.
\item\label{re41-3} Being given $\Gamma_{2l}\in\mc{H}_{2l}^{N}$, if $\#\bs{j}\lesssim l$ for any $(\bs{j},\bs{h},\bs{k},n)\in\Gamma_{2l}$, then the rational Hamiltonian function \eqref{form44} can be rewritten as
\begin{equation}
\label{form45}
Q_{\Gamma_{2l}}[\bs{c}](z)=\sum_{(p,q,q')\in\mc{F}_{2l}}\sum_{(\bs{j},\bs{h},\bs{k},n)\in\Gamma_{2l}^{(p,q,q')}}c_{J}\frac{z_{\bs{j}}}{\prod\limits_{m=1}^{n}\Omega^{(4)}_{\bs{h}_{m}}\prod\limits_{m=n+1}^{q}\Omega^{(6)}_{\bs{h}_{m}}\prod\limits_{m=1}^{q'}\Omega^{(6)}_{\bs{k}_{m}}},
\end{equation}
where $\mc{F}_{2l}\subseteq\mb{N}^{3}$ is a finite set with $\#\mc{F}_{2l}\lesssim l^{3}$ and for any $(p,q,q')\in\mc{F}_{2l}$, one has $p-q-2q'=l$, $\Gamma^{(p,q,q')}_{2l}=\Gamma_{2l}\bigcap\big(\mc{R}_{2p}\times(Irr(\mc{R}))^{q}\times(Irr(\mc{R}))^{q'}\times\mb{N}\big)$. Besides, the condition \eqref{def41-2}  in \Cref{def21} shows that if $\bs{j}$ is fixed, then the number of $(\bs{j},\bs{h},\bs{k},n)$ is finite and can be controlled $p$.
%
\item\label{re41-4} In the iterative process, the denominator of solutions of the homological equations is formed by the denominator of the previous term and the small divisor, which corresponds to the irreducible parts of the numerator's index in the previous term. Thus it is nature to propose the conditions (1-3) and (1-4)  in \Cref{def21}, namely the number of all indexes $\bs{h}_{m},\bs{k}_{m}$ in the denominator can be controlled by the number of the index $\bs{j}$ in the numerator.
%
\item\label{re41-5} The conditions \eqref{form42} and \eqref{form43} ensure that the rational Hamiltonian function $Q_{\Gamma_{2l}}$ are well defined in the Hilbert space $\mr{h}_{\mr{w}}$, that is to say, when we use the small divisor conditions \eqref{form7-6-1} and \eqref{form7-6-2} to estimate $\Omega_{\bs{h}_m}^{(4)}$, $\Omega_{\bs{h}_m}^{(6)}$ and $\Omega_{\bs{k}_m}^{(6)}$ in \eqref{form44}, the conditions \eqref{form42} and \eqref{form43} are exactly adequate to bound the rational function and its derivative. 
\end{enumerate}
\end{remark}
The following proposition shows that the structure of rational fractions  in \Cref{def21} is well kept under the Poisson bracket.
\begin{proposition}
\label{le42}
Given integers $l_{1},l_{2}\geq2$, for any $\Gamma_{2l_{1}}\in\mc{H}_{2l_{1}}^{N}$ and $\Gamma_{2l_{2}}\in\mc{H}_{2l_{2}}^{N}$, there exists $\Gamma_{2l}\in\mc{H}_{2l}^{N}$ with $l:=l_{1}+l_{2}-1$ such that
\begin{equation}
\label{form423}
\{Q_{\Gamma_{2l_{1}}}[\tilde{\bs{c}}],Q_{\Gamma_{2l_{2}}}[\tilde{\tilde{\bs{c}}}]\}=Q_{\Gamma_{2l}}[\bs{c}],
\end{equation}
where $Q_{\Gamma_{2l_{1}}}[\tilde{\bs{c}}],Q_{\Gamma_{2l_{2}}}[\tilde{\tilde{\bs{c}}}],Q_{\Gamma_{2l}}[\bs{c}]$ are defined in \eqref{form44}.
Moreover, if $\#\bs{j}'\lesssim l_1$, $\#\bs{j}''\lesssim l_2$ for any $(\bs{j}',\bs{h}',\bs{k}',n_1)\in\Gamma_{2l_1}$, $(\bs{j}'',\bs{h}'',\bs{k}'',n_2)\in\Gamma_{2l_2}$, then one has the coefficient estimate
\begin{equation}
\label{form424}
\|\bs{c}\|_{\ell^{\infty}}\lesssim l^{5}\|\tilde{\bs{c}}\|_{\ell^{\infty}}\|\tilde{\tilde{\bs{c}}}\|_{\ell^{\infty}}.
\end{equation}
\end{proposition}

In the following proposition, we estimate the vector field of rational functions in the Hilbert space $\mr{h}_{\mr{w}}$.
\begin{proposition}
\label{le41}
For any rational function $Q_{\Gamma_{2l}}[\bs{c}](z)$ in \eqref{form45} with $\Gamma_{2l}\in\mc{H}_{2l}^{N}$ and $\|\bs{c}\|_{\ell^\infty}<+\infty$, one has
\begin{equation}
\label{form46}
|X_{Q_{\Gamma_{2l}}[\bs{c}]}(z)|_{\mr{w}}\lesssim l^{3}\max_{(p,q,q')\in\mc{F}_{2l}}(2p)^{2p-3}\|\bs{c}\|_{\ell^{\infty}}\frac{(2N)^{12p(q+q')+15p}}{\gamma^{q+2q'+1}}|z|_{\mr{w}}^{2l-1}.
\end{equation}
\end{proposition}
The above two propositions will be proved in \Cref{sec3}.
\subsection{Nonlinear Schr\"{o}dinger equation in Gevrey spaces}
\label{subsec22}
Under the standard inner product on $L^{2}(\mb{T})$, the nonlinear Schr\"{o}dinger equation \eqref{form11} can be written as the following form:
\begin{equation}
\label{form12}
{\rm i}\frac{\pa u}{\pa t}=\frac{\pa H}{\pa \bar{u}}
\end{equation}
with the Hamiltonian function
\begin{equation}
\label{form13}
H=\frac{1}{2\pi}\int_{\mb{T}}|u_x|^{2}dx+\frac{1}{2\pi}\int_{\mb{T}}g(|u|^{2})dx,
\end{equation}
where $g(y)=\int_{0}^{y}\varphi(\eta)d\eta$.

For Gevrey spaces $\mc{G}_{\rho,\theta}$ in \eqref{form7-7-1''}, introduce the complex Gaussian measure formally defined by
\begin{equation}
d\mu_g=\frac{e^{-\sum_{a\in\mb{Z}}e^{2\rho|a|^{\theta}}|a|^{2}|u_{a}|^{2}}dud\bar{u}}{\int_{\mc{G}_{\rho,\theta}}e^{-\sum_{a\in\mb{Z}}e^{2\rho|a|^{\theta}}|a|^{2}|u_a|^{2}}dud\bar{u}}.
\end{equation}
Notice that the measure $\mu_g$ is viewed as weak limit of finite dimensional Gaussian measures. By Theorem 2.4 in \cite{K19}, the measure $\mu_g$ is countably additive on the space $\mc{G}_{\rho,\theta}$ with $\mu_g(\mc{G}_{\rho,\theta})=1$. Denote by $B_{\rho,\theta}(R)$ the open ball centered at the origin and of radius $R>0$ in $\mc{G}_{\rho,\theta}$. Then $0<\mu_g(B_{\rho,\theta}(R))<1$, and thus we could further define the Gaussian measure $\mu$ in the unit ball $B_{\rho,\theta}(1)$ by
\begin{equation}
\label{form15}
d\mu=\frac{d\mu_g}{\mu_g(B_{\rho,\theta}(1))}.
\end{equation}
Then we have the following result.
\begin{theorem}
\label{th11}
Fix $\rho>0$ and $\theta\in(0,1)$. For any $\beta\in(0,1)$, there exist $\varepsilon_{0}>0$ and an open set $\mc{V}_{\rho,\theta,\beta}\subseteq B_{\rho,\theta}(\varepsilon_0)$ such that for any $0<\varepsilon\leq\varepsilon_{0}$, if the initial datum $u(0,x)\in\mc{V}_{\rho,\theta,\beta}\cap B_{\rho,\theta}(\varepsilon)$,
then for any
\begin{equation}
\label{form7-2-1}
|t|\leq\varepsilon^{-|\ln\varepsilon|^{\beta}},
\end{equation}
the solution $u(t,x)$ of the nonlinear Schr\"{o}dinger equation \eqref{form11} satisfies
\begin{align}
\label{form16}
\|u(t,x)\|_{\rho,\theta}\leq2\varepsilon,&
\\\label{form17}
\sup_{a\in\mb{Z}}e^{2\rho|a|^{\theta}}|I_{a}(t)-I_{a}(0)|&\leq\varepsilon^{\frac{5}{2}},
\end{align}
where  the action $I_a=|u_a|^{2}$.
Moreover, one has the measure estimate
\begin{equation}
\label{form18}
\mu(\varepsilon u\in\mc{V}_{\rho,\theta,\beta})\geq1-\varepsilon^{\frac{1}{8}}.
\end{equation}
\end{theorem}
\begin{remark}
\label{re12}
To our knowledge, there is no result of sub-exponentially long time stability for equations with internal parameters in previous papers. Now, for instance, we compare \eqref{form7-2-1} with the stability time of equations with external parameters in \cite{CMW20,BMP20}.
%
%
In \cite{CMW20}, for derivate nonlinear schr\"{o}dinger equations in Gevrey spaces, the stability time is of order $\varepsilon^{-|\ln\varepsilon|^{\beta}}$ for any $0<\beta<\frac{1}{2}$; in \cite{BMP20}, for  nonlinear schr\"{o}dinger equations in Gevrey spaces, the stability time is of order $\varepsilon^{-|\ln\varepsilon|^{\frac{\theta}{4}}}$;
%
in the present paper, the stability time is of order $\varepsilon^{-|\ln\varepsilon|^{\beta}}$  for any $0<\beta<1$.  
The difference is due to coefficient estimates.
Very roughly speaking, in \cite{CMW20}, after $r$ iterative steps, the coefficients are bounded by $N^{r^{3}}$; in \cite{BMP20}, due to different small divisor conditions, the coefficients are bounded by $e^{r^{1+\frac{3}{\theta}}}$; in the present paper, the coefficients are bounded by $N^{r^{2}}$.  
\end{remark}
\begin{remark}
\label{re13}
Similarly with \cite{BFG20b,BG21}, we could assume that the actions $\{I_{a}\}_{a\in\mb{Z}}$ are independent and uniformly distributed in an infinite cube $\prod_{a\in\mb{Z}_{*}}(0,e^{-2\rho|a|^{\theta}}|a|^{-2})$, and then estimate the measure.
%
%
Moreover, instead of the measure \eqref{form15}, we could use the following Gaussian measure
$$d\mu=\frac{e^{-\sum_{a\in\mb{Z}}e^{2\rho'|a|^{\theta}}|u_{a}|^{2}}dud\bar{u}}{\int_{B_{\rho,\theta}(1)}e^{-\sum_{a\in\mb{Z}}e^{2\rho'|a|^{\theta}}|u_a|^{2}}dud\bar{u}},$$
where $\rho'$ is a constant with $\rho'>\rho$.
For long time stability by Gaussian and Gibbs measure, see \cite{BMT19}, in which one dimensional defocusing nonlinear Schr\"{o}dinger equation is investigated in the low regular Sobolev space $H^s$, $s<\frac{1}{2}$.
\end{remark}
\begin{remark}
Instead of the Gevrey space \eqref{form7-7-1''}, \Cref{th11} still holds true in the following analytic space
$$\mc{G}_{\alpha,\rho,\theta}:=\{u=\{u_{a}\}_{a\in\mb{Z}} \mid \|u\|^{2}_{\alpha,\rho,\theta}:=\sum_{a\in\mb{Z}}e^{2\alpha|a|+2\rho|a|^{\theta}}|u_a|^{2}<+\infty\},$$
where $\alpha$ is a positive constant. The proof is completely parallel.
\end{remark}

\section{Proof of properties of rational Hamiltonian functions}
\label{sec3}

\subsection{Proof of \Cref{le42}}
\label{subsec31}
%
Denote
$$Q_{\Gamma_{2l}}[\bs{c}](z)=\sum_{J:=(\bs{j},\bs{h},\bs{k},n)\in\Gamma_{2l}}c_{J}\frac{z_{\bs{j}}}{\prod\limits_{m=1}^{n}\Omega^{(4)}_{\bs{h}_{m}}\prod\limits_{m=n+1}^{\#\bs{h}}\Omega^{(6)}_{\bs{h}_{m}}\prod\limits_{m=1}^{\#\bs{k}}\Omega^{(6)}_{\bs{k}_{m}}}.$$
For any $J':=(\bs{j}',\bs{h}',\bs{k}',n_1)\in\Gamma_{2l_1}$ and $J'':=(\bs{j}'',\bs{h}'',\bs{k}'',n_2)\in\Gamma_{2l_2}$, one has
\begin{equation}
\label{form427}
\prod_{m=1}^{\#\bs{h}'}\mr{w}_{\mu_{\min}(\bs{h}'_{m})}^{2}\leq\frac{\prod_{m=1}^{\#\bs{j}'}\mr{w}_{j'_{m}}}{\mr{w}_{j'^{*}_{1}}^{2}},
\end{equation}
\begin{equation}
\label{form428}
\big(\prod_{m=1}^{\#\bs{h}'}\mr{w}_{\mu_{\min}(\bs{h}'_{m})}^{2}\big)\max_{m}\{\mr{w}_{\mu_{\min}(\bs{h}'_{m})}^{2},\mr{w}_{\mu_{\min}(\bs{k}'_{m})}^{2}\}\leq\prod_{m=1}^{\#\bs{j}'}\mr{w}_{j'_{m}},
\end{equation}
\begin{equation}
\label{form429}
\prod_{m=1}^{\#\bs{h}''}\mr{w}_{\mu_{\min}(\bs{h}''_{m})}^{2}\leq\frac{\prod_{m=1}^{\#\bs{j}''}\mr{w}_{j''_{m}}}{\mr{w}_{j''^{*}_{1}}^{2}},
\end{equation}
\begin{equation}
\label{form430}
\big(\prod_{m=1}^{\#\bs{h}''}\mr{w}_{\mu_{\min}(\bs{h}''_{m})}^{2}\big)\max_{m}\{\mr{w}_{\mu_{\min}(\bs{h}''_{m})}^{2},\mr{w}_{\mu_{\min}(\bs{k}''_{m})}^{2}\}\leq\prod_{m=1}^{\#\bs{j}''}\mr{w}_{j''_{m}}.
\end{equation}
Then we will mainly prove that conditions \eqref{form42} and \eqref{form43} still hold for any $J\in\Gamma_{2l}$.
In the following, we discuss $\{Q_{\Gamma_{2l_{1}}}[\tilde{\bs{c}}],Q_{\Gamma_{2l_{2}}}[\tilde{\tilde{\bs{c}}}]\}$ in four different cases.
\\\indent 1) Consider the first case that denominators do not appear in the poisson bracket. Then $Q_{\Gamma_{2l}^{(1)}}[\bs{c}^{(1)}]$ is of the form
\begin{align*}
\sum_{\substack{J'\in\Gamma_{2l_1}\\J''\in\Gamma_{2l_2}}}
\tilde{c}_{J'}\tilde{\tilde{c}}_{J''}
\frac{\{z_{\bs{j}'},z_{\bs{j}''}\}}{\prod\limits_{m=1}^{n_{1}}\Omega^{(4)}_{\bs{h}'_{m}}\prod\limits_{m=n_{1}+1}^{\#\bs{h}'}\Omega^{(6)}_{\bs{h}'_{m}}\prod\limits_{m=1}^{\#\bs{k}'}\Omega^{(6)}_{\bs{k}'_{m}}\prod\limits_{m=1}^{n_{2}}\Omega^{(4)}_{\bs{h}''_{m}}\prod\limits_{m=n_{2}+1}^{\#\bs{h}''}\Omega^{(6)}_{\bs{h}''_{m}}\prod\limits_{m=1}^{\#\bs{k}''}\Omega^{(6)}_{\bs{k}''_{m}}}.
\end{align*}
For any $J\in\Gamma_{2l}^{(1)}$, one has 
\begin{equation}
\label{form7-10-1}
\#\bs{j}=\#\bs{j}'+\#\bs{j}''-2,\quad \bs{h}=(\bs{h}',\bs{h}''),\quad \bs{k}=(\bs{k}',\bs{k}'')\quad\text{and}\quad n=n_1+n_2.
\end{equation}
Hence,
\begin{align}
\label{form7-12-2}
\#\bs{j}-2\#\bs{h}-4\#\bs{k}&=\#\bs{j}'+\#\bs{j}''-2-2(\#\bs{h}'+\#\bs{h}'')-4(\#\bs{k}'+\#\bs{k}'')
\\\notag&=2l_{1}+2l_{2}-2
\\\notag&=2l.
\end{align}
By the facts $\#\bs{j}'-2\geq0$, $\#\bs{j}''-2\geq0$, one has 
$6\leq\#\bs{h}'_{m},\#\bs{k}'_{m}\leq\#\bs{j}'\leq\#\bs{j}$ and $6\leq\#\bs{h}''_{m},\#\bs{k}''_{m}\leq\#\bs{j}''\leq\#\bs{j}$, namely 
$6\leq\#\bs{h}_{m},\#\bs{k}_{m}\leq\#\bs{j}$.
%
\\\indent In view of \eqref{form7-10-1}, if $\bs{j}\in\mc{R}_{2p}$ and $\#\bs{j}'=2p_1$ are fixed, then $\#\bs{j}''=2p-2p_1+2$, and there are at most $2p_1$ different $\bs{j}'$ and $2p-2p_1+2$ different $\bs{j}''$;
if $J'$, $J''$ are fixed, then $\bs{h},\bs{k},n$ are determined. Hence, one has
\begin{align}
\label{form7-12-1}
&\max_{\tilde{\bs{j}}\in\mc{R}_{2p}}\#\{J\in\Gamma_{2l}^{(1)}\mid \bs{j}=\tilde{\bs{j}}\}
\\\notag\leq&\sum_{p_1=1}^{p}2p_1\max_{\tilde{\bs{j}'}\in\mc{R}_{2p_1}}\#\{J'\in\Gamma_{2l_1}\mid \bs{j}'=\tilde{\bs{j}'}\}
\\\notag&\;\qquad(2p-2p_1+2)\max_{\tilde{\bs{j}''}\in\mc{R}_{2p-2p_1+2}}\#\{J''\in\Gamma_{2l_2}\mid \bs{j}''=\tilde{\bs{j}''}\}
\\\notag\leq&\sum_{p_1=1}^{p}(2p_1)^{2p_1-4}(2p-2p_1+2)^{2p-2p_1-2}
\\\notag\leq& p(2p)^{2p-6}.
\end{align}
\indent In order to prove $\Gamma_{2l}^{(1)}\in\mc{H}_{2l}^{N}$, we need prove estimates \eqref{form42} and \eqref{form43}.
Without loss of generality, let $\bs{j}=(j'_{1},\cdots,j'_{k_1-1},j'_{k_1+1},\cdots,j'_{\#\bs{j}'},j''_{1},\cdots,j''_{k_2-1},j''_{k_2+1},\cdots,j''_{\#\bs{j}''})$ and $j'_{k_1}=\bar{j''_{k_2}}$.
By conditions \eqref{form427} and \eqref{form429}, one has
\begin{align*}
\prod_{m=1}^{\#\bs{h}}\mr{w}_{\mu_{\min}(\bs{h}_{m})}^{2}&=\Big(\prod_{m=1}^{\#\bs{h}'}\mr{w}_{\mu_{\min}(\bs{h}'_{m})}^{2}\Big)\Big(\prod_{m=1}^{\#\bs{h}''}\mr{w}_{\mu_{\min}(\bs{h}''_{m})}^{2}\Big)
\\&\leq\frac{\prod_{m=1}^{\#\bs{j}'}\mr{w}_{j'_{m}}}{\mr{w}_{j'^{*}_{1}}^{2}}\frac{\prod_{m=1}^{\#\bs{j}''}\mr{w}_{j''_{m}}}{\mr{w}_{j''^{*}_{1}}^{2}}
\\&\leq\frac{\prod_{m=1}^{\#\bs{j}}\mr{w}_{j_{m}}}{\mr{w}_{j^{*}_{1}}^{2}},
\end{align*}
where the last inequality follows from the facts
$\mr{w}_{j'_{k_1}}=\mr{w}_{j''_{k_2}}\leq\min\{\mr{w}_{j'^{*}_1},\mr{w}_{j''^{*}_1}\}$ and $\mr{w}_{j_1^{*}}\leq\max\{\mr{w}_{j'^{*}_1}, \mr{w}_{j''^{*}_1}\}$.
Without loss of generality, let
$$\max_{m}\{\mr{w}_{\mu_{\min}(\bs{h}_{m})}^{2},\mr{w}_{\mu_{\min}(\bs{k}_{m})}^{2}\}=\max_{m}\{\mr{w}_{\mu_{\min}(\bs{h}''_{m})}^{2},\mr{w}_{\mu_{\min}(\bs{k}''_{m})}^{2}\},$$
and then by \eqref{form427}, \eqref{form430}, one has
\begin{align*}
&\Big(\prod_{m=1}^{\#\bs{h}}\mr{w}_{\mu_{\min}(\bs{h}_{m})}^{2}\Big)\max_{m}\{\mr{w}_{\mu_{\min}(\bs{h}_{m})}^{2},\mr{w}_{\mu_{\min}(\bs{k}_{m})}^{2}\}
\\=&\Big(\prod_{m=1}^{\#\bs{h}'}\mr{w}_{\mu_{\min}(\bs{h}'_{m})}^{2}\Big)\Big(\prod_{m=1}^{\#\bs{h}''}\mr{w}_{\mu_{\min}(\bs{h}''_{m})}^{2}\Big)\max_{m}\{\mr{w}_{\mu_{\min}(\bs{h}''_{m})}^{2},\mr{w}_{\mu_{\min}(\bs{k}''_{m})}^{2}\}
\\\leq&\frac{\prod_{m=1}^{\#\bs{j}'}\mr{w}_{j'_{m}}}{\mr{w}_{j'^{*}_{1}}^{2}}\prod_{m=1}^{\#\bs{j}''}\mr{w}_{j''_{m}}
\\\leq&\prod_{m=1}^{\#\bs{j}}\mr{w}_{j_{m}},
\end{align*}
where the last inequality follows from the fact $\mr{w}_{j'_{k_1}}=\mr{w}_{j''_{k_2}}\leq\mr{w}_{j'^{*}_1}$.
Hence, one has $\Gamma_{2l}^{(1)}\in\mc{H}_{2l}^{N}$ with $l=l_1+l_2-1$. 

Besides, in view of \eqref{form7-10-1}, for any given $J\in\Gamma_{2l}^{(1)}$, if $\#\bs{h}',\#\bs{k}',n_1$ are fixed, then $\#\bs{j}'$, $\#\bs{j}''$, $\bs{h}'$, $\bs{h}''$, $\bs{k}'$, $\bs{k}''$, $n_2$ are determined; and there are at most $\#\bs{j}'$ different $\bs{j}'$ and $\#\bs{j}''$ different $\bs{j}''$. Hence, by $\#\bs{j}',\#\bs{j}'',\#\bs{h}',\#\bs{k}',n_1\lesssim l$, one has
\begin{align}
\label{form7-14-1}
\|\bs{c}^{(1)}\|_{\ell^{\infty}}&=\sup_{J\in\Gamma_{2l}^{(1)}}|c_{J}|
\\\notag&\lesssim l^{5}\sup_{\substack{J'\in\Gamma_{2l_1}\\J''\in\Gamma_{2l_2}}}|\bs{c}_{J'}||\bs{c}_{J''}|
\\\notag&= l^{5}\|\tilde{\bs{c}}\|_{\ell^{\infty}}\|\tilde{\tilde{\bs{c}}}\|_{\ell^{\infty}}.
\end{align}

2)Consider the second case that $\Omega^{(4)}$ appears in the poisson bracket. Without loss of generality, let $n_{1},n_2\geq1$ and then $Q_{\Gamma_{2l}^{(2)}}[\bs{c}^{(2)}]$ is of the form
\begin{align*}
&\sum_{\substack{J'\in\Gamma_{2l_1}\\J''\in\Gamma_{2l_2}}}\sum_{i=1}^{n_1}\tilde{c}_{J'}\tilde{\tilde{c}}_{J''}\frac{z_{\bs{j}'}\{\frac{1}{\Omega^{(4)}_{\bs{h}'_{i}}},z_{\bs{j}''}\}}{\prod\limits_{m\neq i}\Omega^{(4)}_{\bs{h}'_{m}}\prod\limits_{m=n_{1}+1}^{\#\bs{h}'}\Omega^{(6)}_{\bs{h}'_{m}}\prod\limits_{m=1}^{\#\bs{k}'}\Omega^{(6)}_{\bs{k}'_{m}}\prod\limits_{m=1}^{n_{2}}\Omega^{(4)}_{\bs{h}''_{m}}\prod\limits_{m=n_{2}+1}^{\#\bs{h}''}\Omega^{(6)}_{\bs{h}''_{m}}\prod\limits_{m=1}^{\#\bs{k}''}\Omega^{(6)}_{\bs{k}''_{m}}}
\\+&\sum_{\substack{J'\in\Gamma_{2l_1}\\J''\in\Gamma_{2l_2}}}\sum_{i=1}^{n_2}\tilde{c}_{J'}\tilde{\tilde{c}}_{J''}\frac{\{z_{\bs{j}'},\frac{1}{\Omega^{(4)}_{\bs{h}''_{i}}}\}z_{\bs{j}''}}{\prod\limits_{m=1}^{n_1}\Omega^{(4)}_{\bs{h}'_{m}}\prod\limits_{m=n_{1}+1}^{\#\bs{h}'}\Omega^{(6)}_{\bs{h}'_{m}}\prod\limits_{m=1}^{\#\bs{k}'}\Omega^{(6)}_{\bs{k}'_{m}}\prod\limits_{m\neq i}\Omega^{(4)}_{\bs{h}''_{m}}\prod\limits_{m=n_{2}+1}^{\#\bs{h}''}\Omega^{(6)}_{\bs{h}''_{m}}\prod\limits_{m=1}^{\#\bs{k}''}\Omega^{(6)}_{\bs{k}''_{m}}}.
\end{align*}
Then for any $J\in\Gamma_{2l}^{(2)}$, one has 
\begin{equation}
\label{form7-10-2}
\begin{array}{l}
\displaystyle \bs{j}=(\bs{j}',\bs{j}''),
\\\displaystyle \bs{h}=(\bs{h}'_i,\bs{h}',\bs{h}'') \quad \text{or} \quad \bs{h}=(\bs{h}''_i,\bs{h}',\bs{h}''),
\\\displaystyle \bs{k}=(\bs{k}',\bs{k}''),
\\\displaystyle n=n_1+n_2+1. 
\end{array}
\end{equation}
Hence,
\begin{align*}
\#\bs{j}-2\#\bs{h}-4\#\bs{k}&=\#\bs{j}'+\#\bs{j}''-2(\#\bs{h}'+\#\bs{h}''+1)-4(\#\bs{k}'+\#\bs{k}'')
\\&=2l_{1}+2l_{2}-2
\\&=2l,
\end{align*}
and $6\leq\#\bs{h}_{m},\#\bs{k}_{m}\leq\#\bs{j}'+\#\bs{j}''=\#\bs{j}$.
\\\indent In view of \eqref{form7-10-2}, if $\bs{j}\in\mc{R}_{2p}$ and $\#\bs{j}'=2p_1$ are fixed, then $\bs{j}', \bs{j}''$ are determined and $\#\bs{j}''=2p-2p_1$; if $J'$, $J''$ are fixed, $\bs{k},n$ are determined and there are at most $n_1+n_2$ different $\bs{h}$. By the fact $n_1+n_2\leq\#\bs{h}'+\#\bs{h}''<p$, one has
\begin{align}
\label{form7-12-2}
&\max_{\tilde{\bs{j}}\in\mc{R}_{2p}}\#\{J\in\Gamma_{2l}^{(2)}\mid \bs{j}=\tilde{\bs{j}}\}
\\\notag\leq&\sum_{p_1=1}^{p}p\max_{\tilde{\bs{j}'}\in\mc{R}_{2p_1}}\#\{J'\in\Gamma_{2l_1}\mid \bs{j}'=\tilde{\bs{j}'}\}\max_{\tilde{\bs{j}''}\in\mc{R}_{2(p-p_1)}}\#\{J''\in\Gamma_{2l_2}\mid \bs{j}''=\tilde{\bs{j}''}\}
\\\notag\leq&p\sum_{p_1=1}^{p}(2p_{1})^{2p_1-5}(2p-2p_1)^{2p-2p_1-5}
\\\notag\leq&p^{2}(2p)^{2p-10}.
\end{align}
\indent Now, we will prove the estimates \eqref{form42} and \eqref{form43}.  Without loss of generality, consider the case $\bs{h}=(\bs{h}'_i,\bs{h}',\bs{h}'')$. Then $\mu_{\min}(\bs{h}'_{i})\leq j''^{*}_1$ and thus $\mr{w}_{\mu_{\min}(\bs{h}'_{i})}\leq\mr{w}_{j''^{*}_1}$.
\\$\bullet$ When $j_{1}^{*}=j'^{*}_1$, by \eqref{form427}, \eqref{form429}, one has
\begin{align*}
\prod_{m=1}^{\#\bs{h}}\mr{w}_{\mu_{\min}(\bs{h}_{m})}^{2}&=\Big(\prod_{m=1}^{\#\bs{h}'}\mr{w}_{\mu_{\min}(\bs{h}'_{m})}^{2}\Big)\Big(\prod_{m=1}^{\#\bs{h}''}\mr{w}_{\mu_{\min}(\bs{h}''_{m})}^{2}\Big)\mr{w}_{\mu_{\min}(\bs{h}'_{i})}^{2}
\\&\leq\frac{\prod_{m=1}^{\#\bs{j}'}\mr{w}_{j'_{m}}}{\mr{w}_{j'^{*}_{1}}^{2}}\frac{\prod_{m=1}^{\#\bs{j}''}\mr{w}_{j''_{m}}}{\mr{w}_{j''^{*}_{1}}^{2}}\mr{w}_{\mu_{\min}(\bs{h}'_{i})}^{2}
\\&\leq\frac{\prod_{m=1}^{\#\bs{j}}\mr{w}_{j_{m}}}{\mr{w}_{j^{*}_{1}}^{2}};
\end{align*}
$\bullet$ when $j_{1}^{*}=j''^{*}_1$, by \eqref{form428} and \eqref{form429}, one has
\begin{align*}
\prod_{m=1}^{\#\bs{h}}\mr{w}_{\mu_{\min}(\bs{h}_{m})}^{2}&=\mr{w}_{\mu_{\min}(\bs{h}'_{i})}^{2}\Big(\prod_{m=1}^{\#\bs{h}'}\mr{w}_{\mu_{\min}(\bs{h}'_{m})}^{2}\Big)\Big(\prod_{m=1}^{\#\bs{h}''}\mr{w}_{\mu_{\min}(\bs{h}''_{m})}^{2}\Big)
\\&\leq\Big(\prod_{m=1}^{\#\bs{j}'}\mr{w}_{j'_{m}}\Big)\frac{\prod_{m=1}^{\#\bs{j}''}\mr{w}_{j''_{m}}}{\mr{w}_{j''^{*}_{1}}^{2}}
\\&=\frac{\prod_{m=1}^{\#\bs{j}}\mr{w}_{j_{m}}}{\mr{w}_{j^{*}_{1}}^{2}}.
\end{align*}
Hence, the estimate \eqref{form42} holds. Then, we are going to prove the estimate \eqref{form43}.
\\$\bullet$ When $\max_{m}\{\mr{w}_{\mu_{\min}(\bs{h}_{m})}^{2},\mr{w}_{\mu_{\min}(\bs{k}_{m})}^{2}\}=\max_{m}\{\mr{w}_{\mu_{\min}(\bs{h}''_{m})}^{2},\mr{w}_{\mu_{\min}(\bs{k}''_{m})}^{2}\}$, by \eqref{form428} and \eqref{form430}, one has
\begin{align*}
&\Big(\prod_{m=1}^{\#\bs{h}}\mr{w}_{\mu_{\min}(\bs{h}_{m})}^{2}\Big)\max_{m}\{\mr{w}_{\mu_{\min}(\bs{h}_{m})}^{2},\mr{w}_{\mu_{\min}(\bs{k}_{m})}^{2}\}
\\=&\mr{w}_{\mu_{\min}(\bs{h}'_{i})}^{2}\Big(\prod_{m=1}^{\#\bs{h}'}\mr{w}_{\mu_{\min}(\bs{h}'_{m})}^{2}\Big)\Big(\prod_{m=1}^{\#\bs{h}''}\mr{w}_{\mu_{\min}(\bs{h}''_{m})}^{2}\Big)\max_{m}\{\mr{w}_{\mu_{\min}(\bs{h}''_{m})}^{2},\mr{w}_{\mu_{\min}(\bs{k}''_{m})}^{2}\}
\\\leq&\Big(\prod_{m=1}^{\#\bs{j}'}\mr{w}_{j'_{m}}\Big)\Big(\prod_{m=1}^{\#\bs{j}''}\mr{w}_{j''_{m}}\Big)
\\=&\prod_{m=1}^{\#\bs{j}}\mr{w}_{j_{m}};
\end{align*}
$\bullet$ when $\max_{m}\{\mr{w}_{\mu_{\min}(\bs{h}_{m})}^{2},\mr{w}_{\mu_{\min}(\bs{k}_{m})}^{2}\}=\max_{m}\{\mr{w}_{\mu_{\min}(\bs{h}'_{m})}^{2},\mr{w}_{\mu_{\min}(\bs{k}'_{m})}^{2}\}$, by \eqref{form428} and \eqref{form429}, one has
\begin{align*}
&\Big(\prod_{m=1}^{\#\bs{h}}\mr{w}_{\mu_{\min}(\bs{h}_{m})}^{2}\Big)\max_{m}\{\mr{w}_{\mu_{\min}(\bs{h}_{m})}^{2},\mr{w}_{\mu_{\min}(\bs{k}_{m})}^{2}\}
\\=&\Big(\prod_{m=1}^{\#\bs{h}'}\mr{w}_{\mu_{\min}(\bs{h}'_{m})}^{2}\Big)\max_{m}\{\mr{w}_{\mu_{\min}(\bs{h}'_{m})}^{2},\mr{w}_{\mu_{\min}(\bs{k}'_{m})}^{2}\}\Big(\prod_{m=1}^{\#\bs{h}''}\mr{w}_{\mu_{\min}(\bs{h}''_{m})}^{2}\Big)\mr{w}_{\mu_{\min}(\bs{h}'_{i})}^{2}
\\\leq&\Big(\prod_{m=1}^{\#\bs{j}'}\mr{w}_{j'_{m}}\Big)\frac{\prod_{m=1}^{\#\bs{j}''}\mr{w}_{j''_{m}}}{\mr{w}_{j''^{*}_{1}}^{2}}\mr{w}_{\mu_{\min}(\bs{h}'_{i})}^{2}
\\\leq&\prod_{m=1}^{\#\bs{j}}\mr{w}_{j_{m}},
\end{align*}
where the last inequality follows from the fact $\mr{w}_{\mu_{\min}(\bs{h}'_{i})}\leq\mr{w}_{j''^{*}_1}$.
This completes the proof of the estimate \eqref{form43}. Hence, one has $\Gamma_{2l}^{(2)}\in\mc{H}_{2l}^{N}$ with $l=l_1+l_2-1$. 
\\\indent Next, we will estimate the coefficient $\bs{c}^{(2)}$. In view of \eqref{form7-10-2}, for any given $J\in\Gamma_{2l}^{(2)}$, if $\#\bs{j}',\#\bs{k}',n_1$ are fixed, then $J',J''$ are determined.  
%
When $\bs{h}=(\bs{h}'_i,\bs{h}',\bs{h}'')$,  there are at most $\#\bs{j}''$ different $a\in\mb{Z}$ such that 
$$\Big|\{\frac{1}{\Omega^{(4)}_{\bs{h}'_{i}}},z_{\bs{j}''}\}\Big|=|\pa_{I_a}\Omega_{\bs{h}'_i}^{(4)}|\frac{|z_{\bs{j}''}|}{(\Omega^{(4)}_{\bs{h}'_{i}})^{2}}\neq0;$$ 
and similarly, when $\bs{h}=(\bs{h}''_i,\bs{h}',\bs{h}'')$, there are at most $\#\bs{j}'$ different $a\in\mb{Z}$ such that 
$$\Big|\{z_{\bs{j}'},\frac{1}{\Omega^{(4)}_{\bs{h}''_{i}}}\}\Big|=|\pa_{I_a}\Omega_{\bs{h}''_i}^{(4)} |\frac{|z_{\bs{j}'}|}{(\Omega^{(4)}_{\bs{h}''_{i}})^{2}}\neq0.$$
Notice that for any $a\in\mb{Z}$, the coefficients of $\pa_{I_a}\Omega_{\bs{h}'_i}^{(4)}$ and $\pa_{I_a}\Omega_{\bs{h}''_i}^{(4)}$ are respectively  less than $\#\bs{h}_{i}'$ and $\#\bs{h}''_{i}$ up to a constant multiply, i.e., $|\pa_{I_a}\Omega_{\bs{h}'_i}^{(4)}|\lesssim\#\bs{h}'_{i}$ and $|\pa_{I_a}\Omega_{\bs{h}''_i}^{(4)}|\lesssim\#\bs{h}''_{i}$. 
Then by $\#\bs{h}'_{i}<\#\bs{j}'$, $\#\bs{h}''_{i}<\#\bs{j}''$ and $\#\bs{j}',\#\bs{j}'',\#\bs{k}',n_1\lesssim l$, one has
\begin{align}
\label{form7-14-2}
\|\bs{c}^{(2)}\|_{\ell^{\infty}}=&\sup_{J\in\Gamma_{2l}^{(2)}}|c_{J}|
\\\notag\lesssim& l^3\sup_{\substack{J'\in\Gamma_{2l_1}\\J''\in\Gamma_{2l_2}}}|\bs{c}_{J'}||\bs{c}_{J''}|((\#\bs
{j}'')\sup_{\substack{\bs{h}'_{i}\in\bs{h}'\\a\in\mb{Z}}}|\pa_{I_a}\Omega^{(4)}_{\bs{h}'_{i}}|+(\#\bs{j}')\sup_{\substack{\bs{h}''_{i}\in\bs{h}''\\a\in\mb{Z}}}|\pa_{I_a}\Omega^{(4)}_{\bs{h}''_{i}}|)
\\\notag\lesssim&l^5\sup_{\substack{J'\in\Gamma_{2l_1}\\J''\in\Gamma_{2l_2}}}|\bs{c}_{J'}||\bs{c}_{J''}|
\\\notag=& l^{5}\|\tilde{\bs{c}}\|_{\ell^{\infty}}\|\tilde{\tilde{\bs{c}}}\|_{\ell^{\infty}}.
\end{align}
\indent 3) Consider the third case that $\Omega^{(6)}_{\bs{h}}$ appears in the poisson bracket. Then $Q_{\Gamma_{2l}^{(3)}}[\bs{c}^{(3)}]$ is of the form
\begin{align*}
&\sum_{\substack{J'\in\Gamma_{2l_1}\\J''\in\Gamma_{2l_2}}}\sum_{i=n_1+1}^{\#\bs{h}'}\tilde{c}_{J'}\tilde{\tilde{c}}_{J''}\frac{z_{\bs{j}'}\{\frac{1}{\Omega^{(6)}_{\bs{h}'_{i}}(I)},z_{\bs{j}''}\}}{\prod\limits_{m=1}^{n_1}\Omega^{(4)}_{\bs{h}'_{m}}\prod\limits_{m\neq i}\Omega^{(6)}_{\bs{h}'_{m}}\prod\limits_{m=1}^{\#\bs{k}'}\Omega^{(6)}_{\bs{k}'_{m}}\prod\limits_{m=1}^{n_2}\Omega^{(4)}_{\bs{h}''_{m}}\prod\limits_{m=n_2+1}^{\#\bs{h}''}\Omega^{(6)}_{\bs{h}''_{m}}\prod\limits_{m=1}^{\#\bs{k}''}\Omega^{(6)}_{\bs{k}''_{m}}}
\\+&\sum_{\substack{J'\in\Gamma_{2l_1}\\J''\in\Gamma_{2l_2}}}\sum_{i=n_1+1}^{\#\bs{h}''}\tilde{c}_{J'}\tilde{\tilde{c}}_{J''}\frac{\{z_{\bs{j}'},\frac{1}{\Omega^{(6)}_{\bs{h}''_{i}}(I)}\}z_{\bs{j}''}}{\prod\limits_{m=1}^{n_1}\Omega^{(4)}_{\bs{h}'_{m}}\prod\limits_{m=n_1+1}^{\bs{h}'}\Omega^{(6)}_{\bs{h}'_{m}}\prod\limits_{m=1}^{\#\bs{k}'}\Omega^{(6)}_{\bs{k}'_{m}}\prod\limits_{m=1}^{n_2}\Omega^{(4)}_{\bs{h}''_{m}}\prod\limits_{m\neq i}\Omega^{(6)}_{\bs{h}''_{m}}\prod\limits_{m=1}^{\#\bs{k}''}\Omega^{(6)}_{\bs{k}''_{m}}}.
\end{align*}
We divide $Q_{\Gamma_{2l}^{(3)}}[\bs{c}^{(3)}]$ into the following two parts:
\\ $\bullet$ One part $Q_{\Gamma_{2l}^{(3,1)}}[\bs{c}^{(3,1)}]$, where $\Gamma_{2l}^{(3,1)}$ and $[\bs{c}^{(3,1)}]$ are determined by applying the derivate of the linear part of $\Omega_{\bs{h}}^{(6)}$. 
Then for any $J\in\Gamma_{2l}^{(3,1)}$, one has
\begin{equation}
\label{form7-10-3}
\begin{array}{l}
\displaystyle \bs{j}=(\bs{j}',\bs{j}''), 
\\\displaystyle \bs{h}=(\bs{h}'_i,\bs{h}',\bs{h}'') \quad \text{or} \quad \bs{h}=(\bs{h}''_i,\bs{h}',\bs{h}''),
\\\displaystyle \bs{k}=(\bs{k}',\bs{k}'')
 \\\displaystyle n=n_1+n_2. 
\end{array}
\end{equation}
It is similar with the case 2).
\\$\bullet$ The other part $Q_{\Gamma_{2l}^{(3,2)}}[\bs{c}^{(3,2)}]$, where $\Gamma_{2l}^{(3,2)}$ and $[\bs{c}^{(3,2)}]$ are determined by applying the derivate of the quadratic part of $\Omega_{\bs{h}}^{(6)}$. Then for any $J\in\Gamma_{2l}^{(3,2)}$, one has 
\begin{equation}
\label{form7-10-4}
\begin{array}{l}
\displaystyle \bs{j}=(\bs{j}',\bs{j}'',j,\bar{j})\quad \text{with} \quad j\in\mb{U}_{2}\times\mb{Z},
\\\displaystyle\bs{h}=(\bs{h}',\bs{h}''),
\\\displaystyle \bs{k}=(\bs{h}'_i,\bs{k}',\bs{k}'') \quad \text{or} \quad \bs{k}=(\bs{h}''_i,\bs{k}',\bs{k}''),
\\\displaystyle n=n_1+n_2.
\end{array}
\end{equation}
%
Hence,
\begin{align*}
\#\bs{j}-2\#\bs{h}-4\#\bs{k}=&\#\bs{j}'+\#\bs{j}''+2-2(\#\bs{h}'+\#\bs{h}'')-4(\#\bs{k}'+\#\bs{k}''+1)
\\=&2l_{1}+2l_{2}-2
\\=&2l,
\end{align*}
 and $6\leq\#\bs{h}_{m},\#\bs{k}_{m}\leq\#\bs{j}'+\#\bs{j}''\leq\#\bs{j}$. 
\\\indent By conditions \eqref{form427}, \eqref{form429}, one has
\begin{align*}
\prod_{m=1}^{\#\bs{h}}\mr{w}_{\mu_{\min}(\bs{h}_{m})}^{2}&=\Big(\prod_{m=1}^{\#\bs{h}'}\mr{w}_{\mu_{\min}(\bs{h}'_{m})}^{2}\Big)\Big(\prod_{m=1}^{\#\bs{h}''}\mr{w}_{\mu_{\min}(\bs{h}''_{m})}^{2}\Big)
\\&\leq\frac{\prod_{m=1}^{\#\bs{j}'}\mr{w}_{j'_{m}}}{\mr{w}_{j'^{*}_{1}}^{2}}\frac{\prod_{m=1}^{\#\bs{j}''}\mr{w}_{j''_{m}}}{\mr{w}_{j''^{*}_{1}}^{2}}
\\&\leq\frac{\prod_{m=1}^{\#\bs{j}}\mr{w}_{j_{m}}}{\mr{w}_{j^{*}_{1}}^{2}},
\end{align*}
and by conditions \eqref{form428}, \eqref{form430}, one has
\begin{align*}
&\Big(\prod_{m=1}^{\#\bs{h}}\mr{w}_{\mu_{\min}(\bs{h}_{m})}^{2}\Big)\max_{m}\{\mr{w}_{\mu_{\min}(\bs{h}_{m})}^{2},\mr{w}_{\mu_{\min}(\bs{k}_{m})}^{2}\}
\\\leq&\Big(\prod_{m=1}^{\#\bs{h}'}\mr{w}_{\mu_{\min}(\bs{h}'_{m})}^{2}\Big)\max_{m}\{\mr{w}_{\mu_{\min}(\bs{h}'_{m})}^{2},\mr{w}_{\mu_{\min}(\bs{k}'_{m})}^{2}\}\Big(\prod_{m=1}^{\#\bs{h}''}\mr{w}_{\mu_{\min}(\bs{h}''_{m})}^{2}\Big)\max_{m}\{\mr{w}_{\mu_{\min}(\bs{h}''_{m})}^{2},\mr{w}_{\mu_{\min}(\bs{k}''_{m})}^{2}\}
\\\leq&\Big(\prod_{m=1}^{\#\bs{j}'}\mr{w}_{j'_{m}}\Big)\Big(\prod_{m=1}^{\#\bs{j}''}\mr{w}_{j''_{m}}\Big)
\\\leq&\prod_{m=1}^{\#\bs{j}}\mr{w}_{j_{m}}.
\end{align*}
This completes the proof of the estimates \eqref{form42}, \eqref{form43}. 
\\\indent In view of \eqref{form7-10-4},  if $\bs{j}\in\mc{R}_{2p}$ and $\#\bs{j}'=2p_1$ are fixed, then $\bs{j}'$, $\bs{j}''$ are determined and  $\#\bs{j}''=2p-2p_1-2$; if $J'$, $J''$ are fixed, then $\bs{h},n$ are determined, and there are at most $(\#\bs{h}'+\#\bs{h}'')$ different $\bs{k}$.  By the fact $\#\bs{h}'+\#\bs{h}''<p$, one has
\begin{align}
\label{form7-12-3}
&\max_{\tilde{\bs{j}}\in\mc{R}_{2p}}\#\{J\in\Gamma_{2l}^{(3,2)}\mid \bs{j}=\tilde{\bs{j}}\}
\\\notag\leq&\sum_{p_1=1}^{p}p\max_{\tilde{\bs{j}'}\in\mc{R}_{2p_1}}\#\{J'\in\Gamma_{2l_1}\mid \bs{j}'=\tilde{\bs{j}'}\}
\max_{\tilde{\bs{j}''}\in\mc{R}_{2(p-p_1-1)}}\#\{J''\in\Gamma_{2l_2}\mid \bs{j}''=\tilde{\bs{j}''}\}
\\\notag\leq&p^{2}(2p_{1})^{2p_1-5}(2p-2p_1-2)^{2p-2p_1-7}
\\\notag\leq&p^{2}(2p)^{2p-12}.
\end{align}
\indent Next, we will estimate the coefficient $\bs{c}^{(3,2)}$. In view of \eqref{form7-10-4}, for any given $J\in\Gamma_{2l}^{(3,2)}$,  if $\#\bs{j}',\#\bs{h}',n_1$ are fixed, $J',J''$ are determined. 
When $\bs{k}=(\bs{h}'_i,\bs{k}',\bs{k}'')$,  there are at most $\#\bs{j}''$ different $a\in\mb{Z}$ such that 
$$\Big|\{\frac{1}{\Omega^{(6)}_{\bs{h}'_{i}}},z_{\bs{j}''}\}\Big|=|\pa_{I_a}\Omega_{\bs{h}'_i}^{(6)}|\frac{|z_{\bs{j}''}|}{(\Omega^{(6)}_{\bs{h}'_{i}})^{2}}\neq0;$$ 
and similarly, when $\bs{k}=(\bs{h}''_i,\bs{k}',\bs{k}'')$, there are at most $\#\bs{j}'$ different $a\in\mb{Z}$ such that 
$$\Big|\{z_{\bs{j}'},\frac{1}{\Omega^{(6)}_{\bs{h}''_{i}}}\}\Big|=|\pa_{I_a}\Omega_{\bs{h}''_i}^{(6)}|\frac{|z_{\bs{j}'}|}{(\Omega^{(6)}_{\bs{h}''_{i}})^{2}}\neq0.$$
Notice that for any $a\in\mb{Z}$, the coefficients of $\pa_{I_a}\Omega_{\bs{h}'_i}^{(6)}$ and $\pa_{I_a}\Omega_{\bs{h}''_i}^{(6)}$ are respectively  less than $\#\bs{h}_{i}'$ and $\#\bs{h}''_{i}$ up to a constant multiply. 
Then by $\#\bs{h}'_{i}<\#\bs{j}'$, $\#\bs{h}''_{i}<\#\bs{j}''$ and $\#\bs{j}',\#\bs{j}'',\#\bs{h}',n_1\lesssim l$, one has
\begin{align}
\label{form7-14-3}
\|\bs{c}^{(3,2)}\|_{\ell^{\infty}}&=\sup_{J\in\Gamma_{2l}^{(3,2)}}|c_{J}|
\\\notag&\lesssim l^{3}\sup_{\substack{J'\in\Gamma_{2l_1}\\J''\in\Gamma_{2l_2}}}(\#\bs{j}')(\#\bs{j}'')|\bs{c}_{J'}||\bs{c}_{J''}|
\\\notag&\lesssim l^{5}\|\tilde{\bs{c}}\|_{\ell^{\infty}}\|\tilde{\tilde{\bs{c}}}\|_{\ell^{\infty}}.
\end{align}
\indent 4) Consider the fourth case that $\Omega^{(6)}_{\bs{k}}$ appears in the poisson bracket. Then 
\\$Q_{\Gamma_{2l}^{(4)}}[\bs{c}^{(4)}]$ is of the form
\begin{align*}
&\sum_{\substack{J'\in\Gamma_{2l_1}\\J''\in\Gamma_{2l_2}}}\sum_{i=2}^{\#\bs{k}'}\tilde{c}_{J'}\tilde{\tilde{c}}_{J''}\frac{z_{\bs{j}'}\{\frac{1}{\Omega^{(6)}_{\bs{k}'_{i}}(I)},z_{\bs{j}''}\}}{\prod\limits_{m=1}^{n_1}\Omega^{(4)}_{\bs{h}'_{m}}\prod\limits_{m=n_1+1}^{\#\bs{h}'}\Omega^{(6)}_{\bs{h}'_{m}}\prod\limits_{m\neq i}\Omega^{(6)}_{\bs{k}'_{m}}\prod\limits_{m=1}^{n_2}\Omega^{(4)}_{\bs{h}''_{m}}\prod\limits_{m=n_2+1}^{\#\bs{h}''}\Omega^{(6)}_{\bs{h}''_{m}}\prod\limits_{m=1}^{\#\bs{k}''}\Omega^{(6)}_{\bs{k}''_{m}}}
\\+&\sum_{\substack{J'\in\Gamma_{2l_1}\\J''\in\Gamma_{2l_2}}}\sum_{i=2}^{\#\bs{k}''}\tilde{c}_{J'}\tilde{\tilde{c}}_{J''}\frac{\{z_{\bs{j}'},\frac{1}{\Omega^{(6)}_{\bs{k}''_{i}}(I)}\}z_{\bs{j}''}}{\prod\limits_{m=1}^{n_1}\Omega^{(4)}_{\bs{h}'_{m}}\prod\limits_{m=n_1+1}^{\#\bs{h}'}\Omega^{(6)}_{\bs{h}'_{m}}\prod\limits_{m=1}^{\#\bs{k}'}\Omega^{(6)}_{\bs{k}'_{m}}\prod\limits_{m=1}^{n_2}\Omega^{(4)}_{\bs{h}''_{m}}\prod\limits_{m=n_2+1}^{\#\bs{h}''}\Omega^{(6)}_{\bs{h}''_{m}}\prod\limits_{m\neq i}\Omega^{(6)}_{\bs{k}''_{m}}}.
\end{align*}
We divide $Q_{\Gamma_{2l}^{(4)}}[\bs{c}^{(4)}]$ into the following two parts:
\\ $\bullet$ One part $Q_{\Gamma_{2l}^{(4,1)}}[\bs{c}^{(4,1)}]$, where $\Gamma_{2l}^{(4,1)}$ and $[\bs{c}^{(4,1)}]$ are determined by applying the derivate of the linear part of $\Omega_{\bs{k}}^{(6)}$. 
Then for any $J\in\Gamma_{2l}^{(4,1)}$, one has
\begin{equation}
\label{form7-10-5}
\begin{array}{l}
\displaystyle \bs{j}=(\bs{j}',\bs{j}''), 
\\\displaystyle \bs{h}=(\bs{k}'_i,\bs{h}',\bs{h}'') \quad \text{or} \quad \bs{h}=(\bs{k}''_i,\bs{h}',\bs{h}''),
\\\displaystyle \bs{k}=(\bs{k}',\bs{k}'')
 \\\displaystyle n=n_1+n_2. 
\end{array}
\end{equation}
\\$\bullet$ The other part $Q_{\Gamma_{2l}^{(4,2)}}[\bs{c}^{(4,2)}]$, where $\Gamma_{2l}^{(4,2)}$ and $[\bs{c}^{(4,2)}]$ are determined by applying the derivate of the quadratic part of $\Omega_{\bs{k}}^{(6)}$. Then for any $J\in\Gamma_{2l}^{(4,2)}$, one has 
\begin{equation}
\label{form7-10-6}
\begin{array}{l}
\displaystyle \bs{j}=(\bs{j}',\bs{j}'',j,\bar{j})\quad \text{with} \quad j\in\mb{U}_{2}\times\mb{Z},
\\\displaystyle\bs{h}=(\bs{h}',\bs{h}''),
\\\displaystyle \bs{k}=(\bs{k}'_i,\bs{k}',\bs{k}'') \quad \text{or} \quad \bs{k}=(\bs{k}''_i,\bs{k}',\bs{k}''),
\\\displaystyle n=n_1+n_2.
\end{array}
\end{equation}
The proof is parallel to the case 3) with $\bs{k}'_i$ and $\bs{k}''_i$ instead of $\bs{h}'_i$ and $\bs{h}''_i$, respectively.
Then the details are omitted.
\\\indent To sum up, letting $\Gamma_{2l}=\Gamma_{2l}^{(1)}\cup\Gamma_{2l}^{(2)}\cup\Gamma_{2l}^{(3)}\cup\Gamma_{2l}^{(4)}$, then \eqref{form42} and \eqref{form43} hold for any $J\in\Gamma_{2l}$. By \eqref{form7-12-1}, \eqref{form7-12-2} and \eqref{form7-12-3}, we conclude that 
$\max_{\tilde{\bs{j}}\in\mc{R}_{2p}}\#\{(\bs{j},\bs{h},\bs{k},n)\in\Gamma_{2l}\mid \bs{j}=\tilde{\bs{j}}\}
<(2p)^{2p-5}$.
Hence, $\Gamma_{2l}\in\mc{H}^{N}_{2l}$ with $l=l_1+l_2-1$.
Moreover, the coefficient estimate \eqref{form424} follows from \eqref{form7-14-1}, \eqref{form7-14-2} and \eqref{form7-14-3}.
\begin{remark}
\label{re7-12}
By the above proof, we can draw the conclusion that for any $(\bs{j},\bs{h},\bs{k},n)\in\Gamma_{2l}$, one has
\begin{equation}
\label{form425}
\#\bs{h}\leq\max_{J'\in\Gamma_{2l_1}}\{\#\bs{h}'\}+\max_{J''\in\Gamma_{2l_2}}\{\#\bs{h}''\}+1,
\end{equation}
\begin{equation}
\label{form425'}
\#\bs{k}\leq\max_{J'\in\Gamma_{2l_1}}\{\#\bs{k}'\}+\max_{J''\in\Gamma_{2l_2}}\{\#\bs{k}''\}+1,
\end{equation}
\begin{equation}
\label{form426}
\#\bs{h}+\#\bs{k}\leq\max_{J'\in\Gamma_{2l_1}}\{\#\bs{h}'+\#\bs{k}'\}+\max_{J''\in\Gamma_{2l_2}}\{\#\bs{h}''+\#\bs{k}''\}+1,
\end{equation}
with $J':=(\bs{j}',\bs{h}',\bs{k}',n_1)$ and $J'':=(\bs{j}'',\bs{h}'',\bs{k}'',n_2)$.
\end{remark}

\subsection{Proof of \Cref{le41}}
\label{subsec32}
Write $Q_{\Gamma_{2l}}[\bs{c}](z)=\sum_{(p,q,q')\in\mc{F}_{2l}}Q^{(p,q,q')}[\bs{c}](z)$ with
$$Q^{(p,q,q')}[\bs{c}](z)=\sum_{J:=(\bs{j},\bs{h},\bs{k},n)\in\Gamma_{2l}^{(p,q,q')}}c_{J}f_{J}(I)z_{\bs{j}},$$
where
\begin{equation}
\label{form47}
f_{J}(I)=\frac{1}{\prod\limits_{m=1}^{n}\Omega^{(4)}_{\bs{h}_{m}}(I)\prod\limits_{m=n+1}^{q}\Omega^{(6)}_{\bs{h}_{m}}(I)\prod\limits_{m=1}^{q'}\Omega^{(6)}_{\bs{k}_{m}}(I)}.
\end{equation}
Notice that $\mc{F}_{2l}\subseteq\mb{N}^{3}$ is a finite set with $\#\mc{F}_{2l}\lesssim l^{3}$, and hence
\begin{equation}
\label{form48}
|X_{Q_{\Gamma_{2l}}[\bs{c}]}(z)|_{\mr{w}}\lesssim l^{3}\max_{(p,q,q')\in\mc{F}_{2l}}|X_{Q^{(p,q,q')}[\bs{c}]}(z)|_{\mr{w}}.
\end{equation}
For any given $j=(\delta,a)$, let
\begin{equation}
\label{form49}
\big|(X_{Q^{(p,q,q')}})_{j}\big|=\left|\frac{\pa Q^{(p,q,q')}}{\pa z_{j}}\right|
\leq X^{(1)}_{j}+X^{(2)}_{j}
\end{equation}
with
\begin{align*}
X^{(1)}_{j}&:=\sum_{J\in\Gamma_{2l}^{(p,q,q')}}|c_{J}|\left|\pa_{I_a} f_{J}(I)\right||z_{\bar{j}}||z_{\bs{j}}|,
\\X^{(2)}_{j}&:=\sum_{J\in\Gamma_{2l}^{(p,q,q')}}|c_{J}||f_{J}(I)|\left|\frac{\pa z_{\bs{j}}}{\pa z_{j}}\right|.
\end{align*}
Then one has
\begin{equation}
\label{form410}
|X_{Q^{(p,q,q')}}(z)|_{\mr{w}}\leq|X^{(1)}(z)|_{\mr{w}}+|X^{(2)}(z)|_{\mr{w}}.
\end{equation}
\indent 1) Estimation of $|X^{(1)}(z)|_{\mr{w}}$. Obviously, one has
\begin{equation}
\label{form411}
|X^{(1)}(z)|_{\mr{w}}=\big(\sum_{j\in\mb{U}_2\times\mb{Z}}\mr{w}_j^2|X_{j}^{(1)}|^{2}\big)^{\frac{1}{2}}\leq S|z|_{\mr{w}}
\end{equation}
with
\begin{equation}
\label{form412}
S:=\sup_{a\in\mb{Z}}\big(\sum_{J\in\Gamma_{2l}^{(p,q,q')}}|c_{J}||\pa_{I_{a}}f_{J}(I)||z_{\bs{j}}|\big).
\end{equation}
Now we are going to estimate $S$. One has
\begin{equation}
\label{form413}
\pa_{I_{a}}f_{J}(I)=-\big(\sum_{m=1}^{n}\frac{\pa_{I_{a}}\Omega^{(4)}_{\bs{h}_{m}}}{\Omega^{(4)}_{\bs{h}_{m}}}+\sum_{m=n+1}^{q}\frac{\pa_{I_{a}}\Omega^{(6)}_{\bs{h}_{m}}}{\Omega^{(6)}_{\bs{h}_{m}}}+\sum_{m=1}^{q'}\frac{\pa_{I_{a}}\Omega^{(6)}_{\bs{k}_{m}}}{\Omega^{(6)}_{\bs{k}_{m}}}\big)f_{J}(I).
\end{equation}
Notice that $\#\bs{h}_{m},\#\bs{k}_{m}\leq 2p$, and then by the assumptions of small divisors, one has
\begin{equation}
\label{form414}
|\pa_{I_{a}}\Omega^{(4)}_{\bs{h}_{m}}|, |\pa_{I_{a}}\Omega^{(6)}_{\bs{h}_{m}}|, |\pa_{I_{a}}\Omega^{(6)}_{\bs{k}_{m}}|\lesssim p.
\end{equation}
By small divisor conditions \eqref{form7-6-1} and \eqref{form7-6-2}, one has
\begin{equation}
\label{form415}
|\Omega^{(4)}_{\bs{h}_{m}}|,|\Omega^{(6)}_{\bs{h}_{m}}|\geq\gamma(2N)^{-12p}|z|_{\mr{w}}^{2}\mr{w}_{\mu_{\min}(\bs{h}_m)}^{-2},
\end{equation}
\begin{equation}
\label{form416}
|\Omega^{(6)}_{\bs{k}_{m}}|\geq\gamma(2N)^{-12p}|z|_{\mr{w}}^{2}\mr{w}_{\mu_{\min}(\bs{k}_m)}^{-2},
\end{equation}
\begin{equation}
\label{form416'}
|\Omega^{(6)}_{\bs{k}_{m}}|\geq\gamma^{2}(2N)^{-12p}|z|_{\mr{w}}^{4}.
\end{equation}
Hence, by \eqref{form414}--\eqref{form416} and the fact $q+q'<p$, one has
\begin{align}
\label{form417'}
&\Big|\sum_{m=1}^{n}\frac{\pa_{I_{a}}\Omega^{(4)}_{\bs{h}_{m}}}{\Omega^{(4)}_{\bs{h}_{m}}}+\sum_{m=n+1}^{q}\frac{\pa_{I_{a}}\Omega^{(6)}_{\bs{h}_{m}}}{\Omega^{(6)}_{\bs{h}_{m}}}+\sum_{m=1}^{q'}\frac{\pa_{I_{a}}\Omega^{(6)}_{\bs{k}_{m}}}{\Omega^{(6)}_{\bs{k}_{m}}}\Big|
\\\notag\lesssim&p^{2}\max_{m}\{\mr{w}_{\mu_{\min}(\bs{h}_{m})}^{2},\mr{w}_{\mu_{\min}(\bs{k}_{m})}^{2}\}\frac{(2N)^{12p}}{\gamma|z|_{\mr{w}}^{2}};
\end{align}
by \eqref{form415} and \eqref{form416'}, one has
\begin{equation}
\label{form417}
|f_{J}(I)|\leq\Big(\prod_{m=1}^{q}\mr{w}_{\mu_{\min}(\bs{h}_{m})}^{2}\Big)\Big(\frac{(2N)^{12p}}{\gamma|z|_{\mr{w}}^{2}}\Big)^{q}\Big(\frac{(2N)^{12p}}{\gamma^{2}|z|_{\mr{w}}^{4}}\Big)^{q'}.
\end{equation}
By \eqref{form413}, \eqref{form417'},\eqref{form417} and the condition \eqref{form43}, one has
\begin{align}
\label{form418}
|\pa_{I_{a}}f_{J}(I)|&\lesssim p^{2}\frac{(2N)^{12p(q+q')+12p}}{\gamma^{q+2q'+1}|z|_{\mr{w}}^{2q+4q'+2}}\max_{m}\{\mr{w}_{\mu_{\min}(\bs{h}_{m})}^{2},\mr{w}_{\mu_{\min}(\bs{k}_{m})}^{2}\}\prod_{m=1}^{q}\mr{w}_{\mu_{\min}(\bs{h}_{m})}^{2}
\\\notag&\leq p^{2}\frac{(2N)^{12p(q+q')+12p}}{\gamma^{q+2q'+1}|z|_{\mr{w}}^{2q+4q'+2}}\prod_{m=1}^{2p}\mr{w}_{j_{m}}.
\end{align}
%
%
Hence, by  \eqref{form412}, \eqref{form418} and the condition \eqref{def41-2}  in \Cref{def21}, one has
\begin{align}
\label{form7-14-1'''}
S&\lesssim p^{2}\frac{(2N)^{12p(q+q')+12p}}{\gamma^{q+2q'+1}|z|_{\mr{w}}^{2q+4q'+2}}\sum_{J\in\Gamma_{2l}^{(p,q,q')}}|c_{J}|\Big(\prod_{m=1}^{2p}\mr{w}_{j_{m}}\Big)|z_{\bs{j}}|
\\\notag&\leq (2p)^{2p-5}p^{2}\|\bs{c}\|_{\ell^{\infty}}\frac{(2N)^{12p(q+q')+12p}}{\gamma^{q+2q'+1}|z|_{\mr{w}}^{2q+4q'+2}}\sum_{\substack{\bs{j}\in\mc{R}_{2p}\\(Irr(\bs{j}))_{1}^{*}\leq N}}\Big(\prod_{m=1}^{2p}\mr{w}_{j_{m}}\Big)|z_{\bs{j}}|.
\end{align}
By the definition of norm and the fact $4N+2\leq(2N)^{3}$, one has
\begin{equation}
\label{form7-14-2'''}
\sum_{|j|\leq N}\mr{w}_{j}|z_j|\leq\sqrt{4N+2}|z|_{\mr{w}}\leq(2N)^{\frac{3}{2}}|z|_{\mr{w}},
\end{equation}
and thus
\begin{equation}
\label{form419'}
\sum_{\substack{\bs{j}\in\mc{R}_{2p}\\(Irr(\bs{j}))_{1}^{*}\leq N}}\Big(\prod_{m=1}^{2p}\mr{w}_{j_{m}}\Big)|z_{\bs{j}}|\leq(2N)^{3p}|z|_{\mr{w}}^{2p}.
\end{equation}
Then by \eqref{form7-14-1'''}, \eqref{form419'} and the fact $p-q-2q'=l$, one has
\begin{align}
\label{form419}
S\lesssim(2p)^{2p-3}\|\bs{c}\|_{\ell^{\infty}}\frac{(2N)^{12p(q+q')+15p}}{\gamma^{q+2q'+1}}|z|_{\mr{w}}^{2l-2}.
\end{align}
By \eqref{form411} and \eqref{form419}, one has
\begin{equation}
\label{form420}
|X^{(1)}(z)|_{\mr{w}}\lesssim (2p)^{2p-3}\|\bs{c}\|_{\ell^{\infty}}\frac{(2N)^{12p(q+q')+15p}}{\gamma^{q+2q'+1}}|z|_{\mr{w}}^{2l-1}.
\end{equation}
\indent 2) Estimation of $|X^{(2)}(z)|_{\mr{w}}$. By \eqref{form417} and  the condition \eqref{form42}, one has
\begin{equation}
\label{form421}
|f_{J}(I)|\leq\frac{\prod_{m=1}^{2p}\mr{w}_{j_{m}}}{\mr{w}_{j^{*}_{1}}^{2}}\big(\frac{(2N)^{12p}}{\gamma|z|_{\mr{w}}^{2}}\big)^{q}\big(\frac{(2N)^{12p}}{\gamma^{2}|z|_{\mr{w}}^{4}}\big)^{q'}.
\end{equation}
Then by \eqref{form421} and the condition \eqref{def41-2}  in \Cref{def21}, one has
\begin{align}
\label{form340'}
|X^{(2)}(z)|_{\mr{w}}&\leq\Big(\sum_{j_{1}} \mr{w}_{j_1}^{2}\big|2p\sum_{J\in\Gamma_{2l}^{(p,q,q')}}|c_{J}||f_{J}(I)||z_{j_{2}}|\cdots|z_{j_{2p_{i}}}|\Big|^{2}\Big)^{\frac{1}{2}}
\\\notag&\leq \Big(\sum_{j_{1}}\Big|2p\sum_{J\in\Gamma_{2l}^{(p,q,q')}}|c_{J}|\frac{(2N)^{12p(q+q')}}{\gamma^{q+2q'}|z|_{\mr{w}}^{2q+4q'}}\prod_{m=2}^{2p}(\mr{w}_{j_{m}}|z_{j_{m}}|)\Big|^{2}\Big)^{\frac{1}{2}}
\\\notag&\lesssim\|\bs{c}\|_{\ell^{\infty}}\Big(\sum_{j_{1}}\Big|(2p)^{2p-4}\sum_{\substack{j_{2},\cdots,j_{2p}\\\bs{j}\in\mc{R}_{2p},(Irr(\bs{j}))_{1}^{*}\leq N}}\frac{(2N)^{12p(q+q')}}{\gamma^{q+2q'}|z|_{\mr{w}}^{2q+4q'}}\prod_{m=2}^{2p}(\mr{w}_{j_{m}}|z_{j_{m}}|)\Big|^{2}\Big)^{\frac{1}{2}}
\\\notag&=(2p)^{2p-4}\|\bs{c}\|_{\ell^{\infty}}\frac{(2N)^{12p(q+q')}}{\gamma^{q+2q'}|z|_{\mr{w}}^{2q+4q'}}\Big(\sum_{j_1}\Big|\sum_{\substack{j_{2},\cdots,j_{2p}\\\bs{j}\in\mc{R}_{2p},(Irr(\bs{j}))_{1}^{*}\leq N}}\prod_{m=2}^{2p}(\mr{w}_{j_{m}}|z_{j_{m}}|)\Big|^{2}\Big)^{\frac{1}{2}}
\end{align}
Divide the second sum above into the two parts according to $\{j_2,\cdots,j_{2p}\}\ni\bar{j}_1$ or not. 
%
%
When $\{j_2,\cdots,j_{2p}\}\ni\bar{j}_1$,  by \eqref{form7-14-2'''}, one has
\begin{align}
\label{form7-15-1}
&\Big(\sum_{j_1}\Big|\sum_{\substack{\{j_{2},\cdots,j_{2p}\}\ni\bar{j}_1\\\bs{j}\in\mc{R}_{2p},(Irr(\bs{j}))_{1}^{*}\leq N}}\prod_{m=2}^{2p}(\mr{w}_{j_{m}}|z_{j_{m}}|)\Big|^{2}\Big)^{\frac{1}{2}}
\\\notag\leq&\Big(\sum_{j_1}\Big|2p\mr{w}_{\bar{j}_{1}}|z_{\bar{j}_1}|\sum_{\substack{\bs{j}'=(j_{3},\cdots,j_{2p})\in\mc{R}_{2p-2}\\(Irr(\bs{j}'))_{1}^{*}\leq N}}\prod_{m=3}^{2p}(\mr{w}_{j_{m}}|z_{j_{m}}|)\Big|^{2}\Big)^{\frac{1}{2}}
\\\notag\leq&2p\Big(\sum_{j_1}\mr{w}_{j_{1}}^{2}|z_{j_1}|^{2}\Big)^{\frac{1}{2}}\sum_{\substack{\bs{j}'=(j_{3},\cdots,j_{2p})\in\mc{R}_{2p-2}\\(Irr(\bs{j}'))_{1}^{*}\leq N}}\prod_{m=3}^{2p}(\mr{w}_{j_{m}}|z_{j_{m}}|)
\\\notag\leq&2p(2N)^{3(p-1)}|z|_{\mr{w}}^{2p-1}.
\end{align}
When $\{j_2,\cdots,j_{2p}\}\not\ni\bar{j}_1$, by $\bs{j}\in\mc{R}_{2p}$ and $(Irr(\bs{j}))_{1}^{*}\leq N$, one has $|j_1|\leq N$. Then by \eqref{form7-14-2'''}   and the fact $4N+2\leq(2N)^{3}$, one has
\begin{align}
\label{form7-15-2}
&\Big(\sum_{|j_1|\leq N}\Big|\sum_{\substack{j_{2},\cdots,j_{2p}\neq \bar{j}_1\\\bs{j}\in\mc{R}_{2p},(Irr(\bs{j}))_{1}^{*}\leq N}}\prod_{m=2}^{2p}(\mr{w}_{j_{m}}|z_{j_{m}}|)\Big|^{2}\Big)^{\frac{1}{2}}
\\\notag\leq&(4N+2)^{\frac{1}{2}}\sup_{|j_1|\leq N}\sum_{\substack{j_{2},\cdots,j_{2p}\neq \bar{j}_1\\\bs{j}\in\mc{R}_{2p},(Irr(\bs{j}))_{1}^{*}\leq N}}\prod_{m=2}^{2p}(\mr{w}_{j_{m}}|z_{j_{m}}|)
\\\notag\leq&(2N)^{3p}|z|_{\mr{w}}^{2p-1}.
\end{align}
Combing \eqref{form7-15-1} and \eqref{form7-15-2}, one has
\begin{align}
\label{form422'}
&\big(\sum_{j_1}\big|\sum_{\substack{j_{2},\cdots,j_{2p}\\\bs{j}\in\mc{R}_{2p},(Irr(\bs{j}))_{1}^{*}\leq N}}\prod_{m=2}^{2p}(\mr{w}_{j_{m}}|z_{j_{m}}|)\big|^{2}\big)^{\frac{1}{2}}
\\\notag\leq&2p(2N)^{3(p-1)}|z|_{\mr{w}}^{2p-1}+(2N)^{3p}|z|_{\mr{w}}^{2p-1}
\\\notag\leq&2p(2N)^{3p}|z|_{\mr{w}}^{2p-1}.
\end{align}
Hence, by \eqref{form340'}, \eqref{form422'} and the fact $p-q-2q'=l$, one has
\begin{align}
\label{form422}
|X^{(2)}(z)|_{\mr{w}}\lesssim&(2p)^{2p-4}\|\bs{c}\|_{\ell^{\infty}}\frac{(2N)^{12p(q+q')}}{\gamma^{q+2q'}|z|_{\mr{w}}^{2q+4q'}}(2p)(2N)^{3p}|z|_{\mr{w}}^{2p-1}.
\\\notag=&(2p)^{2p-3}\|\bs{c}\|_{\ell^{\infty}}\frac{(2N)^{12p(q+q')+3p}}{\gamma^{q+2q'}}|z|_{\mr{w}}^{2l-1}.
\end{align}
To sum up, by \eqref{form48}, \eqref{form410},  \eqref{form420} and \eqref{form422}, we can draw the conclusion.

\section{Resonant normal form and small divisor conditions}
\label{sec4}
In this section, we prove resonant normal form theorem for nonlinear Schr\"{o}dinger equation \eqref{form11} in Gevrey spaces. Then we give the small divisor conditions, which are well preserved under perturbations. 
\subsection{Resonant normal form}
\label{subsec41}
Expand $\varphi(|u|^{2})=\sum_{l\geq0}\frac{\varphi^{(l)}(0)}{l!}|u|^{2l}$ and rewrite $u(x)=\sum_{a\in\mb{Z}}z_{a}e^{{\rm i}ax}$ with a little abuse of the notation. Then the equation \eqref{form11} is equivalent to
\begin{equation}
\label{form211}
\dot{z}_{a}=-{\rm i}\frac{\pa H(z)}{\pa\bar{z}_{a}}
\end{equation}
with the Hamiltonian function
\begin{equation}
\label{form212}
H=H_{0}+\sum_{l=2}^{r}P_{2l}+R_{\geq2r+2}
=\sum_{a\in\mb{Z}_{*}}a^{2}|z_a|^{2}+\sum_{l=2}^{r}\sum_{\bs{j}\in\mc{M}_{2l}}\frac{\varphi^{(l-1)}(0)}{l!}z_{\bs{j}}+R_{\geq2r+2},
\end{equation}
where $R_{\geq2r+2}$ is the remainder of the Taylor expansion.
Notice that in a neighborhood of the origin, there exists a big enough positive constant $C_0\geq1$ such that the coefficient of $P_{2l}$ satisfies the estimate
\begin{equation}
\label{form213}
\left|\frac{\varphi^{(l-1)}(0)}{l!}\right|<C_0^{l-1}
\end{equation}
and the remainder $R_{\geq2r+2}$ satisfies the estimate
\begin{equation}
\label{form214}
\|X_{R_{\geq2r+2}}(z)\|_{\rho,\theta}\leq C_0^{r}\|z\|_{\rho,\theta}^{2r+1}.
\end{equation}
Recall the notation $a\lesssim b$, which mean that $a\leq Cb$ with the positive constant $C$ independent of parameter $r$. In particular, if the constant $C$ depends on other parameters, such as $\rho$, then write $a\lesssim_{\rho}b$.
\begin{theorem}
\label{th21}
Consider the Hamiltonian function \eqref{form212} and fix $\rho>0$, $\theta\in(0,1)$. There exists $\varepsilon_{0}>0$ such that for all positive integer $r\geq3$ and any $0<\varepsilon\leq\varepsilon_{0}$ satisfying $\varepsilon\lesssim_{\rho,\theta}r^{-\frac{3}{2}}$, there exists a canonical transformation $\phi^{(1)}:B_{\rho,\theta}(3\varepsilon)\longmapsto B_{\rho,\theta}(4\varepsilon)$ such that
\begin{equation}
\label{form215}
H\circ\phi^{(1)}=H_{0}+Z_{4}+Z_{6}+\sum_{l=3}^{r}K_{2l}+\ms{R}+R_{\geq2r+2}\circ\phi^{(1)},
\end{equation}
where
\begin{enumerate}[(i)]
 \item  the transformation satisfies the estimate
    \begin{equation}
    \label{form216}
    \|z-\phi^{(1)}(z)\|_{\rho,\theta}\lesssim_{\rho,\theta}\|z\|_{\rho,\theta}^{3},\quad\|(d\phi^{(1)}(z))^{-1}\|_{\ms{P}_{\rho,\theta}\mapsto\ms{P}_{\rho,\theta}}\lesssim1,
    \end{equation}
    and the same estimate is fulfilled by the inverse transformation;
\item $Z_{4}$ and $Z_{6}$ are integrable polynomials of the form
    \begin{align}
    \label{form217}
    Z_{4}(I)=&\varphi'(0)(\sum_{a\in\mb{Z}}I_{a})^{2}-\frac{\varphi'(0)}{2}\sum_{a\in\mb{Z}}I_{a}^{2},
    \\\label{form218}
    Z_{6}(I)=&-\frac{(\varphi'(0))^{2}}{2}\sum_{\substack{a,a'\in\mb{Z}\\a\neq a'}}\frac{I_{a}^{2}I_{a'}}{(a-a')^{2}}+\frac{\varphi''(0)}{6}\big(6(\sum_{a\in\mb{Z}}I_{a})^{3}
    -9(\sum_{a\in\mb{Z}}I_{a})(\sum_{a\in\mb{Z}}I_{a}^{2})
    +4\sum_{a\in\mb{Z}}I_{a}^{3}\big)
    \end{align}
    with the action $I_{a}=z_{a}\bar{z}_{a}$;
\item \label{th21-3} $K_{2l}$ is resonant polynomial of order $2l$ satisfying
    \begin{equation}
    \label{form219}
     K_{6}(z)=\sum_{\bs{j}\in\mc{R}_{6}\backslash\mc{I}_{6}}c_{\bs{j}}^{(6)}z_{\bs{j}}
     \quad\text{and}\quad K_{2l}(z)=\sum_{\bs{j}\in\mc{R}_{2l}}c_{\bs{j}}^{(2l)}z_{\bs{j}}\;\text{for}\;l\geq4
    \end{equation}
    with the coefficient estimate $\|c^{(2l)}\|_{\ell^{\infty}}:=\sup_{\bs{j}\in\mc{R}_{2l}}|c_{\bs{j}}^{(2l)}|\leq(8(l-1)^{3})^{l-2}C_0^{l-1}$;
\item\label{th21-4} the remainder $\ms{R}$ satisfies the estimate
    \begin{equation}
    \label{form220}
    \|X_{\ms{R}}(z)\|_{\rho,\theta}\lesssim (r^{3}C_1)^{r}\|z\|_{\rho,\theta}^{2r+1}
     \end{equation}
     with the positive constant $C_1$ depending on $C_0$, $\rho$ and $\theta$.
\end{enumerate}
\end{theorem}
\begin{proof}
The proof is standard and the intergable terms \eqref{form217}, \eqref{form218} have been calculated in Theorem 7.1 of \cite{BFG20b}. But we need more concrete coefficient estimates in \eqref{th21-3} and \eqref{th21-4} with the dependence on the order $l$ and the number of iteration $r$, respectively.
\\\indent Notice that for any homogenous polynomials $P_{2l_1}[c'], P_{2l_2}[c'']$ of order $2l_{1}$ and $2l_{2}$,  there exists a homogeneous polynomial $P[c]$ of order $2(l_{1}+l_{2}-1)$ such that $P[c]=\{P_{2l_1}[c'], P_{2l_2}[c'']\}$ with the coefficient estimate
\begin{equation}
\label{form221}
\|c\|_{\ell^{\infty}}\leq8l_{1}l_{2}\|c'\|_{\ell^{\infty}}\|c''\|_{\ell^{\infty}}.
\end{equation}
We prove the coefficient estimates by induction. Assume the homogenous polynomial $P'_{2k}[c^{(2k)}]$ of order $2k$ satisfies the estimate $\|c^{(2k)}\|_{\ell^{\infty}}\leq(8(k-1)^{3})^{k-2}C_0^{k-1}$. Notice that
$$K_{2l}{[c^{(2l)}]}=P_{2l}+\sum_{k=2}^{l-1}\{P'_{2k}[c^{(2k)}],P'_{2(l-k+1)}[c^{(2(l-k+1))}]\}.$$ Then by \eqref{form213} and \eqref{form221}, one has
\begin{align*}
\|c^{(2l)}\|_{\ell^{\infty}}\leq& C_0^{l-1}+\sum_{k=2}^{l-1}8k(l-k+1)(8(k-1)^{3})^{k-2}C_0^{k-1}(8(l-k)^{3})^{l-k-1}C_0^{l-k}
\\\leq&C_0^{l-1}+\sum_{k=2}^{l-1}8(l-1)^{2}(8(l-1)^{3})^{l-3}C_0^{l-1}
\\\leq&(8(l-1)^{3})^{l-2}C_0^{l-1}.
\end{align*}
\indent The remainder term $\ms{R}$ consists of the terms $R'=\int_{0}^{1}g(t)R_{2k}\circ\Phi^{(t)}dt$ and $R'\circ\phi$, where $\|g(t)\|_{L^{\infty}}\leq1$, $R_{2k}$  is a homogenous polynomial of order $2k$ with $r+1\leq k\leq2r$, $\Phi^{(t)}$ is a transformation near identity for any $0\leq t\leq1$, and $\phi$ is a composition of some transformations near identity.
Notice that
$$\|X_{R'}(z)\|_{\rho,\theta}\leq\|g(t)\|_{L^{\infty}}\|X_{R_{2k}\circ\Phi^{(t)}}(z)\|_{\rho,\theta}
\lesssim\sup_{0\leq t\leq1}\|X_{R_{2k}}(\Phi^{(t)}(z))\|_{\rho,\theta}$$
and
$\|X_{R'\circ\phi}(z)\|_{\rho,\theta}\lesssim\|X_{R'}(\phi(z))\|_{\rho,\theta}$.
By \Cref{le02} in Appendix, for any homogenous polynomial $R_{2k}[c]$ with the coefficients estimate $\|c\|_{\ell^{\infty}}<+\infty$, there exists a positive constant $C$ depending on $\rho,\theta$ such that
\begin{align*}
\|X_{R_{2k}}(z)\|_{\rho,\theta}\leq C^{2k-1}\|c\|_{\ell^{\infty}}\|z\|_{\rho,\theta}\|z\|_{2^{\theta-1}\rho,\theta}^{2k-2}\leq\|c\|_{\ell^{\infty}}(C\|z\|_{\rho,\theta})^{2k-1}.
\end{align*}
Hence, when $\varepsilon\lesssim_{\rho,\theta}r^{-\frac{3}{2}}$, there exists a constant $C_1$ depending on $C_0,C,\rho,\theta$ such that
$$\|X_{\ms{R}}(z)\|_{\theta,\rho}\lesssim\sum_{k=r+1}^{2r}(8(k-1)^{3})^{k-2}C_{0}^{k-1}(C\|z\|_{\theta,\rho})^{2k-1}
\leq (r^{3}C_1)^{r}\|z\|_{\rho,\theta}^{2r+1}$$
with the positive constant $C$ depending on $\rho,\theta$.
\end{proof}

\subsection{Small divisor conditions}
\label{subsec42}
For any $\bs{j}\in\mc{R}_{2l}$, let
\begin{equation}
\label{form31}
\Omega_{\bs{j}}^{(4)}(I)=\sum_{i=1}^{2l}\delta_{i}\frac{\pa Z_{4}}{\pa I_{a_{i}}}=-\varphi'(0)\sum_{i=1}^{2l}\delta_{i}I_{a_i},
\end{equation}
and
\begin{equation}
\label{form32}
\Omega_{\bs{j}}^{(6)}(I)=\Omega_{\bs{j}}^{(4)}(I)+\Omega_{\bs{j}}^{(6,6)}(I)=\Omega_{\bs{j}}^{(4)}(I)+\sum_{i=1}^{2l}\delta_{i}\frac{\pa Z_{6}}{\pa I_{a_{i}}}.
\end{equation}
Then $\Omega_{\bs{j}}^{(4)}=\Omega_{Irr(\bs{j})}^{(4)}$, $\Omega_{\bs{j}}^{(6)}=\Omega_{Irr(\bs{j})}^{(6)}$ and
$$\{z_{\bs{j}},Z_{4}\}=-{\rm i}\,\Omega_{Irr(\bs{j})}^{(4)}(I)z_{\bs{j}}\quad\text{and}\quad
\{z_{\bs{j}},Z_{4}+Z_{6}\}=-{\rm i}\,\Omega_{Irr(\bs{j})}^{(6)}(I)z_{\bs{j}}.$$
In particular, if $\bs{j}\in\mc{I}$, then $\Omega_{\bs{j}}^{(4)}(I)=\Omega_{\bs{j}}^{(6)}(I)=0$.
\\\indent For any given $r\geq3$, $N\geq1$ and $\gamma>0$, we say that $u:=\sum_{a\in\mb{Z}}z_{a}e^{{\rm i}ax}\in \mc{G}_{\rho,\theta}$ belongs to the set $\mc{U}_{\gamma}^{N}$, if for any $\bs{j}\in Irr(\mc{R})$ with $j_{1}^{*}\leq N$ and $\#\bs{j}=2l\leq2r$, one has
\begin{align}
\label{form33}
&|\Omega_{\bs{j}}^{(4)}(I)|>\gamma\|z\|_{\rho,\theta}^{2}(2N)^{-12l}e^{-2\rho|j^{*}_{2l}|^{\theta}},
\\\label{form34}
&|\Omega_{\bs{j}}^{(6)}(I)|>\gamma\|z\|_{\rho,\theta}^{2}(2N)^{-12l}\max\{e^{-2\rho|j^{*}_{2l}|^{\theta}},\gamma\|z\|_{\rho,\theta}^{2}\}.
\end{align}
Since the function $u$ is equivalent to its sequence of Fourier coefficients $\{z_a\}_{a\in\mb{Z}}$, we also say $z\in\mc{U}_{\gamma}^{N}$.
\\\indent The following lemma shows that the set $\mc{U}_{\gamma}^{N}$ is stability with respect to the action.
\begin{lemma}
\label{le31}
Fix $\rho>0, \theta\in(0,1)$ and $r\geq3$. For any $N\geq1$ and $\gamma>0$, let $z\in\mc{U}_{\gamma}^{N}$. If $z'\in\ms{P}_{\rho,\theta}$ satisfies the condition
\begin{equation}
\label{form35}
\|z'\|_{\rho,\theta}\leq4\|z\|_{\rho,\theta}\quad\text{and}\quad
\sup_{a\in\mb{Z}}e^{2\rho|a|^{\theta}}|I'_{a}-I_{a}|\lesssim\frac{\gamma^{2}\|z\|_{\rho,\theta}^{2}}{r(2N)^{12r}},
\end{equation}
then $z'\in\mc{U}_{\gamma/2}^{N}$.
\end{lemma}
\begin{proof}
For any $\bs{j}=(\delta_i,a_i)_{i=1}^{2l}\in Irr(\mc{R})$ with $j_{1}^{*}\leq N$ and $l\leq r$, one has
\begin{align}
\label{form36}
|\Omega^{(4)}_{\bs{j}}(I')-\Omega^{(4)}_{\bs{j}}(I)|
&\leq|\varphi'(0)|\sum_{i=1}^{2l}|I'_{a_{i}}-I_{a_{i}}|
\\\notag&\leq 2le^{-2\rho|j^{*}_{2l}|^{\theta}}|\varphi'(0)|\sup_{|a|\leq N}e^{2\rho|a|^{\theta}}|I'_{a}-I_{a}|.
\end{align}
When
$\sup_{|a|\leq N}e^{2\rho|a|^{\theta}}|I'_{a}-I_{a}|\leq\frac{\gamma\|z\|_{\rho,\theta}^{2}}{4l|\varphi'(0)|(2N)^{12l}}$,
one has
$$|\Omega_{\bs{j}}^{(4)}(I')|\geq|\Omega_{\bs{j}}^{(4)}(I)|-|\Omega^{(4)}_{\bs{j}}(I')-\Omega^{(4)}_{\bs{j}}(I)|
>\frac{\gamma}{2}\|z\|_{\rho,\theta}^{2}(2N)^{-12l}e^{-2\rho|j^{*}_{2l}|^{\theta}}.$$
\indent Notice that $\Omega_{\bs{j}}^{(6,6)}(I)$ is a homogenous polynomial of order two with respect to $I$. Then there exists a positive constant $C$ depending on $\varphi'(0)$ and $\varphi''(0)$ such that
\begin{equation}
\label{form37}
|\Omega_{\bs{j}}^{(6,6)}(I')-\Omega_{\bs{j}}^{(6,6)}(I)|\leq Cl\|z\|_{\rho,\theta}^{2}\sup_{a\in\mb{Z}}|I'_{a}-I_{a}|.
\end{equation}
When $\sup_{|a|\leq N}e^{2\rho|a|^{\theta}}|I'_{a}-I_{a}|\leq\frac{\gamma\|z\|_{\rho,\theta}^{2}}{8l|\varphi'(0)|(2N)^{12l}}$ and
$\sup_{a\in\mb{Z}}|I'_{a}-I_{a}|\leq\frac{\gamma^{2}\|z\|_{\rho,\theta}^{2}}{8Cl(2N)^{12l}},$
by \eqref{form36} and \eqref{form37}, one has
\begin{align*}
|\Omega_{\bs{j}}^{(6)}(I')-\Omega_{\bs{j}}^{(6)}(I)|
&\leq2\max\{|\Omega^{(4)}_{\bs{j}}(I')-\Omega^{(4)}_{\bs{j}}(I)|,|\Omega_{\bs{j}}^{(6,6)}(I')-\Omega_{\bs{j}}^{(6,6)}(I)|\}
\\&\leq\frac{\gamma}{2}\|z\|_{\rho,\theta}^{2}(2N)^{-12l}\max\{e^{-2\rho|j^{*}_{2l}|^{\theta}},\frac{\gamma}{2}\|z\|_{\rho,\theta}^{2}\}.
\end{align*}
Hence, one has
\begin{align*}
|\Omega_{\bs{j}}^{(6)}(I')|
\geq&|\Omega_{\bs{j}}^{(6)}(I)|-|\Omega^{(6)}_{\bs{j}}(I')-\Omega^{(6)}_{\bs{j}}(I)|
\\>&\frac{\gamma}{2}\|z\|_{\rho,\theta}^{2}(2N)^{-12l}\max\{e^{-2\rho|j^{*}_{2l}|^{\theta}},\frac{\gamma}{2}\|z\|_{\rho,\theta}^{2}\}.
\end{align*}
\end{proof}

\section{Rational normal form}
\label{sec5}
We hope to construct a canonical transformation normalizing the system \eqref{form215}. Firstly, we want to construct $\chi_{6}'$ in such a way that $\{\chi_{6}', Z_{4}\}=K_{6}$ with the resonant polynomial $K_{6}$ in \Cref{th21}. The rational Hamiltonian functions arise from $\chi_{6}'$ associated with the small denominators $\Omega^{(4)}(I)$ in \eqref{form31}.
Similarly, we need to eliminate the non-integrable part of higher order resonant functions by $Z_{4}+Z_{6}$.
\subsection{Elimination of the sextic term}
\label{subsec51}
The following theorem aims at  eliminating the resonant polynomial $K_{6}$ in \Cref{th21} by $Z_{4}$.
\begin{theorem}
\label{th51}
Fix $\rho>0$ and $\theta\in(0,1)$. There exists $\varepsilon_{0}>0$ such that for any $0<\varepsilon\leq\varepsilon_{0}$, all positive integer $r\geq4$, $N\gtrsim r$ and $\gamma\in(0,1)$ satisfying $\varepsilon\lesssim_{\rho,\theta}r^{-\frac{3}{2}}$ and $\varepsilon\lesssim 4^{-r}\gamma^{2}(2N)^{-6r-41}$, there exists $\chi'_{6}\in\mc{H}_{4}^{N}$ such that
\begin{align}
\label{form51}
H\circ\phi^{(1)}\circ\Phi_{\chi'_{6}}^{1}(z)=&H_{0}+Z_{4}+Z_{6}+\sum_{l=4}^{r}Q_{\Gamma_{2l}}[\bs{c}_{2l}]+R'_{\geq2r+2}
\\\notag&+\big(\sum_{l=3}^{r}R_{2l}+\ms{R}+R_{\geq2r+2}\circ\phi^{(1)}\big)\circ\Phi_{\chi'_{6}}^{1}
\end{align}
with $R_{\geq2r+2}$ in \eqref{form212}, $\ms{R}$ in \Cref{th21},
    \begin{equation}
    \label{form52}
     R_{6}(z)=\sum_{\substack{\bs{j}\in\mc{R}_{6}\backslash\mc{I}_{6}\\j_1^*>N}}c_{\bs{j}}^{(6)}z_{\bs{j}}
     \quad\text{and}\quad R_{2l}(z)=\sum_{\substack{\bs{j}\in\mc{R}_{2l}\\(Irr(\bs{j}))_1^*>N}}c_{\bs{j}}^{(2l)}z_{\bs{j}}\;\text{for}\;l\geq4
    \end{equation}
  satisfying the estimate $\|c^{(2l)}\|_{\ell^{\infty}}\leq(8(l-1)^{3})^{l-2}C_{0}^{l-1}$, where
\begin{enumerate}[(i)]
\item  the transformation $\Phi_{\chi'_{6}}^{t}:\mc{U}^{N}_{\frac{\gamma}{2^{r-2}}}\cap B_{s}(2\varepsilon)\longmapsto\mc{U}^{N}_{\frac{\gamma}{2^{r-1}}}\cap B_{s}(3\varepsilon)$ satisfies the estimate
    \begin{equation}
    \label{form53}
    \|z-\Phi_{\chi'_{6}}^{1}(z)\|_{\rho,\theta}\lesssim\frac{4^{r-1}N^{81}}{\gamma^{2}}\|z\|_{\rho,\theta}^{3},\quad\|(d\Phi_{\chi'_{6}}^{1}(z))^{-1}\|_{\ms{P}_{\rho,\theta}\mapsto\ms{P}_{\rho,\theta}}\lesssim1,
    \end{equation}
    and the same estimate is fulfilled by the inverse transformation;
\item $Q_{\Gamma_{2l}}[\bs{c}_{2l}]$ is defined in \eqref{form44} with $\Gamma_{2l}\in\mc{H}_{2l}^{N}$, and there exists a positive constant $C_2$ such that
\begin{equation}
\label{form54}
\|\bs{c}_{2l}\|_{\ell^{\infty}}\leq(64C_2(l-1)^{6})^{l-2}C_0^{2l-3},
\end{equation}
and for any $(\bs{j},\bs{h},\bs{k},n)\in\Gamma_{2l}$, one has
 \begin{equation}
 \label{form55}
  n=\#\bs{h}\leq2(l-3),\; \#\bs{k}=0;
  \end{equation}
\item $R'_{\geq2r+2}$ is the remainder resonant polynomial of order at least $2r+4$ and satisfies the estimate
  \begin{equation}
  \label{form56}
  \|X_{R'_{\geq2r+2}}(z)\|_{\rho,\theta}\lesssim\frac{(C_3N)^{72r^{2}}}{\gamma^{2r-3}}\|z\|_{\rho,\theta}^{2r+1}
  \end{equation}
  with the positive constant $C_3$ depending on $C_0$ and $C_2$.
\end{enumerate}
\end{theorem}
\begin{proof}
In view of the expression of $K_{6}$ in \eqref{form219}, let
\begin{equation}
\label{form57}
\chi'_{6}(z)=\sum_{\substack{\bs{j}\in\mc{R}_{6}\backslash\mc{I}_{6}\\j_{1}^{*}\leq N}}\frac{c_{\bs{j}}^{(6)}z_{\bs{j}}}{{\rm i}\Omega_{\bs{j}}^{(4)}(I)},
\end{equation}
with $\|\bs{c}\|_{\ell^{\infty}}\leq 64C_{0}^{2}$, and then
\begin{equation}
\label{form58}
\{\chi'_{6}[\bs{c}](z),Z_{4}\}+K_{6}(z)=\sum_{\substack{\bs{j}\in\mc{R}_{6}\backslash\mc{I}_{6}\\j_{1}^{*}>N}}c_{\bs{j}}^{(6)}z_{\bs{j}}:=R_{6}
\end{equation}
Notice that the small divisor $\Omega_{\bs{j}}^{(4)}(I)$ satisfies the assumption \eqref{form7-6-1}. Then there exists $\Gamma_{4}\in\mc{H}_{4}^{N}$ satisfying $\#\bs{j}=6,n=\#\bs{h}=1,\#\bs{k}=0$ for any $(\bs{j},\bs{h},\bs{k},n)\in\Gamma_{4}$
such that $\chi'_{6}=Q_{\Gamma_{4}}[\bs{c}]$ defined in \eqref{form44}.
Then by  \eqref{re41-3} in \Cref{re41} and \Cref{le41}, one has
\begin{equation}
\label{form59}
\|X_{\chi'_{6}}(z)\|_{\rho,\theta}\lesssim3^{4}\|\bs{c}\|_{\ell^{\infty}}\frac{(2N)^{12\times3\times(1+0)+15\times3}}{(\gamma/2^{r-2})^{1+2\times0+1}}\|z\|_{\rho,\theta}^{3}\lesssim\frac{4^{r-2}N^{81}}{\gamma^{2}}\|z\|_{\rho,\theta}^{3}.
\end{equation}
For any $|t|\leq1$, as long as $\|\Phi_{\chi'_{6}}^{t}(z)\|_{\rho,\theta}\leq\frac{3}{2}\|z\|_{\rho,\theta}$, then
\begin{equation}
\label{form510}
\|\Phi_{\chi'_{6}}^{t}(z)-z\|_{\rho,\theta}\leq\big|\int_{0}^{t}\|X_{\chi'_{6}}(\Phi_{\chi'_{6}}^{t}(z))\|_{\rho,\theta}dt\big|\leq\sup_{|t|\leq1}\|X_{\chi'_{6}}(\Phi_{\chi'_{6}}^{t}(z))\|_{\rho,\theta}.
\end{equation}
Firstly, we need to make sure that the condition \eqref{form35} in \Cref{le31} holds.
Notice that
\begin{align}
\label{form511}
\sup_{a\in\mb{Z}}e^{2\rho|a|^{\theta}}||\Phi_{\chi'_{6}}^{t}(z)_{a}|^{2}-I_{a}|&\leq\|\Phi_{\chi'_{6}}^{t}(z)-z\|_{\rho,\theta}(\|\Phi_{\chi'_{6}}^{t}(z)\|_{\rho,\theta}+\|z\|_{\rho,\theta})
\\\notag&\leq\frac{5}{2}\sup_{|t|\leq1}\|X_{\chi'_{6}}(\Phi_{\chi'_{6}}^{t}(z))\|_{\rho,\theta}\|z\|_{\rho,\theta}.
\end{align}
By \eqref{form59}--\eqref{form511} and a bootstrap argument, when $\|z\|_{\rho,\theta}\lesssim4^{-r}\gamma^{2}(2N)^{-6r-38}$,
$$\Phi_{\chi'_{6}}^{t}:\mc{U}^{N}_{\frac{\gamma}{2^{r-2}}}\cap B_{\rho,\theta}(2\varepsilon)\longmapsto\mc{U}^{N}_{\frac{\gamma}{2^{r-1}}}\cap B_{\rho,\theta}(3\varepsilon)$$
 is well defined for any $|t|\leq1$, and then estimates in \eqref{form53} hold.
\\\indent By the fact that $\chi'_{6}$ commutes with $H_{0}$, one has
\begin{align*}
H\circ\phi^{(1)}\circ\Phi_{\chi'_{6}}^{1}(z)=&H_{0}+Z_{4}\circ\Phi_{\chi'_{6}}^{1}+Z_{6}\circ\Phi_{\chi'_{6}}^{1}+\sum_{l=3}^{r}(K_{2l}-R_{2l})\circ\Phi_{\chi'_{6}}^{1}
\\&+\big(\sum_{l=3}^{r}R_{2l}+\ms{R}+R_{\geq2r+2}\circ\phi^{(1)}\big)\circ\Phi_{\chi'_{6}}^{1}.
\end{align*}
Then by the homological equation \eqref{form58}, we can get \eqref{form51} with
\begin{align}
\label{form512}
Q_{\Gamma_{2l}}[\bs{c}_{2l}]=&\frac{1}{(l-2)!}ad_{\chi'_{6}}^{l-2}Z_{4}+\frac{1}{(l-3)!}ad_{\chi'_{6}}^{l-3}Z_{6}+\sum_{k=3}^{l}\frac{1}{(l-k)!}ad_{\chi'_{6}}^{l-k}(K_{2k}-R_{2k}),
\\\label{form513}
R'_{\geq2r+2}=&\int_{0}^{1}\frac{(1-t)^{r-2}}{(r-2)!}ad_{\chi'_{6}}^{r-1}Z_{4}\circ\Phi_{\chi'_{6}}^{t}dt+\int_{0}^{1}\frac{(1-t)^{r-3}}{(r-3)!}ad_{\chi'_{6}}^{r-2}Z_{6}\circ\Phi_{\chi'_{6}}^{t}dt
\\\notag&+\sum_{k=3}^{r}\int_{0}^{1}\frac{(1-t)^{r-k}}{(r-k)!}ad_{\chi'_{6}}^{r+1-k}(K_{2k}-R_{2k})\circ\Phi_{\chi'_{6}}^{t}dt.
\end{align}
By \Cref{le42}, \Cref{re7-12} and induction, one has $\Gamma_{2l}\in\mc{H}_{2l}^{N}$ with
$$n=\#\bs{h}\leq2(l-3), \quad\#\bs{k}=0$$
for any $(\bs{j},\bs{h},\bs{k},n)\in\Gamma_{2l}$. Then $\#\bs{j}=2\#\bs{h}+4\#\bs{k}+2l\leq 6(l-2)$.

Moreover, by \Cref{le42}, being given $\Gamma_{2l_{1}}\in\mc{H}_{2l_{1}}^{N}$,$\Gamma_{2l_{2}}\in\mc{H}_{2l_{2}}^{N}$ satisfying $\#\bs{j}'\lesssim l_1$, $\#\bs{j}''\lesssim l_2$ for any $(\bs{j}',\bs{h}',\bs{k}',n_1)\in\Gamma_{2l_1}$, $(\bs{j}'',\bs{h}'',\bs{k}'',n_2)\in\Gamma_{2l_2}$, there exists $\Gamma_{2l}\in\mc{H}_{2l}^{N}$ with $l=l_{1}+l_{2}-1$ such that
$$\{Q_{\Gamma_{2l_{1}}}[\tilde{\bs{c}}],Q_{\Gamma_{2l_{2}}}[\tilde{\tilde{\bs{c}}}]\}=Q_{\Gamma_{2l}}[\bs{c}],$$
where there exists a positive constant $C_2$ such that 
\begin{equation}
\label{form7-13-1}
\|\bs{c}\|_{\ell^{\infty}}\leq C_2(l-1)^{5}\|\tilde{\bs{c}}\|_{\ell^{\infty}}\|\tilde{\tilde{\bs{c}}}\|_{\ell^{\infty}}.
\end{equation}
Hence, by \eqref{form58}, \eqref{form512}, \eqref{form7-13-1} and induction, one has
\begin{small}
\begin{align*}
\|\bs{c}_{2l}\|_{\ell^{\infty}}\leq&(8(l-1)^{3})^{l-2}C_{0}^{l-1}+2\sum_{k=2}^{l-1}\frac{(l-1)^{5}\cdots(k+1)^{5}k^{5}}{(l-k)!}C_2^{l-k}(64C_{0}^{2})^{l-k}(8(k-1)^{3})^{k-2}C_{0}^{k-1}
\\\leq&(8(l-1)^{3})^{l-2}C_{0}^{l-1}+2\sum_{k=2}^{l-1}(l-1)^{5(l-k)}(64C_2)^{l-2}(l-2)^{3(k-2)}C_0^{2l-k-1}
\\\leq&(8(l-1)^{3})^{l-2}C_{0}^{l-1}+2\sum_{k=2}^{l-1}(64C_2)^{l-2}(l-1)^{5(l-2)}C_{0}^{2l-3}
\\\leq&(64C_2(l-1)^{6})^{l-2}C_{0}^{2l-3},
\end{align*}
\end{small}%
where the last inequality follows from $1+2(l-2)\leq(l-1)^{l-2}$ for any $l\geq4$.
\\\indent Now, we aim to prove the estimate \eqref{form56}. Similarly, by \Cref{le42}, there exists $\Gamma_{2(r+1)}\in\mc{H}_{2(r+1)}^{N}$ such that $\frac{1}{(r-2)!}ad_{\chi'_{6}}^{r-1}Z_{4}=Q_{\Gamma_{2(r+1)}}[\bs{c}]$ defined in \eqref{form44} satisfying the coefficient estimate $\|\bs{c}\|_{\ell^{\infty}}\leq(64C_0^{2}C_2r^{6})^{r}$, where for any $(\bs{j},\bs{h},\bs{k},n)\in\Gamma_{2(r+1)}$, one has
\begin{equation}
\label{form514}
n=\#\bs{h}\leq2(r-2),\quad\#\bs{k}=0\quad\text{and}\quad\#\bs{j}=2(r+1)+2\#\bs{h}+4\#\bs{k}\leq6(r-1).
\end{equation}
It is the same for $\frac{1}{(r-3)!}ad_{\chi'_{6}}^{r-2}Z_{6}, \frac{1}{(r-k)!}ad_{\chi'_{6}}^{r+1-k}K_{2k}$.
Then we consider these terms of the form
\begin{equation}
\label{form515}
\tilde{R}(z):=\int_{0}^{1}g(t)Q_{\Gamma_{2(r+1)}}\circ\Phi_{\chi'_{6}}^{t}(z)dt
\end{equation}
with $\|g(t)\|_{L^{\infty}}\leq1$. Notice that
\begin{align}
\label{form516}
\|X_{Q_{\Gamma_{2(r+1)}}\circ\Phi_{\chi'_{6}}^{t}}(z)\|_{\rho,\theta}
&=\|(d\Phi_{\chi'_{6}}^{t})^{-1}X_{Q_{\Gamma_{2(r+1)}}}(\Phi_{\chi'_{6}}^{t}(z))\|_{\rho,\theta}
\\\notag&\lesssim\|X_{Q_{\Gamma_{2(r+1)}}}(\Phi_{\chi'_{6}}^{t}(z))\|_{\rho,\theta}.
\end{align}
Notice that $N\gtrsim r$ and $\|\Phi_{\chi'_{6}}^{t}(z))\|_{\rho,\theta}\lesssim\|z\|_{\rho,\theta}$ for $|t|\leq1$. Then by \eqref{re41-3} in \Cref{re41}, \Cref{le41} and \eqref{form514}, there exists a positive constant $C_3$ depending on $C_0$ and $C_2$ such that
\begin{align}
\label{form517}
&\|X_{Q_{\Gamma_{2(r+1)}}}(\Phi_{\chi'_{6}}^{t}(z))\|_{\rho,\theta}
\\\notag\lesssim&(6r)^{6r}(64C_2C_0^{2}r^{6})^{r}\frac{(2N)^{12\times3(r-1)\times(2r-4+0)+15\times3(r-1)}}{(\gamma/2^{r-2})^{2r-4+1}}\|(\Phi_{\chi'_{6}}^{t}(z))\|_{\rho,\theta}^{2(r+1)-1}
\\\notag=&(6^{6}\times64C_2C_0^{2}r^{12})^{r}\frac{2^{(r-2)(2r-3)}(2N)^{72r^{2}-171r+99}}{\gamma^{2r-3}}\|(\Phi_{\chi'_{6}}^{t}(z))\|_{\rho,\theta}^{2(r+1)-1}
\\\notag\leq&\frac{(C_3N)^{72r^{2}-159r+99}}{\gamma^{2r-3}}\|z\|_{\rho,\theta}^{2r+1}.
\end{align}
Hence, by \eqref{form515}--\eqref{form517}, one has
\begin{align}
\label{form518}
\|X_{\tilde{R}}(z)\|_{\rho,\theta}\leq&\|g(t)\|_{L^{\infty}}\int_{0}^{1}\|X_{Q_{\Gamma_{2(r+1)}}\circ\Phi_{\chi'_{6}}^{t}}(z)\|_{\rho,\theta}dt
\\\notag\lesssim&
\frac{(C_3N)^{72r^{2}-159r+99}}{\gamma^{2r-3}}\|z\|_{\rho,\theta}^{2r+1}.
\end{align}
In view of \eqref{form513}, one has
$$\|X_{R'_{\geq2r+2}}(z)\|_{\rho,\theta}\leq r\|X_{\tilde{R}}(z)\|_{\rho,\theta}\lesssim\frac{(C_3N)^{72r^{2}-159r+100}}{\gamma^{2r-3}}\|z\|_{\rho,\theta}^{2r+1}.$$
\end{proof}

\subsection{Elimination of the higher order terms}
\label{subsec52}
This subsection aims at eliminating the non-integrable part of resonant polynomials $Q_{\Gamma_{8}},\cdots Q_{\Gamma_{2r}}$ in Theorem \ref{th51} by $Z_{4}+Z_{6}$.
In fact, we have the following rational normal form theorem.
\begin{theorem}
\label{th52}
Fix $\rho>0$ and $\theta\in(0,1)$. There exists $\varepsilon_{0}>0$ such that for any $0<\varepsilon\leq\varepsilon_{0}$, all positive integer $r\geq4$, $N\gtrsim r$ and $\gamma\in(0,1)$ satisfying $\varepsilon\lesssim_{\rho,\theta}r^{-\frac{3}{2}}$ and $\varepsilon\lesssim\gamma^{\frac{7}{2}}(C_4N)^{-\frac{341}{2}r}$ with some properly large positive constants $C_4$, there exists a canonical transformation $\phi^{(2)}:\mc{U}^{N}_{\frac{\gamma}{2}}\cap B_{\rho,\theta}(\frac{3}{2}\varepsilon)\longmapsto \mc{U}^{N}_{\frac{\gamma}{2^{r-2}}}\cap B_{\rho,\theta}(2\varepsilon)$ such that
\begin{align}
\label{form519}
H\circ\phi^{(1)}\circ\Phi_{\chi'_{6}}^{1}\circ\phi^{(2)}(z)=&H_{0}+Z_{4}+Z_{6}+\sum_{l=4}^{r}Z_{2l}[\bs{c}'_{2l}]+R''_{\geq2r+2}
\\\notag&+\big(\sum_{l=3}^{r}R_{2l}+\ms{R}+R_{\geq2r+2}\circ\phi^{(1)}\big)\circ\Phi_{\chi'_{6}}^{1}\circ\phi^{(2)},
\end{align}
with $R_{2l}$ defined in \eqref{form52}, $\ms{R}$ in \Cref{th21} and $R_{\geq2r+2}$ in \eqref{form212}, where
\begin{enumerate}[(i)]
 \item  the transformation satisfies the estimate
    \begin{equation}
    \label{form520}
    \|z-\phi^{(2)}(z)\|_{\rho,\theta}\lesssim\frac{32^{r}(C_5N)^{377}}{\gamma^{5}}\|z\|_{\rho,\theta}^{3} ,\quad\|(d\phi^{(2)}(z))^{-1}\|_{\ms{P}_{\rho,\theta}\mapsto\ms{P}_{\rho,\theta}}\lesssim1
    \end{equation}
with the positive constant $C_5$ depending on $C_0,C_2$,  and the same estimate is fulfilled by the inverse transformation;
 \item $Z_{2l}[\bs{c}'_{2l}]$ is the integrable rational polynomial and satisfies the estimate
   \begin{equation}
   \label{form521}
  \|\bs{c}'_{2l}\|_{\ell^{\infty}}\leq(C_6(l-1))^{5l^{2}}
   \end{equation}
   with the positive constant $C_6$ depending on $C_0$, $C_2$;
\item $R''_{\geq2r+2}$ is the remainder resonant polynomial of order at least $2r+4$ and satisfies the estimate
   \begin{equation}
   \label{form522}
   \|X_{R''_{\geq2r+2}}(z)\|_{\rho,\theta}\lesssim\frac{(C_7N)^{341r^{2}}}{\gamma^{6r}}\|z\|_{\rho,\theta}^{2r+1}
    \end{equation}
    with the positive constant $C_7$ depending on $C_6$.
\end{enumerate}
\end{theorem}
In order to state the iterative lemma, we need some notations. For $l=4,\cdots, r$, denote
\begin{equation}
\label{form523}
\varepsilon^{(l)}=\varepsilon(2-\frac{l-3}{2(r-3)}), \quad\text{for}\quad l=4,\cdots, r.
\end{equation}
We proceed by induction from $l-1$ to $l$. Initially, write
\begin{equation}
\label{form524}
H^{(0)}(z)=H_{0}+Z_4+Z_6+\sum_{l=4}^{r}Q_{\Gamma_{2l}}[\bs{c}_{2l}]+R'_{\geq2r+2},
\end{equation}
with $Q_{\Gamma_{2l}}[\bs{c}_{2l}]$ and $R'_{\geq2r+2}$ in \Cref{th51}.
\begin{lemma}
\label{le51}
When $\varepsilon\lesssim\gamma^{\frac{7}{2}}(C_4N)^{-\frac{341}{2}r}$, there exists a rational polynomial $\chi'_{2l}[\bs{c}''_{2(l-2)}]$ such that
\begin{align}
\label{form525}
H^{(l-3)}(z):=&H^{(l-4)}\circ\Phi_{\chi'_{2l}}^{1}(z)
\\\notag=&H_{0}+Z_{4}+Z_{6}+\sum_{k=4}^{l}Z_{\Gamma'_{2k}}[\bs{c}'_{2k}]+\sum_{k=l+1}^{r}Q_{\Gamma^{(l-3)}_{2k}}[\bs{c}^{(l-3)}_{2k}]+R^{(l-3)}_{\geq2r+2},
\end{align}
where
\begin{enumerate}[(i)]
 \item the transformation $\Phi_{\chi'_{2l}}^{1}:\mc{U}^{N}_{\frac{\gamma}{2^{r+1-l}}}\cap B_{\rho,\theta}(\varepsilon^{(l)})\longmapsto \mc{U}^{N}_{\frac{\gamma}{2^{r+2-l}}}\cap B_{\rho,\theta}(\varepsilon^{(l-1)})$ satisfies the estimate
    \begin{equation}
    \label{form526}
    \|z-\Phi^{1}_{\chi'_{2l}}(z)\|_{\rho,\theta}\leq\frac{2^{(r+1-l)(6l-19)}(C_{5}N)^{341l^{2}-2049l+3117}}{\gamma^{6l-19}}\|z\|_{\rho,\theta}^{2l-5},
    \end{equation}
    \begin{equation}
    \label{form527}
    \|(d\Phi_{\chi'_{2l}}^{1}(z))^{-1}\|_{\ms{P}_{\rho,\theta}\mapsto\ms{P}_{\rho,\theta}}\lesssim1,
    \end{equation}
    and the same estimate is fulfilled by the inverse transformation;
 \item the integrable polynomial $Z_{\Gamma'_{2k}}[\bs{c}'_{2k}]$ and the resonant polynomial $Q_{\Gamma_{2k}^{(l-3)}}[\bs{c}^{(l-3)}_{2k}]$ are defined in \eqref{form44} satisfying $\Gamma'_{2k},\Gamma_{2k}^{(l-3)}\in\mc{H}_{2k}^{N}$ and the coefficients estimate
\begin{equation}
\label{form528}
 \|\bs{c}'_{2k}\|_{\ell^{\infty}}\leq(64C_{2}(k-1)^{5})^{(k+2)(k-1)}C_0^{6(k-3)},\quad\text{for}\;k=4,\cdots,l,
 \end{equation}
 \begin{equation}
 \label{form529}
 \|\bs{c}^{(l-3)}_{2k}\|_{\ell^{\infty}}\leq(64C_{2}(k-1)^{5})^{l(k-3)}C_0^{6(k-3)},\quad\text{for}\;k\geq l+1,
 \end{equation}
and for any $(\bs{j},\bs{h},\bs{k},n)\in\Gamma'_{2k},\Gamma_{2k}^{(l-3)}$, one has
  \begin{equation}
  \label{form530}
 n\leq\#\bs{h}\leq3k-10,\; \#\bs{k}\leq2(k-4),\; \#\bs{h}+\#\bs{k}\leq4k-14;
  \end{equation}
 \item the remainder resonant polynomial $R^{(l-3)}_{\geq2r+2}$ satisfies the estimate
   \begin{equation}
   \label{form531}
   \|X_{R^{(l-3)}_{\geq2r+2}}(z)\|_{\rho,\theta}\leq\frac{4^{(l-3)r}(C_{8}N)^{341r^{2}-1456r+1585}}{\gamma^{6r-15}}\|z\|_{\rho,\theta}^{2r+1}
    \end{equation}
    with the positive constant $C_8$ depending on $C_0$ and $C_2$.
\end{enumerate}
\end{lemma}
\begin{proof}
Notice that for any $\bs{j}\in Irr(\mc{R})$, the small divisors $\Omega_{\bs{j}}^{(4)}(I)$ and $\Omega_{\bs{j}}^{(6)}(I)$ satisfies the assumption \eqref{form7-6-1}--\eqref{form7-6-2}. By the definition of rational Hamiltonian functions, write
$$Q_{\Gamma_{2l}^{(l-4)}}[\bs{c}_{2l}^{(l-4)}](z)=\sum_{J:=(\bs{j},\bs{h},\bs{k},n)\in\Gamma_{2l}^{(l-4)}}c_{J}\frac{z_{\bs{j}}}{\prod\limits_{m=1}^{n}\Omega^{(4)}_{\bs{h}_{m}}\prod\limits_{m=n+1}^{\#\bs{h}}\Omega^{(6)}_{\bs{h}_{m}}\prod\limits_{m=1}^{\#\bs{k}}\Omega^{(6)}_{\bs{k}_{m}}}$$
and let
\begin{equation}
\label{form532}
\chi'_{2l}[\bs{c}''_{2(l-2)}](z)=\sum_{\substack{J:=(\bs{j},\bs{h},\bs{k},n)\in\Gamma_{2l}^{(l-4)}\\\bs{j}\in\mc{R}\backslash\mc{I}}}\frac{c_{J}}{{\rm i}\Omega^{(6)}_{Irr(\bs{j})}(I)}\frac{z_{\bs{j}}}{\prod\limits_{m=1}^{n}\Omega^{(4)}_{\bs{h}_{m}}(I)\prod\limits_{m=n+1}^{\#\bs{h}}\Omega^{(6)}_{\bs{h}_{m}}(I)\prod\limits_{m=1}^{\#\bs{k}}\Omega^{(6)}_{\bs{k}_{m}}(I)}.
\end{equation}
Then there exists $\Gamma'_{2l}=\Gamma_{2l}^{(l-4)}\bigcap\big(\mc{I}\times\bigcup_{q\geq0}(Irr(\mc{R}))^{q}\times\bigcup_{q'\geq0}(Irr(\mc{R}))^{q'}\times\mb{N}\big)$ such that
\begin{equation}
\label{form533}
\{\chi'_{2l}[\bs{c}''_{2(l-2)}],Z_{4}+Z_{6}\}+Q_{\Gamma_{2l}^{(l-4)}}[\bs{c}^{(l-4)}_{2l}]=Z_{\Gamma'_{2l}}[\bs{c}'_{2l}],
\end{equation}
where
$$Z_{\Gamma'_{2l}}[\bs{c}'_{2l}](z)=\sum_{J:=(\bs{j},\bs{h},\bs{k},n)\in\Gamma'_{2l}}c_J\frac{z_{\bs{j}}}{\prod\limits_{m=1}^{n}\Omega^{(4)}_{\bs{h}_{m}}(I)\prod\limits_{m=n+1}^{\#\bs{h}}\Omega^{(6)}_{\bs{h}_{m}}(I)\prod\limits_{m=1}^{\#\bs{k}}\Omega^{(6)}_{\bs{k}_{m}}(I)}.$$
Notice that by \eqref{form530} for $k=l$ and \eqref{form532}, we can write $\chi'_{2l}[\bs{c}''_{2(l-2)}]=Q_{\Gamma''_{2(l-2)}}$ defined in \eqref{form44}, where $\Gamma''_{2(l-2)}\in\mc{H}_{2(l-2)}^{N}$ and for any $(\bs{j},\bs{h},\bs{k},n)\in\Gamma''_{2(l-2)}$, one has
\begin{align}
\label{form534}
&\#\bs{h}\leq3l-10,
\\\label{form534-1}
&\#\bs{k}\leq2(l-4)+1=2l-7,
\\\label{form534-2}
&\#\bs{h}+\#\bs{k}\leq4l-14+1=4l-13,
\\\label{form535}
&\#\bs{j}=2(l-2)+2\#\bs{h}+4\#\bs{k}\leq2(7l-22).
\end{align}
Besides, by \eqref{form529}, one has
\begin{align}
\label{form536}
\|\bs{c}'_{2l}\|_{\ell^{\infty}},\|\bs{c}''_{2(l-2)}\|_{\ell^{\infty}}\leq\|\bs{c}_{2l}^{(l-4)}\|_{\ell^{\infty}}\leq(64C_{2}(l-1)^{5})^{(l-1)(l-3)}C_0^{6(l-3)}.
\end{align}
Hence, by \eqref{re41-3} in \Cref{re41}, \Cref{le41}, \eqref{form534}--\eqref{form536}, there exists a positive constant $C_5$ depending on $C_0$ and $C_2$ such that
\begin{align}
\label{form537}
&\|X_{\chi'_{2l}}(z)\|_{\rho,\theta}
\\\notag\lesssim&(14l)^{14l}(64C_{2}(l-1)^{5})^{(l-1)(l-3)}C_0^{6(l-3)}\frac{(2N)^{12(7l-22)(4l-13)+15(7l-22)}}{(\gamma/2^{r+1-l})^{4l-13+2l-7+1}}\|z\|_{\rho,\theta}^{2l-5}
\\\notag\leq&14^{14l}C_0^{6l}(64C_2)^{l^{2}}l^{5l^{2}-6l+15}\frac{2^{(r+1-l)(6l-19)}(2N)^{336l^{2}-2043l+3102}}{\gamma^{6l-19}}\|z\|_{\rho,\theta}^{2l-5}
\\\notag\leq&\frac{2^{(r+1-l)(6l-19)}(C_5N)^{341l^{2}-2049l+3117}}{\gamma^{6l-19}}\|z\|_{\rho,\theta}^{2l-5}.
\end{align}
When $\varepsilon\lesssim8^{-r}\gamma^{3}(C_5N)^{-\frac{341}{2}l}$, one has $\|X_{\chi'_{2l}}(z)\|_{\rho,\theta}\leq\frac{\varepsilon}{2(r-3)}:=\varepsilon^{(l-1)}-\varepsilon^{(l)}$. Then for any $|t|\leq1$, one has $\|\Phi_{\chi'_{2l}}^{t}(z)\|_{\rho,\theta}\lesssim\|z\|_{\rho,\theta}$ and
$\|\Phi_{\chi'_{2l}}^{t}(z)-z\|_{\rho,\theta}\leq\sup_{|t|\leq1}\|X_{\chi'_{2l}}(\Phi_{\chi'_{2l}}^{t}(z))\|_{\rho,\theta}.$
Hence, one has
\begin{align*}
&\sup_{a\in\mb{Z}}e^{2\rho|a|^{\theta}}||\Phi_{\chi'_{2l}}^{t}(z)_{a}|^{2}-I_{a}|
\\\leq&\|\Phi_{\chi'_{2l}}^{t}(z)-z\|_{\rho,\theta}(\|\Phi_{\chi'_{2l}}^{t}(z)\|_{\rho,\theta}+\|z\|_{\rho,\theta})
\\\lesssim&\frac{2^{(r+1-l)(6l-19)}(C_5N)^{341l^{2}-2049l+3117}}{\gamma^{6l-19}}\sup_{|t|\leq1}\|\Phi_{\chi'_{2l}}^{t}(z)\|_{\rho,\theta}^{2l-5}(\Phi_{\chi'_{2l}}^{t}(z)\|_{\rho,\theta}+\|z\|_{\rho,\theta}).
\\\lesssim&\frac{2^{(r+1-l)(6l-19)}(C_5N)^{341l^{2}-2049l+3117}}{\gamma^{6l-19}}(2\|z\|_{\rho,\theta})^{2l-4}.
\end{align*}
When $\varepsilon\lesssim\gamma^{\frac{7}{2}}(C_5N)^{-\frac{341}{2}r}$,  the condition \eqref{form35} in \Cref{le31} holds. Hence, for any $|t|\leq1$, $\Phi_{\chi'_{2l}}^{t}$
 is well defined in $\mc{U}^{N}_{\frac{\gamma}{2^{r+1-l}}}\cap B_{\rho,\theta}(\varepsilon^{(l)})$ and the estimates \eqref{form526}, \eqref{form527} hold.
\\\indent Notice that
\begin{align}
\label{form538}
H^{(l-4)}\circ\Phi_{\chi'_{2l}}^{1}=&H_{0}+(Z_{4}+Z_{6})\circ\Phi_{\chi'_{2l}}^{1}+\sum_{m=4}^{l-1}Z_{\Gamma'_{2m}}[\bs{c}'_{2m}]\circ\Phi_{\chi'_{2l}}^{1}
\\\notag&+\sum_{m=l}^{r}Q_{\Gamma^{(l-4)}_{2m}}[\bs{c}^{(l-4)}_{2m}]\circ\Phi_{\chi'_{2l}}^{1}+R^{(l-4)}_{\geq2r+2}\circ\Phi_{\chi'_{2l}}^{1}.
\end{align}
By \Cref{le42}, \eqref{form534}--\eqref{form535}, there exists $\Gamma^{(l-3)}_{2k}\in\mc{H}_{2k}^{N}$ such that
\begin{align}
\label{form539}
Q_{\Gamma^{(l-3)}_{2k}}[\bs{c}^{(l-3)}_{2k}]=&\frac{1}{(\frac{k-3}{l-3})!}ad_{\chi'_{2l}}^{\frac{k-3}{l-3}}(Z_{4}+Z_{6})
\\\notag&+\sum_{m=4}^{l-1}\frac{1}{(\frac{k-m}{l-3})!}ad_{\chi'_{2l}}^{\frac{k-m}{l-3}}Z_{\Gamma'_{2m}}+\sum_{m=l}^{k}\frac{1}{(\frac{k-m}{l-3})!}ad_{\chi'_{2l}}^{\frac{k-m}{l-3}}Q_{\Gamma_{2m}^{(l-4)}}
\end{align}
with convention that $ad^{x}Q=0$ if $x\notin\mb{N}$.
By \Cref{le42}, \Cref{re7-12} and induction,  for any $(\bs{j},\bs{h},\bs{k},n)\in\Gamma^{(l-3)}_{2k}$, one has
\begin{align*}
&\#\bs{h}\leq \sup_{m=3,\cdots,k}\{(3m-10)+\frac{k-m}{l-3}(3l-10+1)\}=3k-10,
\\&\#\bs{k}\leq\sup_{m=3,\cdots,k}\{2(m-4)+\frac{k-m}{l-3}(2l-7+1)\}=2(k-4),
\\&\#\bs{h}+\#\bs{k}\leq\sup_{m=3,\cdots,k}\{(4m-14)+\frac{k-m}{l-3}(4l-13+1)\}=4k-14,
\end{align*}
and thus
\begin{align}
\label{form539''}
\#\bs{j}&=2\#\bs{h}+4\#\bs{k}+2k
\\\notag&\leq2(4k-14)+4(k-4)+2k
\\\notag&\leq2(7k-22).
\end{align}
Besides, by the coefficient estimate \eqref{form7-13-1} and induction, one has
\begin{small}
\begin{align*}
\|\bs{c}^{(l-3)}_{2k}\|_{\ell^{\infty}}\leq&\sum_{m=3}^{k-1}(C_2(k-1)^{5})^{\frac{k-m}{l-3}}(64C_{2}(k-1)^{5})^{(l-1)(m-3)+\frac{k-m}{l-3}(l-1)(l-3)}C_0^{6(m-3)+\frac{k-m}{l-3}6(l-3)}
\\&+(64C_{2}(k-1)^{6})^{(l-1)(k-3)}C_0^{6(k-3)}
\\\leq&(k-3)(C_2(k-1)^{5})^{k-3}(64C_{2}(k-1)^{5})^{(l-1)(k-3)}C_0^{6(k-3)}
\\&+(64C_{2}(k-1)^{5})^{(l-1)(k-3)}C_0^{6(k-3)}
\\\leq&(64C_{2}(k-1)^{5})^{l(k-3)}C_0^{6(k-3)},
\end{align*}
\end{small}%
where the last inequality follows from the fact $k-2\leq64^{k-3}$.
Then by \eqref{form533}, \eqref{form538} and \eqref{form539}, we can get \eqref{form525}, where $R^{(l-3)}_{\geq2r+2}=R_{\geq2r+2}^{(l-4)}\circ\Phi_{\chi'_{2l}}^{1}+\tilde{R}^{(l-3)}_{\geq2r+2}$ with
\begin{align*}
\tilde{R}^{(l-3)}_{\geq2r+2}=&\int_{0}^{1}\frac{(1-t)^{[\frac{r-3}{l-3}]}}{[\frac{r-3}{l-3}]!}ad_{\chi'_{2l}}^{[\frac{r-3}{l-3}]+1}(Z_{4}+Z_{6})\circ\Phi_{\chi'_{2l}}^{t}dt
\\\notag&+\sum_{m=4}^{l-1}\int_{0}^{1}\frac{(1-t)^{[\frac{r-m}{l-3}]}}{[\frac{r-m}{l-3}]!}ad_{\chi'_{2l}}^{[\frac{r-m}{l-3}]+1}Z_{\Gamma'_{2m}}\circ\Phi_{\chi'_{2l}}^{t}dt
\\\notag&+\sum_{m=l}^{r}\int_{0}^{1}\frac{(1-t)^{[\frac{r-m}{l-3}]}}{[\frac{r-m}{l-3}]!}ad_{\chi'_{2l}}^{[\frac{r-m}{l-3}]+1}Q_{\Gamma_{2m}^{(l-4)}}\circ\Phi_{\chi'_{2l}}^{t}dt.
\end{align*}
\indent Finally, we will prove the estimate \eqref{form531}. By \Cref{le42},  there exists $\Gamma_{2\tilde{r}_3}\in\mc{H}_{2\tilde{r}_3}^{N}$, $\Gamma_{2\tilde{r}_m}\in\mc{H}_{2\tilde{r}_m}^{N}$ for $m=4,\cdots,r$ with $r+1\leq\tilde{r}_3,\tilde{r}_m\leq r+l-3\leq2r$ such that
$$ad_{\chi'_{2l}}^{[\frac{r-3}{l-3}]+1}(Z_{4}+Z_{6})=Q_{\Gamma_{2\tilde{r}_3}}[\bs{c}],$$
$$ad_{\chi'_{2l}}^{[\frac{r-m}{l-3}]+1}Z_{\Gamma'_{2m}}=Q_{\Gamma_{2\tilde{r}_m}}[\bs{c}],\;\text{for}\;m=4,\cdots,l-1,$$
$$ad_{\chi'_{2l}}^{[\frac{r-m}{l-3}]+1}Q_{\Gamma_{2m}^{(l-4)}}=Q_{\Gamma_{2\tilde{r}_m}}[\bs{c}],\;\text{for}\;m=l,\cdots,r.$$
Parallelly to \eqref{form515}--\eqref{form518} in the proof of Theorem \ref{th51}, by \Cref{le42}, \Cref{le41}, \eqref{form529} and \eqref{form530},\eqref{form539''} for $k=r+1$, there exits a positive constant $C_8$ depending on $C_0$ and $C_2$ such that when $\varepsilon\lesssim\gamma^{3}(C_8N)^{-\frac{341}{2}r}$,  one has
\begin{align}
\label{form540}
&\|X_{\tilde{R}^{(l-3)}_{\geq2r+2}}(z)\|_{\rho,\theta}
\\\notag\lesssim&(14r)^{14r}(64C_2r^{5})^{(l-1)(r-2)}C_0^{6(r-2)}\frac{(2N)^{12(7r-15)(4r-10)+15(7r-15)}}{(\gamma/2^{r+1-l})^{4r-10+2(r-3)+1}}\|z\|_{\rho,\theta}^{2r+1}
\\\notag\leq&14^{14r}C_0^{6r}(64C_2)^{r^{2}}r^{5r^{2}-r+10}\frac{2^{6r^{2}}(2N)^{336r^{2}-1455r+1575}}{\gamma^{6r-15}}\|z\|_{\rho,\theta}^{2r+1}
\\\notag\leq&\frac{(C_{8}N)^{341r^{2}-1456r+1585}}{\gamma^{6r-15}}\|z\|_{\rho,\theta}^{2r+1}.
\end{align}
Notice that by \eqref{form531} and the estimate $\|\Phi_{\chi'_{2l}}^{1}(z)\|_{\rho,\theta}\leq2\|z\|_{\rho,\theta}$, one has
\begin{align}
\label{form541}
\|X_{R_{\geq2r+2}^{(l-4)}\circ\Phi_{\chi'_{2l}}^{1}}(z)\|_{\rho,\theta}=&\|(d\Phi_{\chi'_{2l}}^{1})^{-1}X_{R_{\geq2r+2}^{(l-4)}}(\Phi_{\chi'_{2l}}^{1}(z))\|_{\rho,\theta}
\\\notag&\lesssim\|X_{R_{\geq2r+2}^{(l-4)}}(\Phi_{\chi'_{2l}}^{1}(z))\|_{\rho,\theta}
\\\notag&\leq\frac{4^{(l-4)r}(C_{8}N)^{341r^{2}-1456r+1585}}{\gamma^{6r-15}}\|\Phi_{\chi'_{2l}}^{1}(z)\|_{\rho,\theta}^{2r+1}
\\\notag&\leq \frac{4^{(l-3)r-1}(C_{8}N)^{341r^{2}-1456r+1585}}{\gamma^{6r-15}}\|z\|_{\rho,\theta}^{2r+1}.
\end{align}
 By \eqref{form540} and \eqref{form541},  one has
\begin{align*}
\|X_{R^{(l-3)}_{\geq2r+2}}(z)\|_{\rho,\theta}
&\leq\|X_{\tilde{R}^{(l-3)}_{\geq2r+2}}(z)\|_{\rho,\theta}+\|X_{R_{\geq2r+2}^{(l-4)}\circ\Phi_{\chi'_{2l}}^{1}}(z)\|_{\rho,\theta}
\\&\leq\frac{4^{(l-3)r}(C_{8}N)^{341r^{2}-1456r+1585}}{\gamma^{6r-15}}\|z\|_{\rho,\theta}^{2r+1}.
\end{align*}
Take $C_4=\max\{8C_5,C_8\}$ and then we can draw the conclusion.
\end{proof}

\begin{proof}[Proof of Theorem {\sl\ref{th52}}]
By \Cref{le51}, After $r$ iterations, let $\phi^{(2)}=\Phi_{\chi_{8}}^{1}\circ\cdots\circ\Phi_{\chi_{2r}}^{1}$ and $R''_{\geq2r+2}=R^{(r-3)}_{\geq2r+2}$.
When $\varepsilon\lesssim\gamma^{\frac{7}{2}}(C_4N)^{-\frac{341}{2}r}$, the canonical transformation $\phi^{(2)}:\mc{U}^{N}_{\frac{\gamma}{2}}\cap B_{\rho,\theta}(\frac{3}{2}\varepsilon)\longmapsto \mc{U}^{N}_{\frac{\gamma}{2^{r-2}}}\cap B_{\rho,\theta}(2\varepsilon)$ satisfies the estimate
$$\sup_{\|z\|_{\rho,\theta}\leq\frac{3}{2}\varepsilon}\|z-\phi^{(2)}(z)\|_{\rho,\theta}\leq\sum_{l=4}^{r} \sup_{\|z\|_{\rho,\theta}\leq\varepsilon^{(l)}}\|z-\Phi^{1}_{\chi'_{2l}}(z)\|_{\rho,\theta}\lesssim\frac{32^{r}(C_5N)^{377}}{\gamma^{5}}\|z\|_{\rho,\theta}^{3}.$$
Might as well take $C_6=4C_0C_2$ and $C_7=4C_8$. Then estimates \eqref{form521} and \eqref{form522} follow from \eqref{form528} and \eqref{form531}, respectively.
\end{proof}

\section{Measure estimate}
\label{sec6}
In this section, we show that almost all small initial data belong to the set $\mc{U}_{\gamma}^{N}$ defined in \Cref{subsec42}. The Gaussian measure \eqref{form15}  is equivalent to the measure
\begin{equation}
\label{form38}
d\mu=\frac{e^{-\sum_{a\in\mb{Z}}e^{2\rho|a|^{\theta}}|a|^{2}I_{a}}dId\theta}{\int_{\sum_{a\in\mb{Z}}e^{2\rho|a|^{\theta}}I_{a}<1}e^{-\sum_{a\in\mb{Z}}e^{2\rho|a|^{\theta}}|a|^{2}I_{a}}dId\theta},
\end{equation}
where $z_{a}=I_{a}^{\frac{1}{2}}e^{{\rm i}\theta_{a}}$ with $I_{a}>0$ and $\theta_{a}\in[0,2\pi]$. Then we have the following measure estimate.
\begin{lemma}
\label{le32}
Fix $\rho>0$, $\theta\in(0,1)$, $r\geq3$, $N\geq12r$ and $\gamma\lesssim e^{-2\rho(3r)^{\theta}}$. If
\begin{equation}
\label{form39}
\varepsilon^{2}\lesssim\frac{\gamma}{r(2N)^{12r}},
\end{equation}
then there exists a positive constant $\lambda$ such that
\begin{equation}
\label{form310}
\mu(\varepsilon z\in\mc{U}_{\gamma}^{N})\geq1-e^{2\rho(3r)^{\theta}}\lambda\gamma.
\end{equation}
\end{lemma}
In order to prove the above measure estimate, we need the next two lemmas.
\begin{lemma}
\label{le33}
Fix $\rho>0$, $\theta\in(0,1)$, $r\geq3$, $N\geq r$ and $\gamma\in(0,1)$. Then there exists a positive constant $\lambda_{1}$ depending on $\varphi'(0)$ such that
\begin{equation}
\label{form311}
\mu(\forall\bs{j}\in Irr(\mc{R})\;\text{with}\;j_{1}^{*}\leq N\;\text{and}\;\#\bs{j}=2l\leq2r,\;|\Omega_{\bs{j}}^{(4)}(\varepsilon^{2}I)|\geq\gamma')>1-\lambda_{1}\gamma,
\end{equation}
where $\gamma'=\gamma\varepsilon^{2}(2N)^{-12l}e^{-2\rho|j^{*}_{2l}|^{\theta}}$.
\end{lemma}
\begin{proof}
Consider the complementary
$$\Theta^{(1)}=\{\exists \bs{j}\in Irr(\mc{R})\;\text{with}\;j_{1}^{*}\leq N\;\text{and}\;\#\bs{j}=2l\leq2r\;\text{such that}\;|\Omega_{\bs{j}}^{(4)}(\varepsilon^{2}I)|\leq\gamma'\}.$$
In view of \eqref{form38}, for any $\bs{j}=(\delta_i,a_i)_{i=1}^{2l}\in Irr(\mc{R})$ with $j_{1}^{*}\leq N$, one has
\begin{align}
\label{form312}
&\mu(|\Omega_{\bs{j}}^{(4)}(\varepsilon^{2}I)|\leq\gamma')
\\\notag=&\lim_{M\to\infty}\frac{\int_{\sum_{|a|\leq M}e^{2\rho|a|^{\theta}}I_{a}<1,|\Omega_{\bs{j}}^{(4)}(\varepsilon^{2}I)|\leq\gamma'}e^{-\sum_{|a|\leq M}e^{2\rho|a|^{\theta}}|a|^{2}I_{a}}\prod_{|a|\leq M}dI_{a}}{\int_{\sum_{|a|\leq M}e^{2\rho|a|^{\theta}}I_{a}<1} e^{-\sum_{|a|\leq M}e^{2\rho|a|^{\theta}}|a|^{2}I_{a}}\prod_{|a|\leq M}dI_{a}}
\\\notag=&\lim_{M\to\infty}\frac{\int_{I_0+\sum_{0\neq|a|\leq M}|a|^{-2}x_{a}<1,|\Omega_{\bs{j}}^{(4)}(\varepsilon^{2}I)|\leq\gamma'}e^{-\sum_{0\neq|a|\leq M}x_{a}}dI_0\prod_{0\neq|a|\leq M}dx_{a}}{\int_{I_0+\sum_{0\neq|a|\leq M}|a|^{-2}x_{a}<1} e^{-\sum_{0\neq|a|\leq M}x_{a}}dI_0\prod_{0\neq|a|\leq M}dx_{a}}
\end{align}
with $x_{a}=e^{2\rho|a|^{\theta}}|a|^{2}I_{a}$ for $a\neq0$. Let $A_1=(1+\sum_{a\in\mb{Z}\backslash\{0\}}|a|^{-\frac{3}{2}})^{-1}$ and then
\begin{equation}
\label{form313}
\{I_0+\sum_{0\neq|a|\leq M}|a|^{-2}x_{a}<1\}\supseteq\{I_0<A_1,\; \text{and}\;|a|^{-\frac{1}{2}}x_a<A_1, \;\text{for any}\;a\neq0\}.
\end{equation}
Hence, the denominator of \eqref{form312} has the lower bound
\begin{equation}
\label{form314}
A_1\prod_{0\neq|a|\leq M}\int_{0}^{A_1\sqrt{|a|}}e^{-x_a}dx_{a}.
\end{equation}
Let $j_{2l}^{*}=|\tilde{a}|$, and then by $|\Omega_{\bs{j}}^{(4)}(\varepsilon^{2}I)|<\gamma'=\gamma\varepsilon^{2}(2N)^{-12l}e^{-2\rho|j^{*}_{2l}|^{\theta}}$, one has
\begin{equation}
\label{form315}
A_2-\frac{\gamma|\tilde{a}|^{2}}{(2N)^{12l}|\varphi'(0)|}\leq x_{\tilde{a}}\leq A_2+\frac{\gamma|\tilde{a}|^{2}}{(2N)^{12l}|\varphi'(0)|}.
\end{equation}
with $A_2=|\sum_{a_i\neq\tilde{a}}\delta_{i}e^{2\rho(|\tilde{a}|^{\theta}-|a|^{\theta})}\big|\frac{\tilde{a}}{a_i}\big|^{2}x_{a_{i}}|$.
By \eqref{form312}, \eqref{form314} and \eqref{form315}, one has
\begin{align*}
&\mu\big(|\Omega_{\bs{j}}^{(4)}(\varepsilon^{2}I)|\leq\gamma'\big)
\\\leq&\lim_{M\to\infty}\frac{\prod_{a\neq\tilde{a},0\neq|a|\leq M}(\int_{0}^{+\infty} e^{-x_{a}}dx_{a})\int_{A_2-\frac{\gamma|\tilde{a}|^{2}}{(2N)^{12l}|\varphi'(0)|}}^{A_2+\frac{\gamma|\tilde{a}|^{2}}{(2N)^{12l}|\varphi'(0)|}}dx_{\tilde{a}}}{A_1\prod_{0\neq|a|\leq M}\int_{0}^{A_1\sqrt{|a|}}e^{-x_a}dx_{a}}
\\\leq&\lim_{M\to\infty}\frac{2\gamma(j_{2l}^{*})^{2}(2N)^{-12l}|\varphi'(0)|^{-1}}{A_1\prod_{0\neq|a|\leq M}(1-e^{-A_1\sqrt{|a|}})}
\\\leq&\frac{\gamma(2N)^{2-12l}}{2|\varphi'(0)|A_1\prod_{a\in\mb{Z}\backslash\{0\}}(1-e^{-A_1\sqrt{|a|}})}.
\end{align*}
Hence, by the fact $4N+2\leq(2N)^{2}$, one has
\begin{align*}
\mu(\Theta^{(1)})\leq&\sum_{l=3}^{r}\sum_{\substack{\bs{j}\in Irr(\mc{R})\\j_{1}^{*}\leq N,\#\bs{j}=2l}}\frac{\gamma(2N)^{2-12l}}{2|\varphi'(0)|A_1\prod_{a\in\mb{Z}\backslash\{0\}}(1-e^{-A_1\sqrt{|a|}})}
\\\leq&(r-2)4(4N+2)^{2l-2}\frac{\gamma(2N)^{2-12l}}{2|\varphi'(0)|A_1\prod_{a\in\mb{Z}\backslash\{0\}}(1-e^{-A_1\sqrt{|a|}})}
\\\leq&\lambda_{1}\gamma
\end{align*}
with the constant $\lambda_1$ depending on $\varphi'(0)$ and $A_1$. Then we can draw the conclusion.
\end{proof}
For convenience, consider $\tilde{\Omega}_{\bs{j}}^{(6)}(I):=\Omega_{\bs{j}}^{(4)}(I)+\tilde{\Omega}_{\bs{j}}^{(6,6)}(I)$ with
\begin{align}
\label{form316}
\tilde{\Omega}_{\bs{j}}^{(6,6)}(I)=-\frac{(\varphi'(0))^{2}}{2}\sum_{i=1}^{2l}\delta_{i}\sum_{\substack{a\in\mb{Z}\\a\neq a_1,\cdots,a_{2l}}}\frac{I_{a}^{2}}{(a-a_i)^{2}}.
\end{align}
\begin{lemma}
\label{le34}
Fix $\rho>0$, $\theta\in(0,1)$, $r\geq3$, $N\geq12r$ and $\gamma\in(0,1)$. Then there exists a positive constant $\lambda_{2}$ depending on $\varphi'(0)$ such that
\begin{equation}
\label{form317}
\mu(\forall\bs{j}\in Irr(\mc{R})\;\text{with}\;j_{1}^{*}\leq N\;\text{and}\;\#\bs{j}=2l\leq2r,\;|\tilde{\Omega}_{\bs{j}}^{(6)}(\varepsilon^{2}I)|\geq\gamma'')>1-e^{2\rho(3r)^{\theta}}\lambda_{2}\gamma,
\end{equation}
where $\gamma''=2\gamma\varepsilon^{2}(2N)^{-12l}\max\{e^{-2\rho|j^{*}_{2l}|^{\theta}},2\gamma\varepsilon^{2}\}$.
\end{lemma}
\begin{proof}
Consider the complementary
$$\Theta^{(2)}=\{\exists \bs{j}\in Irr(\mc{R})\;\text{with}\;j_{1}^{*}\leq N\;\text{and}\;\#\bs{j}=2l\leq2r\;\text{such that}\;|\tilde{\Omega}_{\bs{j}}^{(6)}(\varepsilon^{2}I)|\leq\gamma''\}.$$
Let $x_{a}=e^{2\rho|a|^{\theta}}|a|^{2}I_{a}$ for $a\neq0$, and then
\begin{align}
\label{form318}
&\mu\big(|\tilde{\Omega}_{\bs{j}}^{(6)}(\varepsilon^{2}I)|\leq\gamma''\big)
\\\notag=&\lim_{M\to\infty}\frac{\int_{\sum_{|a|\leq M}e^{2\rho|a|^{\theta}}I_{a}<1,|\tilde{\Omega}_{\bs{j}}^{(6)}(\varepsilon^{2}I)|\leq\gamma''}e^{-\sum_{|a|\leq M}e^{2\rho|a|^{\theta}}|a|^{2}I_{a}}\prod_{|a|\leq M}dI_{a}}{\int_{\sum_{|a|\leq M}e^{2\rho|a|^{\theta}}I_{a}<1} e^{-\sum_{|a|\leq M}e^{2\rho|a|^{\theta}}|a|^{2}I_{a}}\prod_{|a|\leq M}dI_{a}}
\\\notag=&\lim_{M\to\infty}\frac{\int_{I_0+\sum_{0\neq|a|\leq M}|a|^{-2}x_{a}<1,|\tilde{\Omega}_{\bs{j}}^{(6)}(\varepsilon^{2}I)|\leq\gamma''}e^{-\sum_{0\neq|a|\leq M}x_{a}}dI_0\prod_{0\neq|a|\leq M}dx_{a}}{\int_{I_0+\sum_{0\neq|a|\leq M}|a|^{-2}x_{a}<1} e^{-\sum_{0\neq|a|\leq M}x_{a}}dI_0\prod_{0\neq|a|\leq M}dx_{a}}.
\end{align}
Now, we are going to estimate \eqref{form318} in the following two aspects.
\\\indent Firstly, notice that for any $\bs{j}=(\delta_{i},a_{i})_{i=1}^{2l}\in Irr(\mc{R})$, $\tilde{\Omega}_{\bs{j}}^{(6)}(\varepsilon^{2}I)-\Omega_{\bs{j}}^{(4)}(\varepsilon^{2}I)$ is independent of $I_{a_{i}}$. Similarly with the proof of \Cref{le33}, one has
\begin{equation}
\label{form319}
\mu\big(|\tilde{\Omega}_{\bs{j}}^{(6)}(\varepsilon^{2}I)|\leq\gamma''\big)\leq\frac{2\gamma''e^{2\rho|j^{*}_{2l}|^{\theta}}|j_{2l}^{*}|^{2}}{\varepsilon^{2}|\varphi'(0)|A_1\prod_{a\in\mb{Z}\backslash\{0\}}(1-e^{-A_1\sqrt{|a|}})}
\end{equation}
with $A_1=(1+\sum_{a\in\mb{Z}\backslash\{0\}}|a|^{-\frac{3}{2}})^{-1}$.
\\\indent Next, by \Cref{le03} in Appendix, for any $\bs{j}=(\delta_{i},a_{i})_{i=1}^{2l}\in Irr(\mc{R})$ with $j_{1}^{*}\leq N$, there exists an integer $\tilde{a}\in(-3l,3l)\backslash\{a_{1},\cdots,a_{2l}\}$ such that
$$\Big|\sum_{i=1}^{2l}\frac{\delta_{i}}{(\tilde{a}-a_{i})^{2}}\Big|\geq\frac{1}{(6lN)^{4l}}\geq \frac{2^{4l}}{N^{8l}}.$$
Then rewrite
$\tilde{\Omega}_{\bs{j}}^{(6)}(\varepsilon^{2}I)=(d_{\bs{j}})_{\tilde{a}}\varepsilon^{4}I_{\tilde{a}}^{2}+Q_{\bs{j}}((\varepsilon^{2}I_{b})_{b\neq\tilde{a}})$, where
\begin{equation}
\label{form320}
|(d_{\bs{j}})_{\tilde{a}}|=\Big|-\frac{(\varphi'(0))^{2}}{2}\sum_{i=1}^{2l}\frac{\delta_{i}}{(\tilde{a}-a_{i})^{2}}\Big|\geq\frac{2^{4l}(\varphi'(0))^{2}}{2N^{8l}}.
\end{equation}
Notice that
$$\big||(d_{\bs{j}})_{\tilde{a}}|\varepsilon^{4}I_{\tilde{a}}^{2}-|Q_{\bs{j}}((\varepsilon^{2}I_{b})_{b\neq\tilde{a}})|\big|\leq|\tilde{\Omega}_{\bs{j}}^{(6)}(\varepsilon^{2}I)|\leq\gamma''.$$
When $|Q_{\bs{j}}((\varepsilon^{2}I_{b})_{b\neq\tilde{a}})|<3\gamma''$, one has
$$|x_{\tilde{a}}|\leq\frac{2e^{2\rho|\tilde{a}|^{\theta}}|\tilde{a}|^{2}}{\varepsilon^{2}}\sqrt{\frac{\gamma''}{|(d_{\bs{j}})_{\tilde{a}}|}}.$$
 By \eqref{form313} and \eqref{form318}, one has
\begin{align}
\label{form321}
\mu\big(|\tilde{\Omega}_{\bs{j}}^{(6)}(\varepsilon^{2}I)|\leq\gamma''\big)
&\leq\lim_{M\to\infty}\frac{\prod_{a\neq\tilde{a},0\neq|a|\leq M}(\int_{0}^{+\infty} e^{-x_{a}}dx_{a})\int_{-\frac{2e^{2\rho|\tilde{a}|^{\theta}}|\tilde{a}|^{2}}{\varepsilon^{2}}\sqrt{\frac{\gamma''}{|(d_{\bs{j}})_{\tilde{a}}|}}}^{\frac{2e^{2\rho|\tilde{a}|^{\theta}}|\tilde{a}|^{2}}{\varepsilon^{2}}\sqrt{\frac{\gamma''}{|(d_{\bs{j}})_{\tilde{a}}|}}}dx_{\tilde{a}}}{A_1\prod_{0\neq|a|\leq M}\int_{0}^{A_1\sqrt{|a|}}e^{-x_a}dx_{a}}
\\\notag&\leq\frac{4e^{2\rho|\tilde{a}|^{\theta}}|\tilde{a}|^{2}}{\varepsilon^{2}A_1\prod_{a\in\mb{Z}\backslash\{0\}}(1-e^{-A_1\sqrt{|a|}})}\sqrt{\frac{\gamma''}{|(d_{\bs{j}})_{\tilde{a}}|}}.
\end{align}
When $|Q_{\bs{j}}((\varepsilon^{2}I_{b})_{b\neq\tilde{a}})|\geq3\gamma''$, one has
$$\frac{e^{2\rho|\tilde{a}|^{\theta}}|\tilde{a}|^{2}}{\varepsilon^{2}}\sqrt{\frac{|Q_{\bs{j}}((\varepsilon^{2}I_{b})_{b\neq\tilde{a}})|-\gamma''}{|(d_{\bs{j}})_{\tilde{a}}|}}\leq x_{\tilde{a}}\leq\frac{e^{2\rho|\tilde{a}|^{\theta}}|\tilde{a}|^{2}}{\varepsilon^{2}}\sqrt{\frac{|Q_{\bs{j}}((\varepsilon^{2}I_{b})_{b\neq\tilde{a}})|+\gamma''}{|(d_{\bs{j}})_{\tilde{a}}|}}.$$
By \eqref{form313} and \eqref{form318}, one has
\begin{align}
\label{form322}
&\mu\big(|\tilde{\Omega}_{\bs{j}}^{(6)}(\varepsilon^{2}I)|\leq\gamma''\big)
\\\notag\leq&\lim_{M\to\infty}\frac{\prod_{a\neq\tilde{a},0\neq|a|\leq M}(\int_{0}^{+\infty} e^{-x_{a}}dx_{a})\int_{\frac{e^{2\rho|\tilde{a}|^{\theta}}|\tilde{a}|^{2}}{\varepsilon^{2}}\sqrt{\frac{|Q_{\bs{j}}((I_{b})_{b\neq\tilde{a}})|-\gamma''}{|(d_{\bs{j}})_{\tilde{a}}|}}}^{\frac{e^{2\rho|\tilde{a}|^{\theta}}|\tilde{a}|^{2}}{\varepsilon^{2}}\sqrt{\frac{|Q_{\bs{j}}((I_{b})_{b\neq\tilde{a}})|+\gamma''}{|(d_{\bs{j}})_{\tilde{a}}|}}}dx_{\tilde{a}}}{A_1\prod_{0\neq|a|\leq M}\int_{0}^{A_1\sqrt{|a|}}e^{-x_a}dx_{a}}
\\\notag\leq&\frac{e^{2\rho|\tilde{a}|^{\theta}}|\tilde{a}|^{2}\big(\sqrt{|Q_{\bs{j}}((I_{b})_{b\neq\tilde{a}})|+\gamma''}-\sqrt{|Q_{\bs{j}}((I_{b})_{b\neq\tilde{a}})|-\gamma''}\big)}{\varepsilon^{2}A_1\prod_{a\in\mb{Z}\backslash\{0\}}(1-e^{-A_1\sqrt{|a|}})\sqrt{|(d_{\bs{j}})_{\tilde{a}}|}}
\\\notag\leq&\frac{e^{2\rho|\tilde{a}|^{\theta}}|\tilde{a}|^{2}}{\varepsilon^{2}A_1\prod_{a\in\mb{Z}\backslash\{0\}}(1-e^{-A_1\sqrt{|a|}})\sqrt{|(d_{\bs{j}})_{\tilde{a}}|}}\frac{2\gamma''}{\sqrt{|Q_{\bs{j}}((I_{b})_{b\neq\tilde{a}})|-\gamma''}}
\\\notag\leq&\frac{e^{2\rho|\tilde{a}|^{\theta}}|\tilde{a}|^{2}}{\varepsilon^{2}A_1\prod_{a\in\mb{Z}\backslash\{0\}}(1-e^{-A_1\sqrt{|a|}})}\sqrt{\frac{2\gamma''}{|(d_{\bs{j}})_{\tilde{a}}|}}.
\end{align}
By \eqref{form320}--\eqref{form322} and $N\geq6r$, one has
\begin{align}
\label{form323}
\mu\big(|\tilde{\Omega}_{\bs{j}}^{(6)}(\varepsilon^{2}I)|\leq\gamma'\big)\leq&\frac{4e^{2\rho|\tilde{a}|^{\theta}}|\tilde{a}|^{2}}{\varepsilon^{2}A_1\prod_{a\in\mb{Z}\backslash\{0\}}(1-e^{-A_1\sqrt{|a|}})}\sqrt{\frac{\gamma''}{|(d_{\bs{j}})_{\tilde{a}}|}}
\\\notag\leq&\frac{4e^{2\rho(3l)^{\theta}}(3l)^{2}}{\varepsilon^{2}A_1\prod_{a\in\mb{Z}\backslash\{0\}}(1-e^{-A_1\sqrt{|a|}})}\frac{\sqrt{2\gamma''}N^{4l}}{2^{2l}|\varphi'(0)|}
\\\notag\leq&\frac{\sqrt{2\gamma''}e^{2\rho(3l)^{\theta}}N^{4l+2}}{2^{2l+2}\varepsilon^{2}|\varphi'(0)|A_1\prod_{a\in\mb{Z}\backslash\{0\}}(1-e^{-A_1\sqrt{|a|}})}.
\end{align}
\indent To sum up, by estimates \eqref{form319} and \eqref{form323}, for any $\bs{j}\in Irr(\mc{R})$ with $j_{1}^{*}\leq N$ and $\#\bs{j}=2l$, one has
\begin{align}
\label{form324}
\mu\big(|\tilde{\Omega}_{\bs{j}}^{(6)}(\varepsilon^{2}I)|\leq\gamma''\big)\leq&\frac{2\min\{\gamma''e^{2\rho|j^{*}_{2l}|^{\theta}}|j_{2l}^{*}|^{2},\sqrt{\gamma''}e^{2\rho(3l)^{\theta}}2^{-2l-2}N^{4l+2}\}}{\varepsilon^{2}|\varphi'(0)|A_1\prod_{a\in\mb{Z}\backslash\{0\}}(1-e^{-A_1\sqrt{|a|}})}
\\\notag\leq&\frac{\gamma e^{2\rho(3l)^{\theta}}}{|\varphi'(0)|A_1\prod_{a\in\mb{Z}\backslash\{0\}}(1-e^{-A_1\sqrt{|a|}})2^{8l}N^{2l-2}}.
\end{align}
Hence, by the fact $4N+2\leq2^{3}N$, one has
\begin{align}
\label{form325}
\mu(\Theta^{(2)})
&\leq\sum_{l=3}^{r}\sum_{\substack{\bs{j}\in Irr(\mc{R})\\j_{1}^{*}\leq N,\#\bs{j}=2l}}\frac{\gamma e^{2\rho(3l)^{\theta}}}{|\varphi'(0)|A_1\prod_{a\in\mb{Z}\backslash\{0\}}(1-e^{-A_1\sqrt{|a|}})2^{8l}N^{2l-2}}
\\\notag&\leq\sum_{l=3}^{r}4(4N+2)^{2l-2}\frac{\gamma e^{2\rho(3l)^{\theta}}}{|\varphi'(0)|A_1\prod_{a\in\mb{Z}\backslash\{0\}}(1-e^{-A_1\sqrt{|a|}})2^{8l}N^{2l-2}}
\\\notag&\leq e^{2\rho(3r)^{\theta}}\lambda_2\gamma
\end{align}
with the constant $\lambda_2$ depending on $\varphi'(0)$ and $A_1$.
Hence, we can draw the conclusion.
\end{proof}
\begin{proof}[Proof of Lemma {\sl\ref{le32}}]
Consider the event $\Theta$ defined by
\begin{align*}
\{\forall\bs{j}\in\mc{R}_{2l}\backslash\mc{I}_{2l}\;\text{with}\;3\leq l\leq r\;\text{and}\;j_{1}^{*}\leq N,\;\;|\Omega_{\bs{j}}^{(4)}(\varepsilon^{2}I)|>\tilde{\gamma}'\;\text{and}\;|\tilde{\Omega}_{\bs{j}}^{(6)}(\varepsilon^{2}I)|>\tilde{\gamma}''\}
\end{align*}
with
$$\tilde{\gamma}'=\gamma\varepsilon^{2}\|z\|_{\rho,\theta}^{2}(2N)^{-12l}e^{-2\rho|j^{*}_{2l}|^{\theta}},\;\tilde{\gamma}''=2\gamma\varepsilon^{2}\|z\|_{\rho,\theta}^{2}(2N)^{-12l}\max\{e^{-2\rho|j^{*}_{2l}|^{\theta}},2\gamma\varepsilon^{2}\|z\|_{\rho,\theta}^{2}\}.$$
By \Cref{le33}, \Cref{le34} and the fact $\gamma'\geq\tilde{\gamma}'$, $\gamma''\geq\tilde{\gamma}''$ for $z\in B_{\rho,\theta}(1)$, there exists $\lambda$ depending on $\varphi'(0)$ and $A_1$ such that
\begin{equation}
\label{form326}
\mu(\Theta)\geq1-e^{2\rho(3r)^{\theta}}\lambda\gamma.
\end{equation}
\indent Now, we will show that when the event $\Theta$ holds, one has $\varepsilon z\in\mc{U}_{\gamma}^{N}$, that is to say, we need prove that when $|\tilde{\Omega}_{\bs{j}}^{(6)}(\varepsilon^{2}I)|>\tilde{\gamma}''$, one has $|\Omega_{\bs{j}}^{(6)}(\varepsilon^{2}I)|>\tilde{\gamma}$. In view of \eqref{form32} and \eqref{form316}, there exists a positive constant $C$ depending on $\varphi'(0)$ and $\varphi''(0)$ such that
\begin{align*}
|\Omega_{\bs{j}}^{(6)}(\varepsilon^{2}I)-\tilde{\Omega}_{\bs{j}}^{(6)}(\varepsilon^{2}I)|\leq&Cl\varepsilon^{4}\|z\|_{\rho,\theta}^{4}e^{-2\rho|j_{2l}^{*}|^{\theta}}
\\\leq&\gamma \varepsilon^{2}\|z\|_{\rho,\theta}^{2}(2N)^{-12l}e^{-2\rho|j_{2l}^{*}|^{\theta}},
\end{align*}
where the last inequality follows from the condition \eqref{form39}.
Then
\begin{align*}
|\Omega_{\bs{j}}^{(6)}(\varepsilon^{2}I)|\geq&|\tilde{\Omega}_{\bs{j}}^{(6)}(\varepsilon^{2}I)|-|\Omega_{\bs{j}}^{(6)}(\varepsilon^{2}I)-\tilde{\Omega}_{\bs{j}}^{(6)}(\varepsilon^{2}I)|
\\>&\gamma\varepsilon^{2}\|z\|_{\rho,\theta}^{2}(2N)^{-12l}\max\{e^{-2\rho|j_{2l}^{*}|^{\theta}},\gamma\varepsilon^{2}\|z\|_{\rho,\theta}^{2}\}.
\end{align*}
Hence, we can draw the conclusion.
\end{proof}

\section{Proof of \Cref{th11}}
\label{sec7}
Applying \Cref{th21}, \Cref{le31}, \Cref{th51} and \Cref{th52}, we are going to prove \Cref{th11} in this section.
\\\indent Consider Hamiltonian function \eqref{form212} and fix $\rho>0$, $\theta\in(0,1)$, $\beta\in(0,1)$. Let $r=[|\ln\varepsilon|^{\beta}]$, $N=|\ln\varepsilon|^{\frac{5}{\theta}}$, $\gamma=\varepsilon^{\frac{1}{7}}$ and the open set
\begin{equation}
\label{form61}
\mc{V}_{\rho,\theta,\beta}:=\bigcup_{0<\varepsilon\leq\varepsilon_{0}}\Big(\mc{U}_{\gamma}^{N}\bigcap\big(B_{\rho,\theta}(\varepsilon)\backslash \overline{B_{\rho,\theta}(\frac{\varepsilon}{2})}\big)\Big),
\end{equation}
where $\overline{B_{\rho,\theta}(\frac{\varepsilon}{2})}$ is the closure of $B_{\rho,\theta}(\frac{\varepsilon}{2})$.
\\\indent Notice that $\varepsilon^{2}\lesssim\frac{\gamma}{r(2N)^{12r}}$ and then by \Cref{le32}, there exists a positive constant $\lambda$ such that
\begin{equation}
\label{form62}
\mu(\varepsilon z(0)\in\mc{U}_{\gamma}^{N})\geq1-e^{2\rho(3r)^{\theta}}\lambda\gamma>1-\varepsilon^{\frac{1}{8}}.
\end{equation}
By the construction of $\mc{V}_{\rho,\theta,\beta}$, one has
\begin{equation}
\label{form63}
\mu(\varepsilon z(0)\in\mc{V}_{\rho,\theta,\beta})\geq1-\varepsilon^{\frac{1}{8}}.
\end{equation}
\indent Notice that the conditions $N\gtrsim r$, $\varepsilon\lesssim_{\rho,\theta}r^{-\frac{3}{2}}$ and $\varepsilon\lesssim\gamma^{\frac{7}{2}}(C_4N)^{-\frac{341}{2}r}$ hold. Hence, by \Cref{th21}, \Cref{le31}, \Cref{th51} and \Cref{th52}, there exists a canonical transformation $\phi=\phi^{(1)}\circ\Phi_{\chi'_{6}}^{1}\circ\phi^{(2)}$ such that  $z'=\phi^{-1}(z)$ with $z'\in\mc{U}_{\frac{\gamma}{2}}^{N}$ satisfies the following estimates
\begin{align}
\label{form64}
&\sup_{\|z\|_{\rho,\theta}\leq 3\varepsilon/2}\|z-\phi^{-1}(z)\|_{\rho,\theta}\lesssim\frac{32^{r}(C_5N)^{377}}{\gamma^{5}}\|z\|_{\rho,\theta}^{3},
\\\label{form65}
&\sup_{\|z'\|_{\rho,\theta}\leq 3\varepsilon/2}\|z'-\phi(z')\|_{\rho,\theta}\lesssim\frac{32^{r}(C_5N)^{377}}{\gamma^{5}}\|z'\|_{\rho,\theta}^{3},
\end{align}
 and then
\begin{equation}
\label{form66}
\|z'(0)\|_{\rho,\theta}\leq\|z(0)\|_{\rho,\theta}+\|z(0)-\phi^{-1}(z(0))\|_{\rho,\theta}\leq\frac{5}{4}\varepsilon.
\end{equation}
Define $F(z)=\|z\|_{\rho,\theta}^{2}$ and denote by $T$ the escape time of $z'$ from $B_{\rho,\theta}(\frac{3}{2}\varepsilon)$. For any $|t|\leq T$, one has
\begin{align}
\label{form67}
&\;|F(z'(t))-F(z'(0))|=\left|\int_{0}^{t}\{H\circ\phi,F\}(z'(t))dt\right|
\\\notag=&\left|\int_{0}^{t}\{R''_{\geq2r+2}+\big(\sum_{l=3}^{r}R_{2l}+\ms{R}+R_{\geq2r+2}\circ\phi^{(1)}\big)\circ\Phi_{\chi'_{6}}^{1}\circ\phi^{(2)},F\}(z'(t))dt\right|
\\\notag\leq&T\sup_{|t|\leq T}\Big(\|X_{R''_{\geq2r+4}}(z'(t))\|_{\rho,\theta}+\sum_{l=3}^{r}\|X_{R_{2l}\circ\Phi_{\chi'_{6}}^{1}\circ\phi^{(2)}}(z'(t))\|_{\rho,\theta}
\\\notag&\qquad\qquad+\|X_{\ms{R}\circ\Phi_{\chi'_{6}}^{1}\circ\phi^{(2)}}(z'(t))\|_{\rho,\theta}+\|X_{R_{\geq2r+2}\circ\phi}(z'(t))\|_{\rho,\theta}\Big)\|z'(t)\|_{\rho,\theta}.
\end{align}
\indent Next, we will estimate these vector fields in \eqref{form67}.
Notice that for any $\bs{j}\in\mc{R}_{2l}$ with $(Irr(\bs{j}))_{1}^{*}>N$ and $l\leq r$, one has
\begin{equation}
\label{form68}
j_{3}^{*}>\sqrt{\frac{N}{2l}}\geq\sqrt{\frac{N}{2r}}.
\end{equation}
In view of \eqref{form52}, by \Cref{le02} in Appendix and \eqref{form68}, there exits a positive constant $C$ depending on $\rho,\theta$ such that
\begin{align}
\label{form69}
\|X_{R_{2l}}(z)\|_{\rho,\theta}\leq&C^{2l-1}(8C_{0}(l-1)^{3})^{l-1}\|z\|_{\rho,\theta}\|z\|_{2^{\theta-1}\rho,\theta}^{2l-2}
\\\notag\leq&C^{2l-1}(8C_{0}(l-1)^{3})^{l-1}e^{-(1-2^{\theta-1})\rho (\frac{N}{2r})^{\frac{\theta}{2}}}\|z\|_{\rho,\theta}^{2l-1}.
\end{align}
Then there exists a positive constant $C_{9}$ depending on $C_0,C$ such that
\begin{align}
\label{form610}
&\|X_{R_{2l}\circ\Phi_{\chi'_{6}}^{1}\circ\phi^{(2)}}(z'(t))\|_{\rho,\theta}
\\\notag\lesssim&\|X_{R_{2l}}(\Phi_{\chi'_{6}}^{1}\circ\phi^{(2)}(z'(t)))\|_{\rho,\theta}
\\\notag\leq&C^{2l-1}(8C_{0}(l-1)^{3})^{l-1}e^{-(1-2^{\theta-1})\rho (\frac{N}{2r})^{\frac{\theta}{2}}}\|\Phi_{\chi'_{6}}^{1}\circ\phi^{(2)}(z'(t))\|_{\rho,\theta}^{2l-1}
\\\notag\leq&(C_{9}(l-1)^{3})^{l-1}e^{-(1-2^{\theta-1})\rho (\frac{N}{2r})^{\frac{\theta}{2}}}\|z\|_{\rho,\theta}^{2l-1}.
\end{align}
Hence, one has
\begin{equation}
\label{form611}
\sum_{l=3}^{r}\|X_{R_{2l}\circ\Phi_{\chi'_{6}}^{1}\circ\phi^{(2)}}(z'(t))\|_{\rho,\theta}\lesssim e^{-(1-2^{\theta-1})\rho (\frac{N}{2r})^{\frac{\theta}{2}}}\|z'(t)\|_{\rho,\theta}^{5}.
\end{equation}
By \eqref{form220} in \Cref{th21}, one has
\begin{align}
\label{form612}
\|X_{\ms{R}\circ\Phi_{\chi'_{6}}^{1}\circ\phi^{(2)}}(z'(t))\|_{\rho,\theta}\lesssim&\|X_{\ms{R}}(\Phi_{\chi'_{6}}^{1}\circ\phi^{(2)}(z'(t))\|_{\rho,\theta}
\\\notag\leq& (r^{3}C_1)^{r}\|\Phi_{\chi'_{6}}^{1}\circ\phi^{(2)}(z'(t))\|_{\rho,\theta}^{2r+1}
\\\notag\lesssim&(4r^{3}C_1)^{r}\|z'(t)\|_{\rho,\theta}^{2r+1}.
\end{align}
Similarly, by \eqref{form214}, one has
\begin{equation}
\label{form613}
\|X_{R_{\geq2r+2}\circ\phi}(z'(t))\|_{\rho,\theta}
\lesssim C_0^{r}\|\phi(z'(t))\|_{\rho,\theta}^{2r+1}
\lesssim(4C_0)^{r}\|z'(t)\|_{\rho,\theta}^{2r+1}.
\end{equation}
By \Cref{th52}, one has
\begin{equation}
\label{form614}
\sup_{\|z'\|_{\rho,\theta}\leq3\varepsilon/2}\|X_{R'_{\geq2r+4}}(z'(t))\|_{\rho,\theta}\lesssim\frac{(C_7N)^{341r^{2}}}{\gamma^{6r}}\|z'(t)\|_{\rho,\theta}^{2r+1}.
\end{equation}
Then by \eqref{form67} and estimates of  the vector field \eqref{form611}--\eqref{form614}, one has
\begin{align}
\label{20230510-1}
|F(z'(t))-F(z'(0))|\lesssim&T\Big(e^{-(1-2^{\theta-1})\rho(\frac{N}{2r})^{\frac{\theta}{2}}}(\frac{3}{2}\varepsilon)^{6}+(4r^{3}C_1)^{r}(\frac{3}{2}\varepsilon)^{2r+2}
\\\notag&\qquad+(4C_0)^{r}(\frac{3}{2}\varepsilon)^{2r+2}+\frac{(C_7N)^{341r^{2}}}{\gamma^{6r}}(\frac{3}{2}\varepsilon)^{2r+2}\Big).
\end{align}

\indent For any $0<\varepsilon\leq\varepsilon_0$ with the enough small $\varepsilon_0$, one has
\begin{equation}
\label{20230510-2}
e^{-(1-2^{\theta-1})\rho(\frac{N}{2r})^{\frac{\theta}{2}}}\leq\varepsilon^{|\ln\varepsilon|^{\beta}},
\end{equation}
\begin{equation}
\label{20230510-3}
\Big((4r^{3}C_1)^{r}+(4C_0)^{r}+\frac{(C_7N)^{341r^{2}}}{\gamma^{6r}}\Big)(\frac{3}{2}\varepsilon)^{2r+2}\leq\varepsilon^{3+|\ln\varepsilon|^{\beta}}.
\end{equation}
Hence, by \eqref{20230510-1}-\eqref{20230510-3}, one has
\begin{equation}
\label{form615}
|F(z'(t))-F(z'(0))|\lesssim T\varepsilon^{3+|\ln\varepsilon|^{\beta}}.
\end{equation}
If the stability time $T\leq\varepsilon^{-|\ln\varepsilon|^{\beta}}$, then by \eqref{form67} and \eqref{form615}, one has
\begin{align*}
(\frac{3}{2}\varepsilon)^{2}=\|z'(\tilde{t})\|_{\rho,\theta}^{2}=F(z'(\tilde{t}))\leq F(z'(0))+|F(z'(\tilde{t}))-F(z'(0))|
<(\frac{5}{4}\varepsilon)^{2}+\frac{1}{2}\varepsilon^{2},
\end{align*}
which is impossible. Hence, $T\geq\varepsilon^{-|\ln\varepsilon|^{\beta}}$.
By \eqref{form65}, for any $t\leq\varepsilon^{-|\ln\varepsilon|^{\beta}}$, one has
$$\|z(t)\|_{\rho,\theta}\leq\|z'(t)\|_{\rho,\theta}+\|z'(t)-\phi(z'(t))\|_{\rho,\theta}\leq2\varepsilon.$$
\indent Finally, we are going to prove the estimate \eqref{form17}. Notice that for any $a\in\mb{Z}$, one has
\begin{equation}
\label{form616}
|I_{a}(t)-I_{a}(0)|\leq|I_{a}(t)-I'_{a}(t)|+|I'_{a}(t)-I'_{a}(0)|+|I'_{a}(0)-I_{a}(0)|.
\end{equation}
Similarly to \eqref{form67}--\eqref{form615} but $e^{2\rho|a|^{\theta}}I'_{a}$ instead of $F(z')$, for any $|t|\leq\varepsilon^{-|\ln\varepsilon|^{\beta}}$, one has
\begin{equation}
\label{form617}
e^{2\rho|a|^{\theta}}|I'_{a}(t)-I'_{a}(0)|\lesssim\varepsilon^{3}.
\end{equation}
By \eqref{form64}, one has
\begin{align}
\label{form618}
e^{2\rho|a|^{\theta}}|I_{a}(t)-I'_{a}(t)|\leq\big(\|z(t)\|_{\rho,\theta}+\|z'(t)\|_{\rho,\theta}\big)\|z(t)-z'(t)\|_{\rho,\theta}
\leq\varepsilon^{3}.
\end{align}
Hence, \eqref{form17} follows from \eqref{form616}--\eqref{form618}.

\appendix
\section{}
\begin{lemma}
\label{le01}
For any $\theta\in(0,1)$ and $|j_1|\geq|j_2|\geq\cdots\geq|j_{l}|\geq0$, one has
\begin{equation}
\label{form01}
(\sum_{i=1}^{2l}|j_{i}|)^{\theta}\leq|j_{1}|^{\theta}+(2^{\theta}-1)\sum_{i=2}^{2l}|j_{i}|^{\theta}
\end{equation}
\end{lemma}
\begin{proof}
It is (2.9) in \cite{CLSY18}.
\end{proof}
\begin{lemma}
\label{le02}
For any $\rho>0$, $\theta\in(0,1)$ and homogeneous polynomial $P_{2l}[c](z)=\sum_{\bs{j}\in\mc{M}_{2l}}c_{\bs{j}}z_{\bs{j}}$ with $\|c\|_{\ell^{\infty}}<+\infty$, there exists a positive constant $C$ depending on $\rho,\theta$ such that
\begin{equation}
\label{form02}
\|X_{P_{2l}}(z)\|_{\rho,\theta}\leq C^{2l-1}\|c\|_{\ell^{\infty}}\|z\|_{\rho,\theta}\|z\|_{2^{\theta-1}\rho,\theta}^{2l-2}.
\end{equation}
\end{lemma}
\begin{proof}
By the definition of norm, one has
\begin{equation}
\label{form03}
\|X_{P_{2l}}(z)\|_{\rho,\theta}\leq\Big(\sum_{j_{1}}e^{2\rho|j_1|^{\theta}}\big|2l\sum_{\substack{j_{2},\cdots,j_{2l}\\\bs{j}\in\mc{M}_{2l}}}c_{\bs{j}}z_{j_{2}}\cdots z_{j_{2l}}\big|^{2}\Big)^{\frac{1}{2}}\leq2l\|c\|_{\ell^{\infty}}\|\underbrace{z\ast\cdots\ast z}_{2l-1}\|_{\rho,\theta}.
\end{equation}
\indent For any $u=\{u_{j}\}_{j\in\mb{U}_{2}\times\mb{Z}}$ and $v=\{v_{j}\}_{j\in\mb{U}_{2}\times\mb{Z}}$, by \Cref{le01} in Appendix, one has
\begin{align*}
&\|u\ast v\|_{\rho,\theta}=\Big(\sum_{j}e^{2\rho|j|^{\theta}}\big|\sum_{\mc{M}(j_{1},j_{2},j)=0}u_{j_{1}}v_{j_{2}}\big|^{2}\Big)^{\frac{1}{2}}
\\\leq&\Big(\sum_{j}\big(\sum_{\substack{\mc{M}(j_{1},j_{2},j)=0}}e^{\rho(|j_1|+|j_2|)^{\theta}}|u_{j_{1}}||v_{j_{2}}|\big)^{2}\Big)^{\frac{1}{2}}
\\\leq&\Big(\sum_{j}\big(\sum_{\substack{|j_1|\geq|j_2|\\\mc{M}(j_{1},j_{2},j)=0}}e^{\rho|j_1|^{\theta}}e^{\rho(2^{\theta}-1)|j_2|^{\theta}}|u_{j_{1}}||v_{j_{2}}|+\sum_{\substack{|j_1|<|j_2|\\\mc{M}(j_{1},j_{2},j)=0}}e^{\rho(2^{\theta}-1)|j_1|^{\theta}}e^{\rho|j_2|^{\theta}}|u_{j_{1}}||v_{j_{2}}|\big)^{2}\Big)^{\frac{1}{2}}
\\\leq&\Big(\sum_{j}\big(\sum_{\substack{\mc{M}(j_{1},j_{2},j)=0}}e^{\rho|j_1|^{\theta}}e^{\rho(2^{\theta}-1)|j_2|^{\theta}}|u_{j_{1}}||v_{j_{2}}|\big)^{2}\Big)^{\frac{1}{2}}
\\&+\Big(\sum_{j}\big(\sum_{\substack{\mc{M}(j_{1},j_{2},j)=0}}e^{\rho(2^{\theta}-1)|j_1|^{\theta}}e^{\rho|j_2|^{\theta}}|u_{j_{1}}||v_{j_{2}}|\big)^{2}\Big)^{\frac{1}{2}}.
\end{align*}
By the H\"{o}lder inequality and the fact $2^{\theta}-2<0$, there exists a constant $\tilde{C}>1$ depending on $\rho,\theta$ such that
\begin{align*}
 &\Big(\sum_{j}\big(\sum_{\mc{M}(j_{1},j_{2},j)=0}e^{\rho|j_1|^{\theta}}e^{(2^{\theta}-1)\rho|j_2|^{\theta}}|u_{j_{1}}||v_{j_{2}}|\big)^{2}\Big)^{\frac{1}{2}}
 \\\leq&\Big(\sum_{j}\big(\sum_{\mc{M}(j_{1},j_{2},j)=0}e^{(2^{\theta}-2)\rho|j_2|^{\theta}}\big)\big(\sum_{\mc{M}(j_{1},j_{2},j)=0}e^{2\rho|j_1|^{\theta}}e^{2^{\theta}\rho|j_2|^{\theta}}|u_{j_{1}}|^{2}|v_{j_{2}}|^{2}\big)\Big)^{\frac{1}{2}}
 \\\leq&\tilde{C}\Big(\sum_{j}\sum_{\mc{M}(j_{1},j_{2},j)=0}e^{2\rho|j_1|^{\theta}}e^{2^{\theta}\rho|j_2|^{\theta}}|u_{j_{1}}|^{2}|v_{j_{2}}|^{2}\Big)^{\frac{1}{2}}
 \\\leq&\sqrt{2}\tilde{C}\|u\|_{\rho,\theta}\|v\|_{2^{\theta-1}\rho,\theta}
 \end{align*}
Similarly, one has
\begin{equation*}
\Big(\sum_{j}\big(\sum_{\mc{M}(j_{1},j_{2},j)=0}e^{\rho(2^{\theta}-1)|j_1|^{\theta}}e^{\rho|j_2|^{\theta}}|u_{j_{1}}||v_{j_{2}}|\big)^{2}\Big)^{\frac{1}{2}}\leq \sqrt{2}\tilde{C}\|u\|_{2^{\theta-1}\rho,\theta}\|v\|_{\rho,\theta}.
\end{equation*}
To sum up, one has
\begin{equation}
\label{form04}
\|u\ast v\|_{\rho,\theta}\leq \sqrt{2}\tilde{C}(\|u\|_{\rho,\theta}\|v\|_{2^{\theta-1}\rho,\theta}+\|u\|_{2^{\theta-1}\rho,\theta}\|v\|_{\rho,\theta}).
\end{equation}
\indent By induction, one has
\begin{equation}
\label{form05}
\|\underbrace{z\ast\cdots\ast z}_{2l-1}\|_{\rho,\theta}\leq (2\sqrt{2}\tilde{C})^{2l-2}\|z\|_{\rho,\theta}\|z\|_{2^{\theta-1}\rho,\theta}^{2l-2}.
\end{equation}
Take $C=(2\sqrt{2}\tilde{C})^{2}$ and by \eqref{form03}, \eqref{form05}, the inequality $2l\leq C^{l}$, one has
\begin{align*}
\|X_{P_{2l}}(z)\|_{\rho,\theta}\leq&2l(2\sqrt{2}\tilde{C})^{2l-2}\|c\|_{\ell^{\infty}}\|z\|_{\rho,\theta}\|z\|_{2^{\theta-1}\rho,\theta}^{2l-2}
\\\leq&C^{2l-1}\|c\|_{\ell^{\infty}}\|z\|_{\rho,\theta}\|z\|_{2^{\theta-1}\rho,\theta}^{2l-2}.
\end{align*}
\end{proof}
\begin{lemma}
\label{le03}
For any $\bs{j}=(\delta_{i},a_{i})_{i=1}^{2l}\in\mc{R}_{2l}\backslash\mc{I}_{2l}$ with $j_{1}^{*}\leq N$, there exists an integer $\tilde{a}\in(-3l,3l)\backslash\{a_{1},\cdots,a_{2l}\}$ such that
\begin{equation}
\label{form06}
\Big|\sum_{i=1}^{2l}\frac{\delta_{i}}{(\tilde{a}-a_{i})^{2}}\Big|\geq\frac{1}{(6lN)^{4l}}.
\end{equation}
\end{lemma}
\begin{proof}
It is Lemma 5.11 in \cite{BFG20b}.
\end{proof}



\end{document}